\documentclass[a4paper,twoside,12pt]{article} 

\usepackage{a4wide}
\usepackage{graphicx}
\usepackage{amsmath,amssymb,amsthm,amsgen}
\usepackage{listings}
\usepackage{color}
\usepackage{hhline}
\usepackage{multirow}
\usepackage{enumitem}  
\usepackage{mathtools}
\usepackage[bookmarks]{}
\usepackage{booktabs}
\usepackage{amsfonts}   
\usepackage{mathrsfs}
\usepackage{bbm}
\usepackage{bm}  
\usepackage{tikz}
\usetikzlibrary{patterns, decorations.pathreplacing}
\usepackage{url}
\usepackage{stmaryrd}
\usepackage{longtable} 
\usepackage{framed}
\usepackage{thmtools,thm-restate} 
\usepackage{upgreek}

 \usepackage[hidelinks]{hyperref}
\newtheorem{thm}{Theorem}[section]

\newtheorem{lem}[thm]{Lemma}

\newtheorem{cor}[thm]{Corollary}

\newtheorem{conv}[thm]{Convention}

\theoremstyle{definition}
\newtheorem{rem}[thm]{Remark}
\newtheorem{ass}{Assumption}

\newcommand{\bell}{\boldsymbol \ell} 
\newcommand{\bvarphi}{\boldsymbol \varphi}
\newcommand{\bin}[2]{ \left( \hspace{-1mm}
                        \begin{array}{c}
                          #1 \\
                          #2 \\
                        \end{array}
                      \hspace{-1mm} \right) }

\newcommand{\bmu}{\boldsymbol \mu}
\newcommand{\btau}{\boldsymbol \tau}
\newcommand{\bnu}{\boldsymbol \nu}

\newcommand{\bgamma}{\boldsymbol \gamma}
\newcommand{\brho}{\boldsymbol \rho}

\newcommand{\bT}{\mathbf T}
\newcommand{\bu}{\mathbf u}
\newcommand{\bv}{\mathbf v}

\newcommand{\dd}{d}  

\renewcommand{\div}{\operatorname{div}}

\newcommand{\E}{\mathop{{}\mathbb E}}

\newcommand{\law}{\operatorname{law}}
\newcommand{\N}{\mathbb N}
\newcommand{\R}{\mathbb R}

\newcommand{\T}{\mathbb T}

\newcommand{\Z}{\mathbb Z}
\def\softd{{\leavevmode\setbox1=\hbox{d}%
   \hbox to 1.05\wd1{d\kern-0.4ex{\char039}\hss}}}

\def\XXint#1#2#3{{\setbox0=\hbox{$#1{#2#3}{\int}$}
     \vcenter{\hbox{$#2#3$}}\kern-.5\wd0}}

\DeclareFontFamily{U}{mathx}{\hyphenchar\font45}
\DeclareFontShape{U}{mathx}{m}{n}{
      <5> <6> <7> <8> <9> <10>
      <10.95> <12> <14.4> <17.28> <20.74> <24.88>
      mathx10
      }{}
\DeclareSymbolFont{mathx}{U}{mathx}{m}{n}
\DeclareFontSubstitution{U}{mathx}{m}{n}
\DeclareMathAccent{\widecheck}{0}{mathx}{"71}

\newcommand\e{\varepsilon}
\def\bx{\boldsymbol{x}}
\def\by{\boldsymbol{y}}
\def\bX{\bm{X}}
\def\bl{\boldsymbol{\ell}}
\def\bb{\boldsymbol{b}}
\def\bh{{\boldsymbol{h}}}
\def\beff{\bm{f}}
\def\bF{\bm{F}}
\def\bB{\boldsymbol{B}}
\def\cC{\mathcal{C}}
\def\cM{\mathcal{M}}

\def\cP{\mathcal{P}}
\def\cR{\mathcal{R}}
\def\cV{\mathcal{V}}
\def\L{\Lambda_\e}
\def\Lpm{\Lambda_{\e,\pm}}
\def\Nhd{\mathcal{N}}
\newcommand\ds{\displaystyle}
\newcommand\filtF{\mathscr F}

\newcommand\filtM{\mathscr M}
\newcommand\Prob{\mathbb P}

\newcommand\obX{\overline{\bX}}
\newcommand\obx{\overline{\bx}}
\newcommand\omu{\overline \mu}
\newcommand\oX{\overline X}
\newcommand\ox{\overline x}
\newcommand\otau{\overline \tau}
\newcommand\obtau{\overline \btau}
\newcommand\orho{\overline\rho}

\long\def\drop#1{}
\def\RW{(RW_n^\e)}
\def\SDE{(SDE_n)}
\def\FP{(FP_n)}
\def\FPe{(FP_n^\e)}
\def\oSDE{(\overline{SDE})}

\def\MFe{(MF^\e)}
\def\MF{(M F)}

\usepackage{ifthen}
\newlength{\leftstackrelawd}
\newlength{\leftstackrelbwd}
\def\leftstackrel#1#2{\settowidth{\leftstackrelawd}%
{${{}^{#1}}$}\settowidth{\leftstackrelbwd}{$#2$}%
\addtolength{\leftstackrelawd}{-\leftstackrelbwd}%
\leavevmode\ifthenelse{\lengthtest{\leftstackrelawd>0pt}}%
{\kern-.5\leftstackrelawd}{}\mathrel{\mathop{#2}\limits^{#1}}}

\begin{document}

\title{Atomistic origins of continuum dislocation dynamics} 

\author{Thomas Hudson \and Patrick van Meurs \and Mark Peletier}


\maketitle

\begin{abstract}
  This paper focuses on the connections between four stochastic and deterministic models for the motion of straight screw dislocations. Starting from a description of screw dislocation motion as interacting random walks on a lattice, we prove explicit estimates of the distance between solutions of this model, an SDE system for the dislocation positions, and two deterministic mean-field models describing the dislocation density. The proof of these estimates uses a collection of various techniques in analysis and probability theory, including a novel approach to establish propagation-of-chaos on a spatially discrete model. The estimates are non-asymptotic and explicit in terms of four parameters: the lattice spacing, the number of dislocations, the dislocation core size, and the temperature. This work is a first step in exploring this parameter space with the ultimate aim to connect and quantify the relationships between the many different dislocation models present in the literature.
\end{abstract}



\section{Introduction}

Plastic deformation of crystals such as metals is a complex phenomenon. It depends crucially on features at widely differing scales, ranging from the thermal motion of individual atoms through the self-organisation of lattice defects to macroscopic aspects of curvature and compatibility.

\emph{Dislocations} are central to plastic deformation; these are curve-like defects in the crystallographic lattice, and their motion is the prime generator of plastic slip~\cite{HirthLothe82}. 
Because of the multi-scale nature of plastic deformation, the literature contains a wide range of models that describe the motion of dislocations. 
At an atomic scale, the motion of a dislocation line is the net result of thermal atomic motion and the stress state of the crystal lattice; models at this scale take into account all atomic positions and momenta~\cite{MoriartyVitekBulatovYip02,AnciauxJungeHodappChoMolinariCurtin18}. 
At scales larger than atomic distances, dislocations are described as zero-thickness curves in a continuum elastic medium, and models at this scale (\emph{discrete dislocations}) represent the system in terms of the positions of the dislocation curves~\cite{BulatovCai06,ArsenlisCaiTangRheeOppelstrupHommesPierceBulatov07,Hudson18TR}. 
At even larger scales, \emph{dislocation densities} represent the net effect of many dislocations together~\cite{Acharya01,Groma97,GromaBalogh99,GromaZaiserIspanovity16}. 
Finally, at macroscopic scales, the dislocation densities on different slip systems combine to form a net plastic slip, leading to descriptions in terms of  continuum plasticity~\cite{Naghdi90,Chaboche08}.
In addition to the variation in scale, models also vary significantly in other ways, such as whether the evolution is stochastic and whether dislocations are curved or straight.

It is important to note that in this zoo of different models for the same physical system,  none of the models is derived \emph{ab initio;} all are \emph{phenomenological}, in the sense that certain aspects are postulated rather than derived. This is a necessity given the complexity of the physical system, but it has led to the following core problem in plasticity:
\begin{quote}
\textit{How can one assess the trustworthiness of such theoretical descriptions, or equivalently, how can one determine regions of parameter space in which one can consider them valid?}
\end{quote}

In this paper we prove a number of rigorous results that exactly address this question. In contrast to the approach taken in the literature previously, we prove quantitative estimates which relate the different models we consider, rather than proving convergence statements directly. The benefit of these estimates is they explicitly characterise the discrepancies between models in various regions of parameter space mentioned in the question above, and are moreover stronger, since they can then be used to deduce convergence statements.

Given the complexity of the physical system, we restrict ourselves to models of straight and parallel screw dislocations, with the same Burgers vector up to a sign $\pm1$. This allows us to represent dislocation positions as points in a two-dimensional cross section with a sign attached to each point. 


\bigskip

We study four models in total, and prove the connections between them that are  illustrated by Figure~\ref{fig:thms}. Two of the four models are discrete in space (the top row in Figure~\ref{fig:thms}), while the other two are set in continuous space; along the other axis, two models are stochastic evolutions of a finite number of dislocations (the left column), while the other two are deterministic evolutions of dislocation densities. 

\let\SystemFontSize\large
 \begin{figure}[h]
 \centering
   \begin{tikzpicture}[scale=.8, >= latex]
     \def \a {10} 
     \def \b {5} 
     \def \c {1.5} 
     \def \d {1} 

     \fill[white!95!black] (-4,-1) rectangle (\a + 1.5, 1);
     \fill[white!95!black] (-4,0) circle (1);
     \fill[white!95!black] (\a+1.5,0) circle (1);
     
     \begin{scope}[shift={(0,-\b)}]
       \fill[white!95!black] (-4,-1) rectangle (\a + 1.5, 1);
       \fill[white!95!black] (-4,0) circle (1);
       \fill[white!95!black] (\a+1.5,0) circle (1);
     \end{scope}
     
     \draw[rounded corners=22, dotted] (-1.5,3.5) rectangle (1.5,-\b-2);
     \draw[rounded corners=22, dotted] (-1.5+\a,3.5) rectangle (1.5+\a,-\b-2);

     \node at (0,0) {\SystemFontSize $\RW$};
     \draw (-\c, -\d) rectangle (\c, \d);
     \draw (-3, 0) node {\begin{tabular}{c}
     Spatially \\ discrete
     \end{tabular}};
     \draw (0, 2.5) node {\begin{tabular}{c}
     Stochastic \\ Process
     \end{tabular}};
     
     \draw[->] (-3,-1) -- (-3,-\b+1) node[midway, left] {\begin{tabular}{c}
       Discrete--to-- \\ Continuum \\ limit
       \end{tabular}};
     \draw[->] (1.5, 2.5) --++ (\a-3,0) node[midway, above] {Mean Field limit};
     \draw[double,<->] (0,-\d) -- (0,-\b+\d) node[midway, anchor=south, rotate=-90]{$\e \to 0$} node[midway, anchor=north, rotate=-90]{Thm.~\ref{t:RWtoSDE}}; 
     \draw[double,<->] (\c,0) -- (\a-\c,0) node[midway, anchor=south]{$n \to \infty$} node[midway, anchor=north]{Thm.~\ref{t:RWtoMFe}}; 
     \draw[double,<->] (\c,-\d) -- (\a-\c,-\b+\d) node[midway, anchor=south, rotate=-23]{$(\e, n) \to (0, \infty)$} node[midway, anchor=north, rotate=-23]{Cor.~\ref{c:RWtoMFvMFe}, \ref{c:RWtoMFvSDE}}; 

     \begin{scope}[shift={(0,-\b)}]
       \node at (0,0) {\SystemFontSize $\SDE$};
       \draw (-\c, -\d) rectangle (\c, \d);
       \draw (-3, 0) node {\begin{tabular}{c}
         Spatially \\ continuous
         \end{tabular}};       
       
       \draw[double,<->] (\c,0) -- (\a-\c,0) node[midway, above]{$n \to \infty$} node[midway, below]{Thm.~\ref{t:SDEtoMF}};
     \end{scope}
  
     \begin{scope}[shift={(\a,0)}]
       \node at (0,0) {\SystemFontSize $\MFe$};
       \draw (-\c, -\d) rectangle (\c, \d);
       \draw (0, 2.5) node {\begin{tabular}{c}
         Mean \\ Field
         \end{tabular}};
       \draw[double,<->] (0,-\d) -- (0,-\b+\d) node[midway, anchor=south, rotate=-90]{$\e \to 0$} node[midway, anchor=north, rotate=-90]{Thm.~\ref{t:MFetoMF}};
     \end{scope}

     \begin{scope}[shift={(\a,-\b)}]
       \node at (0,0) {\SystemFontSize $\MF$};
       \draw (-\c, -\d) rectangle (\c, \d);
     \end{scope}
     
   \end{tikzpicture}
   \caption{Overview of the four models of this paper and the results which connect them.}
   \label{fig:thms}
 \end{figure}
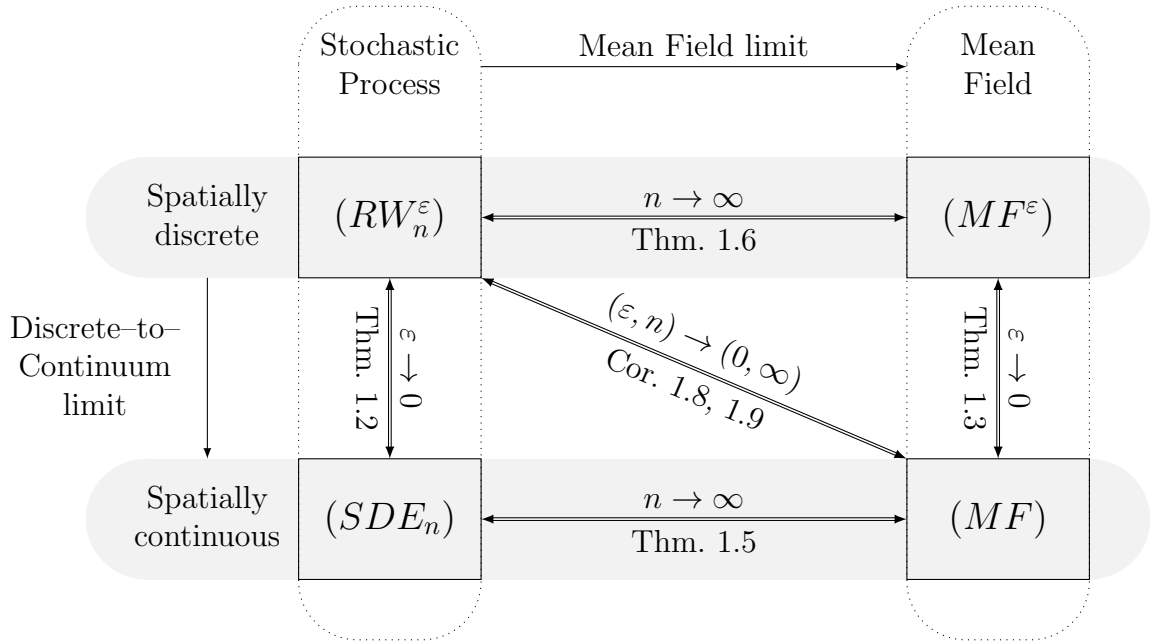 
 
 The model $\SDE$ in the bottom left corner is a stochastic differential equation for~$n$ dislocations in continuous space, with pairwise interaction; single-sign versions of this model have been studied extensively in the community around interacting particle systems~\cite{JabinWang17}. In the bottom right corner, $\MF$ is a mean-field model in continuous space, which corresponds in an appropriate sense to the $n\to\infty$ limit of $\SDE$; the transition from $\SDE$ to $\MF$ is known as `propagation of chaos' and is well studied for the single-sign case with non-singular potentials~\cite{Sznitman91,Philipowski07}.
 
The model  $\RW$ in the top left corner is a random-walk model for $n$ dislocations in a lattice with spacing $\e$. This model appears to be new; following related work on dislocations in lattices~\cite{AlicandroCicalese09,AlicandroDeLucaGarroniPonsiglione14,Hudson17}, we let dislocations jump between the vertices of a lattice, with jump rates that depend on the elastic state of the whole system. 
Finally, in the top right corner, the model $\MFe$ is a mean-field version of this random walk, also in discrete space with lattice spacing $\e$. 
We give a precise definition of these four systems in Section~\ref{sec:models}.

\bigskip
The aim of the paper is to prove the rigorous connections between these four models that are shown in Figure~\ref{fig:thms}. 
The vertical arrows in this figure correspond to estimates of the difference between the laws of $\RW$ and $\SDE$ (on the left) and the difference between the solutions of $\MFe$ and $\MF$ (on the right). For Theorem~\ref{t:RWtoSDE}, this estimate is of the form
\[
\bigl\|\law \RW(t) - \law \SDE(t)\bigr\|\leq f\Bigl(\bigl\|\law\RW(0) - \law \SDE(0)\bigr\|,\e,n,\beta,\delta,t\Bigr),
\]
for some function $f$, and for Theorem~\ref{t:MFetoMF} the structure is similar. For finite values of all the parameters in the argument list of $f$ above, such estimates bound the deviation between the two models; in addition, in certain parameter limits the function $f$ vanishes, implying convergence.

The estimates of Theorems~\ref{t:RWtoSDE}  and~\ref{t:MFetoMF} can be interpreted as convergence results in numerical analysis; for instance, the law of $\SDE$ satisfies a partial differential equation, and the law of $\RW$ satisfies an equation that can be interpreted as a finite-difference discretization.  As a consequence  the method of proof that we use follows the established method due to Lax of combining stability with consistency results.


The two horizontal arrows, on the other hand, indicate estimates of the distance between the solutions of the interacting-particle systems on the left and their mean-field limits on the right. These estimates are of the form 
\[
\mathbb E\bigl\|\mathrm{left}(t) - \mathrm{right}(t) \bigr\| \leq f\Bigl(\bigl\|\mathrm{left}(0) - \mathrm{right}(0)\bigr\|,\e,n,\beta,\delta,t\Bigr).
\]
These estimates are proved by establishing  propagation of chaos in the particle systems $\RW$ and $\SDE$. For $\SDE$ this is a modification of a well-known argument by Sznitman, while for $\RW$ the method of proof appears to be new. 

\medskip
In all of these estimates, the right-hand side $f$  is an explicit function of the initial data and the parameters. The following four parameters play a central role:
\begin{enumerate}[label=(\roman*)]\itemsep0pt
  \item $\e$, the lattice spacing, 
  \item $n$, the number of dislocations, 
  \item $\beta$, the inverse temperature, and 
  \item $\delta$, the size of the dislocation core.  
\end{enumerate}
The parameter $\beta$ characterizes the size of the noise in the two random models $\RW$ and $\SDE$; in the limit $\beta\to\infty$ the noise vanishes. The parameter $\delta$ characterizes the scale at which the interaction between dislocations changes from that of continuum elasticity (at long range) to discrete elasticity (at short range); in a mathematical sense $\delta$ is the scale of regularization of the interaction potential.

\medskip

As consequences of the estimates that we prove, various new convergence statements can be derived. Examples are 
\begin{itemize}
\item $\RW\to \MFe$ as $n\to\infty$, with $\delta = \delta_n\to0$ sufficiently slowly;
\item $\RW\to\MF$ as $n\to\infty$ and $\e = \e_n\to0$, with $\delta_n\to0$ and $\beta_n\to\infty$ sufficiently slowly;
\item $\SDE\to \MF$ as $n\to\infty$, with $\delta_n\to0$ and $\beta_n\to\infty$ sufficiently slowly.
\end{itemize}
We explain these and other consequences in more detail in Section~\ref{sec:convergence-results}. In particular, we highlight that the limiting model $\MF$ is the one developed in \cite{Groma97,GromaBalogh99}. This fundamental model has been used as the basis for many more advanced dislocation density models, and our estimates give a new interpretation of this model as the continuum description of the atomistic and microscopic models $\RW$ and $\SDE$. We make this connection more precise in Section \ref{sec:context}.

\medskip
Section~\ref{sec:main-results} below gives the precise statements of the main theorems of this paper. We first define the discrete and continuous configurations that we will be working with in Section~\ref{sec:notation}, and we specify the dynamics of the four models in Section~\ref{sec:models}.

\subsection{Setting}
\label{sec:notation}


\subsubsection{Configuration spaces}
\label{sec:configuration-spaces}

\paragraph{Continuum configuration space.} For the spatially-continuous models $\SDE$ and $\MF$ we consider the flat torus $\T^2\cong\R^2/\Z^2$ as the spatial domain. This has the advantage that  complications at the boundary and `at infinity' are avoided. We often identify $\T^2$ with translated copies of $Q=[-\tfrac12,\tfrac12)^2$. For $x, y \in \T^2$, we define the metric
\begin{equation*}
  d_{\T^2} (x, y) := \min_{k \in \Z^2} |x-y+k|,
\end{equation*}
where $|\cdot|$ is the Euclidean norm in $\R^2$.

We denote the positions of individual dislocations by $x\in\T^2$ and the positions of $n \geq 2$ dislocations by
\begin{equation*}
  \bx=(x_1,x_2,\dots,x_n) \in (\T^2)^n \cong \T^{2n}.
\end{equation*}
Here and throughout, we will use boldface symbols to distinguish objects which relate to systems of  $n$ dislocations. 

\paragraph{Discrete configuration space.} The spatially discrete models $\RW$ and $\MFe$ are set in the cubic lattice $\L = (\e\Z^2)/ \Z^2 \subset \T^2$, where the atomic lattice spacing $\varepsilon > 0$ is such that $\frac1\varepsilon \in \N$ to fit it inside the torus.  Since $\L \subset \T^2$,  we can use the metric $d_{\T^2}$ to measure the distance between points in $\L$. (The lattice $\L$ contains the positions of the dislocations; the atoms can be considered to be situated on the vertices of the dual lattice~\cite{ArizaOrtiz05,AlicandroDeLucaGarroniPonsiglione14,Hudson17}). 

To distinguish the positions of the dislocations in the lattice from those in the continuous setting above, we write $\ell\in\L$ for a dislocation position on the lattice, and
\begin{equation*}
  \bl=(\ell_1,\ell_2,\dots,\ell_n) \in \L^n := (\L)^n,
\end{equation*}
for the list of positions of $n$ dislocations. 

\paragraph{Lattice increments and difference operators.}
In the lattice model $\RW$ dislocations are assumed to jump at random times to neighbouring lattice sites. The set of directions to neighbouring sites is taken to be
\begin{equation}
  \label{eq:Nhd}
  \Nhd_\e:=\{\pm\e e_1,\pm\e e_2\}\subset\e\Z^2.
\end{equation}
We  refer to such lattice increments as $h\in\Nhd_\e$. To denote possible spatial increments in the $n$-dislocation configuration space, we define
\begin{equation}
  \label{eq:Nhdn}
  \Nhd_\e^n:=\bigcup_{j=1}^n\bigg\{\big(\underbrace{0,\dots,0}_{j-1\text{ times}},h,\underbrace{0,\dots,0}_{n-j\text{ times}}\big) \,\Big|\, h\in\Nhd_\e\bigg\}\subset\e\Z^{2n}.
\end{equation}
We write $\bh$ for an element of $\Nhd^n_\e$.
For a function $f$ defined on $\L$ or $\L^n$ and increments $h\in\Nhd_\e$ and $\bh\in\Nhd_\e^n$, we define the finite-difference operators
\begin{equation*}
  D_h f(\ell) := \frac{f(\ell+h)-f(\ell)}{\e}
  \quad \text{and} \quad
  D_\bh f(\bl) := \frac{f(\bl+\bh)-f(\bl)}{\e},
\end{equation*}
where we have used that $\e = |h| = |\bh|$.

\paragraph{Burgers vectors and extended configuration spaces.}
We assign to each dislocation a Burgers vector, which we identify by its sign $b\in\{\pm 1\}$.
As above, we define lists of Burgers vector signs by
\begin{equation*}
  \bb = (b_1,b_2,\dots,b_n)\in \{\pm1\}^n.
\end{equation*}
For convenience, we divide the indices labelling dislocations according to the sign of their Burgers vector, defining
\begin{equation} \label{Ipm} 
  I^\pm := \{ i : b_i = \pm1 \} 
  \quad \text{and} \quad
  n^\pm = \# I^\pm.
\end{equation}
We note that $n^+ + n^- = n$.

For convenience we extend the spatial configuration spaces introduced
above by identifying $(x,b)\in\T^2\times\{\pm1\}$ and $(\ell,b')\in\L\times\{\pm1\}$ with points in the spaces
\begin{equation*} 
  \T_\pm^2 := \T^2\times\{\pm1\}\quad\text{and}\quad\Lpm := \L\times\{\pm1\}.
\end{equation*}

\subsubsection{Volume measures and probability distributions}
\label{sec:distributions}

Since our focus is on random models, we will consider distributions of dislocation positions, described by probability measures on the configuration spaces. With this aim, we first introduce reference volume measures with total volume scaled to $1$. Using the $n$--fold tensor product, which is defined for any positive measure $\lambda\in\mathcal{M}_+(\T^2)$ by
\begin{equation*}
  \lambda^{\otimes n}:= \underbrace{\lambda\otimes\dots\otimes\lambda}_{n\text{ times}},
\end{equation*}
the reference volume measures on the spaces $\T^2$ and $\T^{2n}$ are
\begin{equation}\label{eq:nu}
  \nu := \mathcal{L}^2\big|_{\T^2}\quad\text{and}\quad \bnu:=\nu^{\otimes n},
\end{equation}
where $\mathcal{L}^2$ is the two-dimensional Lebesgue measure.
Similarly, for the space $\L$ and
$\L^n$ we set
\begin{equation*} 
  \nu_\e := \e^2\sum_{\ell\in\L}\delta_\ell\quad\text{and}\quad \bnu_\e := \nu_\e^{\otimes n}.
\end{equation*}
%

Next we introduce probability distributions. Since the two models involving $n$ dislocations are stochastic, we denote the related $n$--particle probability distribution by $\bmu\in\mathcal{P}(\T^{2n})$ in the continuous case and by $\bmu_\e\in\mathcal{P}(\L^n)$ in the discrete case.

The two models for the dislocation densities are deterministic, however the dislocation densities are most conveniently described as probability measures too. We write these measures as $\rho \in \cP (\T^2_\pm)$ and $\rho_\e \in \cP (\Lpm)$, and note that they are trivially decomposed as
\begin{equation} \label{rhon:pm}
  \rho = \rho^+\otimes \delta_{+1}+\rho^-\otimes \delta_{-1},\quad\text{where}\quad
  \rho^+,\,\rho^- \in \mathcal{M}_+(\T^2).
\end{equation}

For the limit $n \to \infty$ it will be convenient to work with empirical measures for the particle positions rather than with $\bmu$ and $\bmu_\e$. Given $(\bx,\bb)\in\T^{2n}\times\{\pm1\}^n$ and $(\bl,\bb)\in\L^n\times\{\pm1\}^n$, we define the related empirical measures by
\begin{equation*}
  \rho_n:=\frac1n\sum_{i=1}^n\delta_{(x_i,b_i)}\in \mathcal{P}\big(\T_\pm^2\big)
  \quad\text{and}\quad
  \rho_{n,\e}:=\frac1n\sum_{i=1}^n\delta_{(\ell_i,b_i)}\in \mathcal{P}\big(\Lpm\big).
\end{equation*}
Note that the empirical measures are contained in the same spaces as $\rho$ and $\rho_\e$, respectively. For these empirical measures we employ the same decomposition as in \eqref{rhon:pm}.

For the limit  $n \to \infty$ there is no need to distinguish between the distributions $\rho$ and~$\rho_\e$ and their densities given by the usual Radon--Nikodym derivatives $\frac{d \rho}{d\nu}$ and $\frac{d \rho_\e}{d\nu_\e}$. Yet, for the limit $\e \to 0$, the configuration space changes, and therefore we preserve the explicit distinction between the probability distributions $\bmu$ and $\bmu_\e$ and their densities $\frac{d\bmu_\e}{d\bnu_\e}$ and $\frac{d\bmu}{d\bnu}$. 

\subsection{Models of dislocation motion}\label{sec:models}
We now present the four models of dislocation motion that we consider in this paper. 
\subsubsection{Dislocation interaction potential}
\label{sec:disl-inter-potent}
Dislocation motion is driven by the elastic energy of the solid, which in turn is induced by the combination of external loading and the presence of the dislocations. The optimal modelling would therefore be based on appropriate elastic energies for the discrete and continuum systems. This is currently beyond our reach, however, and we take the second-best option: we disregard external loading, and consider in all models energies of a similar, two-point interaction form.

This form is inspired by  linear elasticity theory, which characterizes the interaction between any two dislocations at distances larger than the dislocation core by an interaction potential $V$, defined by
\begin{equation*}
  -\Delta V(\,\cdot\,) = \delta_0 - 1\quad \text{on }\T^2.
\end{equation*}
The derivation of $V$ from linear elasticity is well-established, and the related interaction energy is usually called the `renormalised energy' of dislocations \cite{CermelliLeoni06,BM17}; this name is used in analogy with terminology first coined in the study of Ginzburg--Landau vortices \cite{BBH,SSBook}.

An explicit expression for $V$ is available in terms of Jacobi elliptic functions~\cite[Eq.~(1)]{Mamode14}, and the Fourier series of $V$ is given by
\begin{equation*}
  \widehat V_{k} = \left\{ \begin{aligned}
    -\frac{4 \pi^2}{ |k|^2}
    &&&k \in\Z^2\setminus\{0\}  \\
    0
    &&&k = 0.
  \end{aligned} \right.
\end{equation*}
The function $V$ has a logarithmic singularity at the origin, which is related to the representation of the atomic lattice as a continuum elastic solid. We follow the common approach in the literature to regularise this singularity over a length scale~$\delta$. One might interpret the length scale $\delta$ as the size of the dislocation core; unfortunately, however, with the corresponding assumption $\delta\sim \e$ the estimates that we prove in Section~\ref{sec:main-results} diverge as $\e$ and $\delta$ tend to zero. We address the case $\delta\sim \e$ further in Section \ref{sec:context}. 

We denote the resulting regularised potential by $V_\delta$. The notion of `regularising over the length scale $\delta$' is made precise by the following standing assumption: 

\begin{ass} \label{ass:Vd}
  For each $\delta>0$ the function $V_\delta$ is of class $C^5$, and there exists a constant $\mathsf C_V>0$  such that
  \begin{equation*} 
    \| V_\delta \|_\infty \leq \mathsf C_V \log \frac1\delta 
    \quad \text{and} \quad
    \| d^k V_\delta \|_\infty \leq \frac {\mathsf C_V} {\delta^k}
    \quad \text{for } k = 1,\ldots ,5,
  \end{equation*}
  where $d^k$ is the $k^\mathrm{th}$ order derivative (see Appendix~\ref{app:function-spaces}).
\end{ass}

One possible method to obtain $V_\delta$ is to use a higher-order linear theory of elasticity as in \cite{LM05}; another common choice is to use a mollification $V_\delta = V *\varphi_\delta$ as in \cite{CaiArsenlisWeinbergerBulatov06}. In the latter case, admissible mollifiers $\varphi_\delta$ are non-negative smooth functions which vanish on $\T^2\setminus B_\delta(0)$, scale as $\varphi_\delta (x) := \delta^{-2} \varphi_1 (x/\delta)$, and have unit mass, i.e.
$\int_Q\varphi_\delta = 1$. 

We choose to rescale the total energy in such a way as to ensure that it remains bounded as the number of dislocations tends to infinity. Given a collection of dislocations, the rescaled interaction energy $E_n:\T^{2n}\times \{\pm1\}^n\to\R$ is given by (see~\cite{GarroniVanMeursPeletierScardia19DOI})
\begin{equation*} 
  E_n (\bx,\bb)
  := \frac{1}{n^2}\sum_{i=1}^n \sum_{j=1}^{i-1} b_i b_j V_\delta (x_i - x_j).
\end{equation*}
The rescaled force acting on dislocation $i$ is
\begin{equation} \label{Fi}
  F_i(\bx,\bb) := - n\nabla_{x_i} E_n (\bx,\bb)= - \frac{b_i}{n} \sum_{j\neq i} b_j \nabla V_\delta (x_i - x_j).
\end{equation}
We assemble these individual forces into a configurational force
\begin{equation} \label{F}
  \bF(\bx,\bb) =-n\nabla E_n (\bx,\bb)
  = \biggl(- \frac{b_i}{n} \sum_{j\neq i} b_j \nabla V_\delta (x_i - x_j)\biggr)_{i=1,\dots,n}.
  \end{equation}

\subsubsection{The spatially--discrete random process \texorpdfstring{$\RW$}{RWne}}
\label{s:RW}
We now define the four systems that we consider. 
Model $\RW$ is a continuous--time random Markov process for the
motion of $n$ dislocations in the discrete space~$\L$. The state~$ \bX_\e(t)$ at time $t$ of this process is the vector of dislocation positions $(X_{\e,1}(t),\dots , X_{\e,n}(t))$ $\in \L^n$; the state jumps at random times from a  position $\bl=\bX_\e(t)$
to one of the neighbouring positions $\bl+\bh\in\L^n$ with $\bh\in\Nhd^n_\e$. The jump times are independently and exponentially distributed with rate 
\begin{equation} \label{RW:Rates}
  \mathcal{R}^\e_{n,\bh} (\bl,\bb) 
  := \frac1{\beta\e^2} \exp\Big(\tfrac12\beta \bh \cdot \bF (\bl,\bb) \Big),
\end{equation}
where $\bF$ is the dislocation interaction force defined in \eqref{F}. 


This model with exponential and independent transition times represents a simplified model of a vibrating crystal lattice. In such a lattice, dislocations exist as local minima of the atomistic energy, as demonstrated in \cite{AlicandroDeLucaGarroniPonsiglione14,HudsonOrtner14,HudsonOrtner15,Hudson17}. As a result of thermal fluctuations, a dislocation may overcome the local energetic barrier and move to an adjacent local minimum in any of the directions $h \in \Nhd_\e$. The barrier level varies with the stress near the dislocation, which is reflected in \eqref{RW:Rates} by the dependence on the force $\bF$; for a derivation of~\eqref{RW:Rates}, see~\cite[Sec.~1.6]{BonaschiPeletier16}.

As described above, each dislocation $X_{\e,i}(t)$, $i=1,\dots,n$  has a sign $b_i\in\{\pm1\}$. We assume that the vector of signs $\bb$ is fixed for all time, following the principle that the Burgers vectors of dislocations are conserved \cite{HirthLothe82,HullBacon11}. 
At the initial time we assume that  $\bX_\e$ is randomly distributed according to some distribution, with the initial positions denoted by $\bX_\e^\circ\in \L^n$.

\subsubsection{The spatially--continuous random process \texorpdfstring{$\SDE$}{SDEn}}
\label{s:SDE}

In model $\SDE$, the dislocations are points $\bX = (X_1,\dots,X_n)$ in the continuous spatial domain $\T^{2n}$, and the motion is given by the following family of SDEs on $\T^{2n}$,
\begin{equation} \label{SDE}
  \SDE \qquad 
  d\bX(t)
  = \bF (\bX(t),\bb) \,dt + \sqrt{2\beta^{-1} }\,d\bB(t),
  \quad t \in (0,T),
\end{equation}
where $\bB$ is a $2n$--dimensional Brownian motion. We note that this SDE can be interpreted in the It\^o sense through the identification of $\T^{2n}$ as $(\R^2/\Z^2)^n$, and since $\bF$ is globally Lipschitz, it has well-defined strong solutions.

As in the case of the spatially discrete process $\RW$ described above, we take the Burgers vectors $\bb$ for a given realisation of $\{\bX(t)\}_{t\geq0}$ to be fixed in time. As before, we assume that the initial positions are random, and are denoted by $\bX^\circ\in\T^{2n}$.


\subsubsection{The spatially-discrete mean-field model \texorpdfstring{$\MFe$}{MFe}}
\label{sec:disc-MF}
Model $\MFe$ is a mean-field model which describes the deterministic evolution of the one-particle distribution of dislocations $\rho\in\mathcal{P}\big(\Lpm\big)$, where we recall the definition of the extended configuration space $\Lpm$ from \S\ref{sec:configuration-spaces}. For any $(x,b)\in\Lpm$, $h\in\Nhd_\e$ and $\rho\in\mathcal{P}\big(\Lpm\big)$, the transition rate from $(x,b)$ to $(x+h,b)$ is taken to be
\begin{equation} \label{RW-MF:Rates}
  \mathcal{R}^\e_h (x,b;\rho) 
  := \frac1{\e^2 \beta} \exp\Big(\tfrac12\beta h\cdot  F (x,b;\rho) \Big)
\qquad
 F (x,b;\rho) := 
 - b \,\nabla V_\delta * \big( \rho^+ -\rho^-\big) (x),
\end{equation}
where we recall that $\rho^\pm$ are defined by the decomposition in \S\ref{sec:distributions};
this expression should be compared with the rates for $\RW$, as given in \eqref{RW:Rates}.
The spatially-discrete mean-field model which describes the resulting evolution of $\rho$ is then
\begin{equation} \label{RW:FP-MF}
  \MFe \qquad
\begin{cases}
\ds
\partial_t  \rho^+_\e 
  = \e \sum_{h\in \Nhd_\e} D_{-h}\Big(\mathcal{R}^\e_h (\,\cdot\,,+1;\rho_\e)\rho_\e^+\Big)\\
\ds
\partial_t \rho^-_\e 
  = \e \sum_{h\in \Nhd_\e} D_{-h}\Big(\mathcal{R}^\e_h (\,\cdot\,,-1;\rho_\e)\rho_\e^-\Big)
\end{cases}
\quad \text{on }  \L\times (0, T).
\end{equation}
We suppose that initially the one particle distribution is given by the deterministic initial condition $\rho^\circ_\e\in\mathcal{P}\big(\Lpm\big)$.

\subsubsection{The spatially-continuous mean--field model \texorpdfstring{$\MF$}{MF}}
\label{sec:cont-MF}
The spatially-continuous mean-field model $\MF$ is the continuum analogue of $\MFe$, and is given as the solution to the PDE system
\begin{equation} \label{MF}
  \MF \qquad
\begin{cases}
\ds
\partial_t  \rho^+
  = -\div \big(\rho^+\,F(\,\cdot\,,+1;\rho)\big)+\beta^{-1}\Delta\rho^+\\[2mm]
\ds
\partial_t  \rho^-
  = -\div \big(\rho^-\,F(\,\cdot\,,-1;\rho)\big)+\beta^{-1}\Delta\rho^-
\end{cases}
\quad \text{on }  \T^2 \times (0, T),
\end{equation}
where $F$ is as defined in \eqref{RW-MF:Rates}. Taking $\beta$ large and replacing $V_\delta$ with $V$, $\MF$ was introduced in \cite{CannoneElHajjMonneauRibaud10} as a viscosity approximation of the model in \cite{GromaBalogh99}. As in the case of $\MFe$, we suppose that at initial time the distribution is prescribed deterministically, and is denoted $\rho^\circ\in\mathcal{P}\big(\T_\pm^2\big)$.

\subsubsection{Model parameters and standing assumptions}
\label{sec:model-param-stand}

  We now review the various parameters and their interpretation.
  \begin{itemize}
  \item $\e>0$ is the ratio of the lattice spacing in the atomistic model relative to the domain size; we recall that $\T^2$ is the macroscopic reference domain with side length $1$, and so we assume that $\frac1\e \in \N$. In this case, we note that $\e^{-2}$ is the number of lattice sites, and $\e^2$ is the volume per lattice site.
  \item $\delta>0$ is the length scale in the approximate interaction potential $V_\delta$ taken relative to the domain size; this may be viewed as the `core radius' of the dislocations.
    We assume that $\e\leq\delta \leq1$; the results that we prove require $\e\ll\delta$ to be useful.
  \item $n=n^++n^-\in\N_+$ is the total number of dislocations in the reference domain, where $n^+ \geq 0$ and $n^- \geq 0$ denote the number of positive and negative dislocations respectively.
  \item $\beta>0$ is the inverse temperature of the dislocations in the system, i.e.\ the mean kinetic energy per dislocation in the crystal is assumed to be $\beta^{-1}$.
  \item $T\geq 1$ is the end time of the dynamics, which will be fixed throughout.
  \end{itemize}
  Apart from Assumption~\ref{ass:Vd} and the natural assumptions on the parameters described above, the only technical limitation on the physical parameters we make in order to prove our results is the following.

  \begin{ass}\label{ass:beta-e-delta-bound} 
    We assume that there are fixed constants $C$ and $C'$ such that
    \[0< C\leq\beta \leq \frac{C' \delta}{\e}<+\infty.\]
  \end{ass}
We can interpret this assumption physically as ensuring that the temperature of the system cannot become arbitrarily high, which would lead to $\beta\to0$, nor can it be too low relative to the scale of the lattice spacing; when $\e\ll \delta$, the latter still allows for the 0 temperature limit $\beta \to \infty$.

 Since we seek results which take account of all of the parameters above, we introduce the following convention to clarify this dependence and simplify the statement of our main results.

 \begin{conv}[Polynomial boundedness]\label{conv1}
   A quantity $Q(\alpha_1, \ldots, \alpha_K) \geq 0$ is said to be \emph{polynomially bounded in the parameters }$\alpha_1, \ldots, \alpha_K > 0$ if there exist constants $C > 0$ and $p_1, \ldots, p_k \geq 0$ independent of $\alpha_1, \ldots, \alpha_K$ such that
\begin{equation} \label{pol:growth}
  Q(\alpha_1, \ldots, \alpha_K) \leq C \prod_{k=1}^K \alpha_k^{p_k}
  \quad \text{for all } \alpha_1, \ldots, \alpha_K \text{ large enough}.
\end{equation}
\end{conv}

\subsection{Main results}
\label{sec:main-results}
With the four models identified, we now present our main results, which give estimates of the distance between solutions of the models as a function of the parameters. As mentioned above, Figure~\ref{fig:thms} summarizes these results.

\subsubsection{Discrete-to-continuum estimates}
Our two discrete--to--continuum estimates establish bounds on $L^2$ distances between the laws of the corresponding models. The first of these connects the law of the random walk model $\RW$ introduced in \S\ref{s:RW} with the law of the SDE model $\SDE$ introduced in \S\ref{s:SDE}.

\begin{restatable}[$\RW \leftrightarrow \SDE$]{thm}{thmRWSDE} \label{t:RWtoSDE}
  Let $V_\delta,\e, n, \delta, \beta$ and $T$ satisfy  Assumptions \ref{ass:Vd}--\ref{ass:beta-e-delta-bound}. Let $\bb\in\{\pm1\}^n$ be a fixed collection of Burgers vectors. 
  Suppose that $\bmu_\e: [0,T]\to\mathcal P\big(\L^n\big)$ is
  the law of $\{\bX_\e(t)\}_{t\geq 0}$ evolving under $\RW$
  for some choice of initial conditions
  $\bX_\e^\circ$ with law $\bmu_{\e}^\circ \in\mathcal{P}(\L^n)$.
  Suppose also that $\bmu: [0,T]\to\mathcal P(\T^{2n})$ is
  the law of $\{\bX(t)\}_{t\geq 0}$ evolving under $\SDE$ for some choice of initial conditions
  $\bX^\circ$ with law $\bmu^\circ\in\mathcal{P}(\T^{2n})$.
  
  If $\| \frac{d\bmu^\circ}{d\bnu} \|_{4,\infty}$ is polynomially bounded in $\beta, n, \delta^{-1}$ (see \eqref{pol:growth}), then
    \begin{multline}
      \bigg\|  \frac{d\bmu_\e(t)}{d\bnu_\e}-\frac{d\bmu(t)}{d\bnu}\bigg\|_{L^2(\bnu_\e)}
      \leq
    \bigg\|  \frac{d\bmu_{\e}^\circ }{d\bnu_\e}-\frac{d\bmu^\circ}{d\bnu}\bigg\|_{L^2(\bnu_\e)} \mathrm{e}^{C n \beta \delta^{-2} t}
    + C'' \e^2 \mathrm{e}^{C' n^2 \beta \delta^{-2} T} \\
    \quad \text{for all } t\in[0,T], 
    \end{multline}
   where $C, C', C'' > 0$ depend only on the constants involved in the polynomial bound on the initial data, and those of Assumptions~\ref{ass:Vd} and \ref{ass:beta-e-delta-bound}.
\end{restatable}

\noindent
See Appendix~\ref{app:function-spaces} for the definition of norms such as $\|\cdot\|_{4,\infty}$.

\medskip

Our second main result resembles the one above; it relates the
mean--field models $\MFe$ and $\MF$.

\begin{restatable}[$\MFe \leftrightarrow \MF$]{thm}{thmMFMFe} \label{t:MFetoMF}
  Let $V_\delta,\e, \delta, \beta$ and $T$ satisfy Assumptions~\ref{ass:Vd}--\ref{ass:beta-e-delta-bound}.
  Suppose that $\rho_\e: [0,T]\to \mathcal P\big(\Lpm\big)$ is
  a solution of $\MFe$ for some choice of initial condition $\rho_{\e}^\circ$. 
  Suppose also that $\rho: [0,T]\to\mathcal P\big(\T_\pm^2\big)$ is
  a solution of $\MF$ for some choice of initial condition $\rho^\circ \in C^4(\T^2_\pm)$.
  
  If $\|\frac{d\rho^{\circ,\pm}}{d\nu}\|_{4,\infty}$ is polynomially bounded in $\beta, \delta^{-1}$ (see \eqref{pol:growth}), then
    \begin{multline}
    \sqrt{ \sum_\pm\left\| \frac{d\rho_\e^\pm(t)}{d\nu_\e}-\frac{d\rho^\pm(t)}{d\nu}\right\|_{L^2(\nu_\e)}^2 }
    \leq \left(
           \sqrt{ \sum_\pm\left\|\frac{ d\rho_{\e}^{\circ,\pm} }{d\nu_\e} - \frac{d\rho^{\circ,\pm} }{d\nu} \right\|^2_{L^2(\nu_\e)} }
           + C'' \e^2 \exp \big( C' \beta \delta^{-2} T \big) \right) \\
         \times
  \exp \left( C K \exp \big( 2^5 \mathsf C_V^2 \beta \delta^{-2} t \big) \right) \quad \text{for all } t\in[0,T], 
    \end{multline}
    where 
    \begin{equation}
      K 
      := \beta \delta^{-4} \sum_{\pm} \bigg( \delta^{-1} \Big\| \frac{d\rho^{\circ,\pm} }{d\nu} \Big\|_{1,\infty} + \sqrt{\beta T} \delta^{-4} \Big\| \frac{d\rho^{\circ,\pm} }{d\nu} \Big\|_{\infty} + \beta \bigg),
    \end{equation}
    and $C, C', C'' > 0$ depend only on the constants involved in the polynomial bound on the initial data, and those of Assumptions~\ref{ass:Vd} and \ref{ass:beta-e-delta-bound}.
\end{restatable} 


 \begin{rem}[Small $t$]
 Over short time scales (i.e. when $t \ll 1$) our proof gives sharper estimates than those given in the statements of Theorem~\ref{t:RWtoSDE} and Theorem~\ref{t:MFetoMF}. Indeed, the dependence on the final time $T$ in the exponents in the estimates can be replaced by~$t$ at the cost of a prefactor to the exponential. This prefactor is polynomially bounded in $\beta, n, \delta^{-1}$ for Theorem \ref{t:RWtoSDE}, and polynomially bounded in $\beta, \delta^{-1}$ for Theorem \ref{t:MFetoMF}.
 \end{rem}

\subsubsection{Estimates by mean--field approximations}
The second pair of results  connects particle models for $n$ dislocations to mean-field models; see the horizontal arrows in Figure \ref{fig:thms}. Here, the connection is made through the expectation of the distance between the random empirical measure of the particle system and the solution of the corresponding mean-field model.
The first such result connects the SDE model $\SDE$ to the continuum mean--field model $\MF$.

\begin{restatable}[$\SDE \leftrightarrow \MF$]{thm}{thmSDEMF} \label{t:SDEtoMF}
  Let $V_\delta, n, \delta, \beta$ and $T$ satisfy Assumptions \ref{ass:Vd}--\ref{ass:beta-e-delta-bound}. Let $\rho$ be the solution of $\MF$ for some initial datum $\rho^\circ \in \mathcal P\big(\T_\pm^2\big)$. Fix $\bb\in \{\pm1\}^n$, and let {$\kappa$} be the discrepancy in mass between $\bb$ and $\rho^\circ$,
  \begin{equation}
  \label{def:gamma-mass-disc}
  {\kappa} :=  \left| \int \rho^{\circ,+} - \frac{n^+}n\right|.
  \end{equation}
Let $\bX^\circ = \{ X_i^\circ \}_{1 \leq i \leq n}$ be independent random variables in $\T^2$ with law proportional to $\rho^{\circ,+}$ for $i\in I^+$ and $\rho^{\circ,-}$ for $i\in I^-$. Let $\{\bX_t\}_{0 \leq t \leq T}$ be the stochastic process defined by $\SDE$ with initial datum $\bX^\circ$ and Burgers vectors $\bb$. Then
\begin{equation}
  \E \| \rho_n^\pm (t) - \rho^\pm(t) \|_{1,\infty}^* 
  \leq {\kappa} + C \frac{\log n}{\sqrt n} + 2\mathsf C_V  \Big(\frac1{\sqrt n}  + {\kappa}\Big) \frac t\delta \mathrm{e}^{2\mathsf C_V \delta^{-2} t} 
  \quad \text{for all } t \in [0,T],
\end{equation}
where $\rho_n^\pm$ are the random empirical measures of $(\bX_t, \bb)$ defined in Section \ref{sec:distributions}, and $C>0$ depends only on the initial data and the constants of Assumptions~\ref{ass:Vd} and \ref{ass:beta-e-delta-bound}.
\end{restatable}

\noindent
The dual bounded-Lipschitz norm $\|\cdot\|_{1,\infty}^*$ is defined in Appendix~\ref{app:function-spaces}.

\medskip

The second of these results connects the random walk model $\RW$ to the discrete mean--field model $\MFe$.

\begin{restatable}[$\RW \leftrightarrow \MFe$]{thm}{thmRWMFe} \label{t:RWtoMFe}
Let $V_\delta,\e, n, \delta, \beta$ and $T$ satisfy Assumptions~\ref{ass:Vd}--\ref{ass:beta-e-delta-bound}. 
Let~$\rho_\e$ be the solution of $\MFe$ for some initial condition $\rho^\circ_\e \in \mathcal P\big( \Lpm \big)$. Fix $\bb\in \{\pm1\}^n$, and let~{$\kappa$} be the discrepancy in mass between $\bb$ and $\rho^\circ_\e$, 
  \begin{equation}
    \label{def:gamma-mass-disc-discrete}
  {\kappa} :=  \left| \int \rho^{\circ,+}_\e - \frac{n^+}n\right|.
  \end{equation}
Let $\bX^\circ_\e = \{ X_{i,\e}^\circ \}_{1 \leq i \leq n}$ be independent random variables in $\L$ with with law proportional to $\rho^{\circ,+}_\e$ for $i\in I^+$ and $\rho^{\circ,-}_\e$ for $i\in I^-$. 
Let $\{\bX_{\e,t}\}_{0 \leq t \leq T}$ be the stochastic process defined by $\RW$ with initial datum $\bX^\circ_\e$ and Burgers vectors $\bb$. Then
\begin{equation}
  \E \| \rho_{\e,n}^\pm (t) - \rho_\e^\pm(t) \|_{1,\infty}^*
  \leq {\kappa} + C'' \frac{\log n}{\sqrt n} + C'  \Big(\frac1{\sqrt n}  + {\kappa}\Big) \frac t\delta \mathrm{e}^{C\delta^{-2} t} 
  \quad \text{for all } t \in [0,T],
\end{equation}
where $\rho_{\e,n}^\pm$ are the random empirical measures of $(\bX_{\e,t}, \bb)$, and $C,\ C',\ C'' > 0$ depend only on the initial data and the constants of Assumptions~\ref{ass:Vd} and \ref{ass:beta-e-delta-bound}. 
\end{restatable}

\begin{rem}[Random $\bb$] \label{r:b}
Since the estimates of Theorems~\ref{t:SDEtoMF} and~\ref{t:RWtoMFe} only depend on $\bb$ through the quantity $\kappa$, the statements readily generalise to the case in which the vector $\bb$ is chosen randomly. As an example, consider the situation  in the context of Theorem~\ref{t:SDEtoMF} where we select the pairs $\{(X_i^\circ,b_i)\}_{i=1,\dots,n}$ randomly and independently from the distribution $\rho^\circ\in \mathcal P(\T_\pm^2)$. By conditioning the distribution of $\{X_i^\circ\}_i$ on  $\bb$, Theorem~\ref{t:SDEtoMF} can be applied to each realization of $\bb$ with a mass discrepancy $\kappa(\bb)$ and with constants that are independent of~$\bb$. After taking a final  expectation over $\bb$ we find the estimate
\begin{equation*}
  \E \| \rho_n^\pm (t) - \rho^\pm(t) \|_{1,\infty}^* 
  \leq \E [\kappa(\bb)] + C \frac{\log n}{\sqrt n} + 2\mathsf C_V  \Big(\frac1{\sqrt n}  + \E[\kappa(\bb)]\Big) \frac t\delta \mathrm{e}^{2\mathsf C_V \delta^{-2} t}   \quad \text{for all } t \in [0,T].
\end{equation*}
Since $n^+$ has a binomial distribution with parameters $n$ and  $p := \int\rho^{\circ,+}$, we can estimate the  expectation of $\kappa$ by
\[
\E [\kappa(\bb)] 
\leq \E \left|p-\frac {n^+}n\right| 
\leq \sqrt{ \E \left|p-\frac {n^+}n\right|^2 }
= \sqrt{ \frac{p(1-p)}n } 
\leq \frac1{2\sqrt n}.
\] 
In this way the fixed-$\bb$ estimates of Theorem~\ref{t:SDEtoMF} and~\ref{t:RWtoMFe} generalize to random $\bb$, in which case the terms related to $\kappa$ can be absorbed by the $\kappa$--independent terms.
\end{rem}

\subsubsection{Direct estimates connecting \texorpdfstring{$\RW$ and $\MF$}{RWn and MF}}
As a consequence of the above results, we obtain the following corollary, linking the random-walk model $\RW$ defined in \S\ref{s:RW} directly to the mean--field continuum model $\MF$ defined in \S\ref{sec:cont-MF}.

\begin{cor}[$\RW \leftrightarrow \MF$ via $\MFe$]
 \label{c:RWtoMFvMFe}
Let the setting be as in both Theorems~\ref{t:MFetoMF} and~\ref{t:RWtoMFe}, and denote the right-hand sides of the corresponding estimates as $R_1$ and $R_2$ respectively. Then 
\begin{equation}
  \E \bigg[ \sum_\pm \| \rho_{\e,n}^\pm (t) - \rho^\pm(t) \|_{1,\infty}^* \bigg]
  \leq 2(R_1 + R_2) + C \e \mathrm{e}^{\beta \delta^{-2} T} \mathrm{e}^{2^5 \mathsf C_V \beta \delta^{-2} t}
  \quad \text{for all } t \in [0,T].
\end{equation}
where $C>0$ depends only on the constants involved in the polynomial bound on the initial data and  Assumptions~\ref{ass:Vd} and \ref{ass:beta-e-delta-bound}.
\end{cor}

For the next result, we introduce the Wasserstein distance $W_1$ on $\cP (\T^{2n})$ through the Kantorovich duality:
\begin{equation} \label{W1:def}
  W_1 (\bmu, \bmu') := \sup_{ \| d \bvarphi \|_\infty \leq 1 } \int_{T^{2n}} \bvarphi \, d(\bmu - \bmu').
\end{equation}
By the particular definition of the Lipschitz constant $\| d \bvarphi \|_\infty$ that we use (see Appendix~\ref{app:function-spaces}) the metric $W_1$ scales linearly in $n$:
\[
  W_1( \rho_1 \otimes \dots \otimes \rho_n, \mu_1 \otimes \dots \otimes \mu_n )
  = \sum_{i=1}^n W_1 (\rho_i, \mu_i).
\]
Hence, the estimate below establishes a convergence rate for $\frac1n W_1$.

\begin{cor}[$\RW \leftrightarrow \MF$ via $\SDE$]
 \label{c:RWtoMFvSDE}
Let the setting be as in both Theorems \ref{t:RWtoSDE} and \ref{t:SDEtoMF}. Let $R_1$ be the right-hand side of the corresponding estimate in Theorem \ref{t:RWtoSDE}, and set  
\begin{equation*}
  \brho := \bigotimes_{i=1}^n \frac{ \rho^{b_i} }{ \| \rho^{b_i} \|_{TV} } \;\in C \big([0,T]; \cP (\T^{2n}) \big),
  \qquad \rho^b := \left\{ \begin{array}{ll}
    \rho^+
    & \text{if } b = 1 \\
    \rho^-
    & \text{if } b = -1.
  \end{array} \right.
\end{equation*}
Then 
\begin{equation}
  \frac1n W_1 \big( \bmu_\e (t), \brho (t) \big)
  \leq  \frac{R_1}{\sqrt n} + 2\mathsf C_V  \Big(\frac1{\sqrt n}  + {\kappa}\Big) \frac t\delta \mathrm{e}^{2\mathsf C_V \delta^{-2} t} + C \e \mathrm{e}^{2\mathsf C_V \beta \delta^{-2} n T}
  \quad \text{for all } t \in [0,T],
\end{equation}
where $C>0$ depends only on the constants involved in the polynomial bound on the initial data and  Assumptions~\ref{ass:Vd} and \ref{ass:beta-e-delta-bound}.
\end{cor}

\subsection{Discussion of results}
\label{sec:discussion}

\subsubsection{Proof techniques}
With the exception of Theorem \ref{t:RWtoMFe}, the proofs of our results mainly use techniques drawn from across the field of applied analysis, and we believe it is the wide variety of the tools which we collect together here which makes the study mathematically interesting. 

A fundamental step in all the proofs is an application of Gronwall's Lemma, which results in the exponential terms in the estimates. In the proofs of Theorems \ref{t:RWtoSDE} and \ref{t:MFetoMF}, we apply this result to classic consistency and stability estimates from numerical analysis, combined with regularity estimates on the solutions to $\SDE$ and $\MF$. In the proof of Theorem~\ref{t:SDEtoMF}, we apply the result to an estimate derived from a Glivenko-Cantelli argument and a classical propagation-of-chaos result, adapted here to particle systems of two species. 

For Theorem \ref{t:RWtoMFe}, however, we use a novel proof technique to obtain an estimate to which we can apply Gronwall's Lemma. More specifically, we use a random time-change characterization as a coupling technique to establish a propagation-of-chaos result for a continuous-time random walk, along with arguments similar to those used in the proof of Theorem~\ref{t:SDEtoMF}.

\subsubsection{Assumptions required and sharpness of estimates}
\label{ss:discussion-of-assumptions}

\paragraph{Regularity of $V_\delta$.}
While the standing Assumption \ref{ass:Vd} imposes bounds on the derivatives of $V_\delta$ up to order $5$, we only need this for the discrete-to-continuum estimates of Theorems~\ref{t:RWtoSDE} and~\ref{t:MFetoMF}. In fact, we can weaken this assumption to bounds on lower-order derivatives at the cost of weakening the decay rate in $\e$. On the other hand, the two mean-field estimates of Theorems~\ref{t:SDEtoMF} and~\ref{t:RWtoMFe} only require Assumption \ref{ass:Vd} to hold up to second order, which corresponds to Lipschitz continuity of the vector field~$F$.

\paragraph{Dependence of Theorem~\ref{t:RWtoSDE} on particle number.}
The $n$ in the exponent in Theorem~\ref{t:RWtoSDE} is unavoidable when one uses the $L^2(\bnu_\e)$-norm. To find an $n$-independent alternative, a different metric, such as the relative entropy, could be used. However, such an approach calls for different proof methods, which is beyond the scope of the current work.

\paragraph{Character of estimates.}
In all of our proofs, we do not use any structural properties of the solutions other than their spatial regularity. Indeed, many of our estimates are based on maximizing the interaction force over all admissible particle configurations. Since such bounds are sharp only on a small, unstable region of the phase space, we do not believe that our estimates are accurate on the macroscopic time scale of order $1$. Such estimates are likely to require Lyapunov functions, and even for simpler particle systems (i.e.\ deterministic and single-sign), we are not aware of any generic techniques that can provide sharper estimates.

\paragraph{Extension to other lattice structures.}
For simplicity, we have chosen here to consider only a simple cubic lattice, which results in a square lattice $\L$ of screw dislocation positions. Other physically relevant cases include the triangular and hexagonal lattices (see \cite{Hudson17} for their treatment in the low-temperature limit). Our setting and results could easily be adjusted to incorporate such lattices too. The main two adjustments are the alteration of the set of neighbouring sites for the dislocations, $\Nhd$ (defined in \eqref{eq:Nhd}), and the change in the periodic domains to accommodate the lattice. Apart from a change of domain, the net effect of these adjustments on our results is a change in value of some of the generic constants in our estimates.

\subsubsection{Convergence corollaries}
\label{sec:convergence-results}
In this section we extract from the main theorems and their two corollaries a number of asymptotic regimes in the five-dimensional parameter space $(\e, n, \beta, \delta, T)$ as $(\e, n) \to (0, \infty)$ in which the solutions to our four models are close. Again, we restrict the parameters to those satisfying Assumption~\ref{ass:Vd} and Assumption~\ref{ass:beta-e-delta-bound}. 

\begin{enumerate}
  \item Given the setting of Theorem \ref{t:RWtoSDE}, if 
  \begin{equation} \label{IC:diff:bd:RWtoSDE}
    \bigg\|  \frac{d\bmu_{\e}^\circ }{d\bnu_\e}-\frac{d\bmu^\circ}{d\bnu}\bigg\|_{L^2(\bnu_\e)} 
     \lesssim \e^2 \mathrm{e}^{C n^2 \beta \delta^{-2} T}
     \quad \text{and} \quad \log \frac1\e \gg n^2 \beta \delta^{-2} T, 
   \end{equation} 
     as $(\e, n) \to (0, \infty)$ then
     \[ \bigg\|  \frac{d\bmu_\e(t)}{d\bnu_\e}-\frac{d\bmu(t)}{d\bnu}\bigg\|_{L^2(\bnu_\e)} \to 0
     \quad \text{uniformly in } t \in [0,T]. \]
     
  \item Given the setting of Theorem \ref{t:MFetoMF}, if 
  \begin{equation} \label{IC:diff:bd:MFetoMF}
    \left\|\frac{ d\rho_{\e}^{\circ,\pm} }{d\nu_\e} - \frac{d\rho^{\circ,\pm} }{d\nu} \right\|_{L^2(\nu_\e)}
     \lesssim \e^2 \mathrm{e}^{C' \beta \delta^{-2} T}
     \quad \text{and} \quad \log \log \frac1\e \gg \beta \delta^{-2} T, 
  \end{equation}
     as $(\e, n) \to (0, \infty)$ then
     \[ \left\|\frac{ d\rho_{\e}^{\pm}(t) }{d\nu_\e} - \frac{d\rho^{\pm} (t) }{d\nu} \right\|_{L^2(\nu_\e)} \to 0
     \quad \text{uniformly in } t \in [0,T]. \]   
     
  \item Given the setting of Theorem \ref{t:SDEtoMF}, let $\bb$ be randomly sampled from $\rho^\circ$ (see Remark \ref{r:b}). If
  \[ \log n \gg \delta^{-2} T, \]
     as $n\to\infty$, then
     \[ \E \| \rho_n^\pm (t) - \rho^\pm(t) \|_{1,\infty}^* \to 0
     \quad \text{uniformly in } t \in [0,T]. \] 
     
  \item Given the setting of Theorem \ref{t:RWtoMFe}, let $\bb$ be randomly sampled from $\rho_\e^\circ$. If 
  \[ \log n \gg \delta^{-2} T, \]
     as $n\to\infty$, then
     \[ \E \| \rho_{\e,n}^\pm (t) - \rho_\e^\pm(t) \|_{1,\infty}^* \to 0
     \quad \text{uniformly in } t \in [0,T]. \]
     
  \item Given the setting of Corollary \ref{c:RWtoMFvMFe}, let $\bb$ be randomly sampled from $\rho_\e^\circ$. If the bound on the difference in initial conditions in \eqref{IC:diff:bd:MFetoMF} is satisfied,
  \begin{equation} \label{pmreg1}
     \log n \gg \delta^{-2} T
  \quad \text{and} \quad
  \log \log \frac1\e \gg \beta \delta^{-2} T, 
  \end{equation} 
     as $(\e, n) \to (0, \infty)$, then
     \[ \E \| \rho_{\e,n}^\pm (t) - \rho^\pm(t) \|_{1,\infty}^* \to 0
     \quad \text{uniformly in } t \in [0,T]. \] 
     
  \item Given the setting of Corollary \ref{c:RWtoMFvSDE}, let $\bb$ be randomly sampled from $\rho^\circ$. If the bound on the difference in initial conditions in \eqref{IC:diff:bd:RWtoSDE} is satisfied (note that $\bmu^\circ = \mu^{\otimes n}$),
  \begin{equation} \label{pmreg2} 
    \log n \gg \delta^{-2} T
  \quad \text{and} \quad
  \log \frac1\e \gg n^2 \beta \delta^{-2} T, 
  \end{equation} 
     as $(\e, n) \to (0, \infty)$, then
     \[ \frac1n W_1 \big( \bmu_\e (t), \brho (t) \big) \to 0
     \quad \text{uniformly in } t \in [0,T]. \] 
\end{enumerate}
All of the statements above can be proved by taking logarithms of the right-hand side of the estimates in our main theorems, and taking appropriate limits, using the relevant assumptions.

We remark that neither of the parameter regimes in \eqref{pmreg1} and \eqref{pmreg2} is contained in the other. Indeed, \eqref{pmreg1} requires no upper bound on $n$, and \eqref{pmreg2} has a weaker upper bound on~$\beta$ than \eqref{pmreg1}.
Furthermore, we expect that the convergence in the six statements above holds in much larger regions in parameter space than the specified ones; this is a ramification of the lack of sharpness of the estimates in our main results, as discussed in \S\ref{ss:discussion-of-assumptions}.

\paragraph{Initial conditions.} Note that it is always possible for any given initial distributions $\bmu^\circ$ and $\rho^\circ$ on the continuous state spaces to find sequences of initial distributions $\bmu_\e^\circ$ and $\rho_\e^\circ$ on the discrete state spaces for which the bounds in \eqref{IC:diff:bd:RWtoSDE} and \eqref{IC:diff:bd:MFetoMF} are satisfied. In fact, the choices
  \begin{align*}
    \frac{d\rho_\e^{\circ,\pm}}{ d\nu_\e } (\ell, b) 
    &:= (1 + \alpha_\e) \frac{d\rho^{\circ,\pm}}{ d\nu } (\ell,b)
    &&\text{for all } (\ell,b) \in \Lpm, \\
    \frac{ d\bmu_\e^\circ }{ d\bnu_\e } (\bell) 
    &:= (1 + a_\e) \frac{ d\bmu^\circ }{ d\bnu } (\bell)
    &&\text{for all } \bell \in \L^n,
  \end{align*}
  satisfy these bounds, where the constants $\alpha_\e, a_\e$ are chosen to ensure that the resulting densities correspond to probability measures. Since $\frac{d\rho^{\circ,\pm}}{ d\nu }$ and $\frac{ d\bmu^\circ }{ d\bnu }$ are assumed to be of class~$C^4$, a computation similar to that in \eqref{fvp:2infty} and \eqref{fvp:2fatinfty} shows that $|\alpha_\e|$ and $|a_\e|$ can be bounded respectively by $\e^2 \|\frac{d\rho^{\circ,\pm}}{d\nu}\|_{2,\infty}$ and $\e^2 \| \frac{d\bmu^\circ}{d\bnu} \|_{2,\infty}$. Then, using the polynomial bound assumed on the initial data, the terms other than $\e^2$ can be absorbed in the exponential, possibly by choosing a larger constant in the exponent.
  
\subsubsection{Scientific and mathematical context}
\label{sec:context}
As mentioned in the introduction, this work has been carried out in the context of a series of studies on related questions. In particular, Figure~\ref{fig:overview} illustrates this context. To clearly indicate the parametric dependence of all models shown, we have included $\beta$ and $\delta$ as sub- and superscripts.

 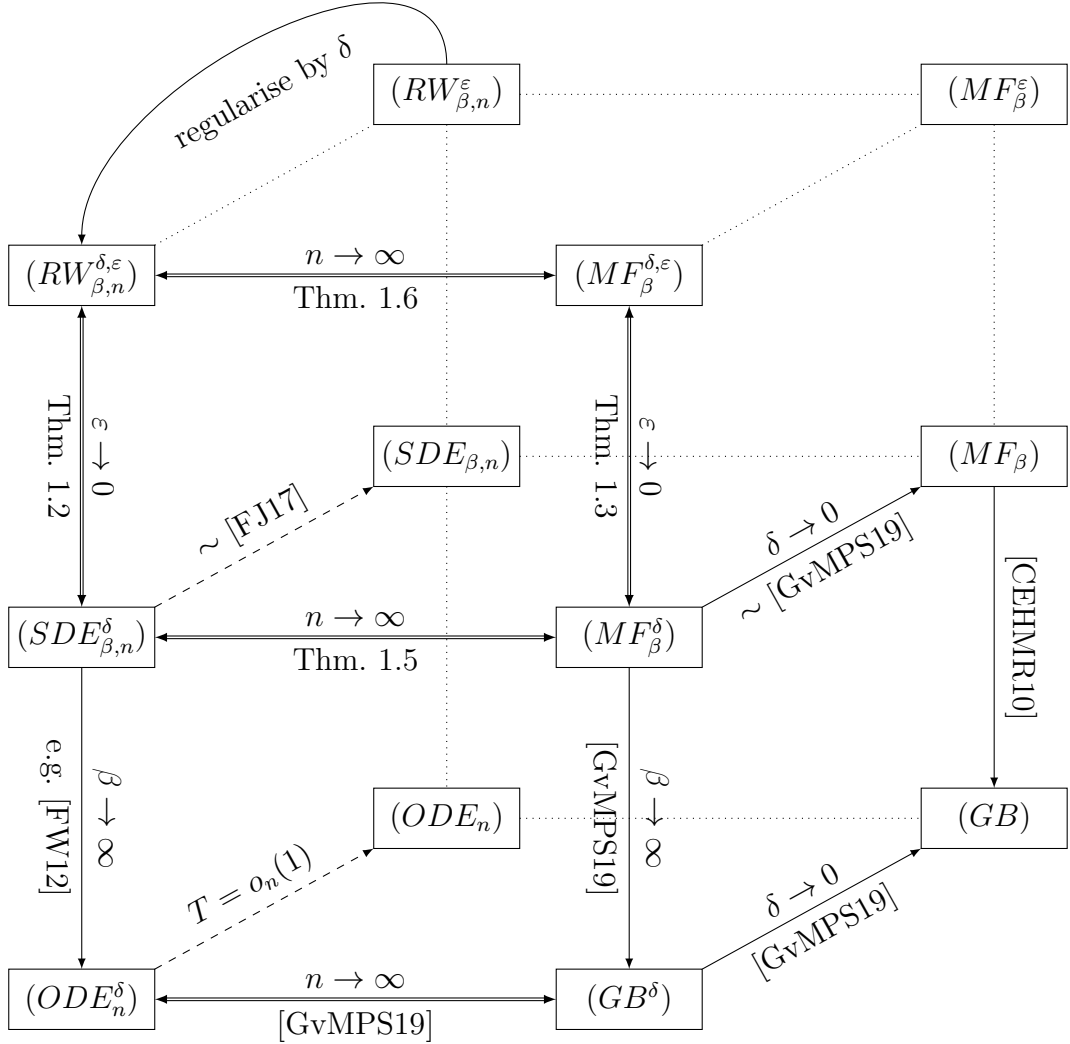
\begin{figure}[p]
 \centering
   \begin{tikzpicture}[scale=.8, >= latex]
     \def \a {9} 
     \def \b {6} 
     \def \c {1.2} 
     \def \d {.5} 
     \def \ee {6} 
    
     \node at (0,0) {$(RW^{\delta, \e}_{\beta,n})$};
     \draw (-\c,-\d) rectangle (\c,\d);
     \draw[double,<->] (0,-\d) -- (0,-\ee+\d) node[midway, anchor=south, rotate=-90]{$\e \to 0$} node[midway, anchor=north, rotate=-90]{Thm.~\ref{t:RWtoSDE}}; 
     \draw[double,<->] (\c,0) -- (\a-\c,0) node[midway, anchor=south]{$n \to \infty$} node[midway, anchor=north]{Thm.~\ref{t:RWtoMFe}}; 
     \draw[dotted] (\c,\d) -- (\b - \c, \a-\b-\d); 
     \draw[<-] (0,\d) to [out=90,in=90, looseness=1] (\b, \a-\b+\d);
     \draw (\b/2, \a-\b) node[rotate=29] {regularise by $\delta$};

     \begin{scope}[shift={(0,-\ee)}]
       \node at (0,0) {$(SDE^{\delta}_{\beta,n})$};
       \draw (-\c,-\d) rectangle (\c,\d);
       \draw[double,<->] (\c,0) -- (\a-\c,0) node[midway, above]{$n \to \infty$} node[midway, below]{Thm.~\ref{t:SDEtoMF}};
       \draw[->, dashed] (\c,\d) -- (\b - \c, \a-\b-\d) node[midway, anchor=south, rotate=29]{$\sim$ \cite{FJ17}};
       \draw[->] (0,-\d) -- (0,-\ee+\d) node[midway, anchor=south, rotate=-90]{$\beta \to \infty$} node[midway, anchor=north, rotate=-90]{e.g.~\cite{FreidlinWentzell12}};
     \end{scope}
     
     \begin{scope}[shift={(\a,0)}]
       \node at (0,0) {$(MF^{\delta, \e}_{\beta})$};
       \draw (-\c,-\d) rectangle (\c,\d);
       \draw[double,<->] (0,-\d) -- (0,-\ee+\d) node[midway, anchor=south, rotate=-90]{$\e \to 0$} node[midway, anchor=north, rotate=-90]{Thm.~\ref{t:MFetoMF}};
       \draw[dotted] (\c,\d) -- (\b - \c, \a-\b-\d); 
     \end{scope}
     
     \begin{scope}[shift={(\a,-\ee)}]
       \node at (0,0) {$(MF^{\delta}_{\beta})$};
       \draw (-\c,-\d) rectangle (\c,\d);
       \draw[->] (\c,\d) -- (\b - \c, \a-\b-\d) node[midway, anchor=south, rotate=29]{$\delta \to 0$} node[midway, anchor=north, rotate=29]{$\sim$ \cite{GarroniVanMeursPeletierScardia19DOI}};
       \draw[->] (0,-\d) -- (0,-\ee+\d) node[midway, anchor=south, rotate=-90]{$\beta \to \infty$} node[midway, anchor=north, rotate=-90]{\cite{GarroniVanMeursPeletierScardia19DOI}};
     \end{scope}
     
     \begin{scope}[shift={(0,-2*\ee)}]
       \node at (0,0) {$(ODE^{\delta}_{n})$};
       \draw (-\c,-\d) rectangle (\c,\d);
       \draw[double,<->] (\c,0) -- (\a-\c,0) node[midway, above]{$n \to \infty$} node[midway, below]{\cite{GarroniVanMeursPeletierScardia19DOI}};
       \draw[->, dashed] (\c,\d) -- (\b - \c, \a-\b-\d) node[midway, anchor=south, rotate=29]{$T = o_n(1)$};
     \end{scope}
     
     \begin{scope}[shift={(\a,-2*\ee)}]
       \node at (0,0) {$(GB^{\delta})$};
       \draw (-\c,-\d) rectangle (\c,\d);
       \draw[->] (\c,\d) -- (\b - \c, \a-\b-\d) node[midway, anchor=south, rotate=29]{$\delta \to 0$} node[midway, anchor=north, rotate=29]{\cite{GarroniVanMeursPeletierScardia19DOI}};
     \end{scope}
     
     \begin{scope}[shift={(\b,\a-\b)}]
       \node at (0,0) {$(RW^{\e}_{\beta,n})$};
       \draw (-\c,-\d) rectangle (\c,\d);
       \draw[dotted] (\c,0) -- (\a-\c,0);
       \draw[dotted] (0,-\d) -- (0,-\ee+\d);
     \end{scope}
     
     \begin{scope}[shift={(\b+\a,\a-\b)}]
       \node at (0,0) {$(MF^{\e}_{\beta})$};
       \draw (-\c,-\d) rectangle (\c,\d);
       \draw[dotted] (0,-\d) -- (0,-\ee+\d); 
     \end{scope}
     
     \begin{scope}[shift={(\b,\a-\b-\ee)}]
       \node at (0,0) {$(SDE_{\beta,n})$};
       \draw (-\c,-\d) rectangle (\c,\d);
       \draw[dotted] (\c,0) -- (\a-\c,0); 
       \draw[dotted] (0,-\d) -- (0,-\ee+\d); 
     \end{scope}
     
     \begin{scope}[shift={(\b+\a,\a-\b-\ee)}]
       \node at (0,0) {$(MF_{\beta})$};
       \draw (-\c,-\d) rectangle (\c,\d);
       \draw[->] (0,-\d) -- (0,-\ee+\d)  node[midway, anchor=south, rotate=-90]{\cite{CannoneElHajjMonneauRibaud10}};
     \end{scope}
     
     \begin{scope}[shift={(\b,\a-\b-2*\ee)}]
       \node at (0,0) {$(ODE_n)$};
       \draw (-\c,-\d) rectangle (\c,\d);
       \draw[dotted] (\c,0) -- (\a-\c,0); 
     \end{scope}
     
     \begin{scope}[shift={(\b+\a,\a-\b-2*\ee)}]
       \node at (0,0) {$(GB)$};
       \draw (-\c,-\d) rectangle (\c,\d);
     \end{scope}
   \end{tikzpicture}
   \caption{Extended version of the overview in Figure \ref{fig:thms}. In this three-dimensional diagram, the passage from foreground to background corresponds to the limit $\delta \to 0$, the passage from left to right corresponds to the limit $n \to \infty$, and from top to bottom corresponds first to the limit $\varepsilon \to 0$, and then the limit $\beta \to \infty$. Double arrows indicate quantitative estimates between the corresponding solutions, single arrows indicate convergence results, dashed arrows indicate partial results, and dotted lines indicate connections which are currently not rigorously proven. Abbreviations such as $(SDE^{\delta}_{\beta,n})$ and $(RW^{\delta, \e}_{\beta,n})$ are used in the rest of the paper without $\beta$ and $\delta$, i.e.\ in this case $\SDE$ and $\RW$.}
   \label{fig:overview}
 \end{figure} 

  \paragraph{The limits $\beta \to \infty$ and $\delta \to 0$ of the continuum models.} 
The main topic of this paper is a study of the models $\RW$, $\SDE$, $\MFe$ and $\MF$. These models are shown at the top of the front face in Figure~\ref{fig:overview} (in their extended notation $(RW^{\delta,\e}_{\beta,n})$, $(SDE^\delta_{\beta,n})$, $(MF^{\delta,\e}_\beta)$, and $(MF^\delta_\beta)$) and the connections between them were already shown in more detail in Figure~\ref{fig:thms}. In the models $(SDE^\delta_{\beta,n})$ and $(MF^\delta_\beta)$ with a continuous state space, classical techniques allow us to consider a further limit in which $\beta \to \infty$; the corresponding limits are shown at the bottom of the front face. In the discrete setting, the limit $\beta \to \infty$ is more complex, and results in more complex dynamics such as rate-independent systems; see for example \cite{BonaschiCarrilloDi-FrancescoPeletier15,MoraPeletierScardia17,Hudson17}.

Taking $\beta\to\infty$ in $(SDE_{\beta,n}^{\delta})$, we obtain
 \begin{equation} \label{ODEen}
   (ODE^\delta_n) \qquad 
   \frac d{dt} \bx (t)
   = \bF(\bx(t), \bb),
   \quad t \in (0,T), \ i = 1,\ldots,n
 \end{equation}
with $\bF$ as in \eqref{F}. 
The system \eqref{ODEen} is the starting point of \cite{GarroniVanMeursPeletierScardia19DOI}, where convergence to $(GB)$ is proven in the joint limit where $(n, \delta) \to (\infty, 0)$. The proof of this joint limit is a combination of two results; the first is a quantitative estimate between the solution of $(ODE^\delta_n)$ and that of the `regularised Groma-Balogh equations' (see~\cite{GromaBalogh99})
\begin{equation*} 
   (GB^\delta) \qquad 
   \partial_t \rho^\pm
   = \div \big( \rho^\pm [ \nabla V_\delta * (\rho^\pm - \rho^\mp) ] \big),
   \quad \mathcal D'(\T^2 \times (0, T)),
 \end{equation*}
which can alternatively be viewed as a zero-viscosity limit in~$(MF_\beta^\delta)$. The second is a convergence result of $(GB^\delta)$ to the (non-regularized) Groma-Balogh equations
\begin{equation} \label{GB}
   (GB) \qquad 
   \partial_t \rho^\pm
   = \div \big( \rho^\pm [ \nabla V * (\rho^\pm - \rho^\mp) ] \big),
   \quad \mathcal D'(\T^2 \times (0, T)).
 \end{equation} 
In the proof, $(MF_\beta^\delta)$ is used as a viscosity approximation of $(GB^\delta)$, which connects these two models by a convergence result. This proof method is an adaptation of the technique used in \cite{CannoneElHajjMonneauRibaud10}, where the existence of finite-entropy solutions to $(GB)$ is proven by passing to the limit $\beta \to \infty$ in
 \begin{equation*} 
   (MF_\beta) \qquad 
   \partial_t \rho^\pm
   = \div \big( \rho^\pm [ \nabla V * (\rho^\pm - \rho^\mp) ] \big) + \beta^{-1} \Delta \rho^\pm,
   \quad \mathcal D'(\T^2 \times (0, T)).
 \end{equation*}
Finally, the connection from $(MF_\beta^\delta)$ to $(MF_\beta)$ follows by the same argument as in the proof of $(GB^\delta)$ to $(GB)$ in \cite{GarroniVanMeursPeletierScardia19DOI}.

We note that since $(GB)$ contains neither a Laplacian nor a regular interaction potential in the right-hand side, the existence and uniqueness of solutions to \eqref{GB} is nontrivial. Besides the existence result of global-in-time finite-entropy solutions in \cite{CannoneElHajjMonneauRibaud10}, uniqueness results are found in \cite{Mainini12a, 
LiMiaoXue14} for regular initial data. Also, while $(GB)$ was originally developed for edge dislocations, it is shown in \cite{GarroniVanMeursPeletierScardia19DOI} that the well-posedness theory in \cite{CannoneElHajjMonneauRibaud10} also extends to the case of screw dislocations.

The connections explained so far allow us to connect $(RW_{\beta,n}^{\delta,\e})$ to $(GB)$. The quantitative estimates of Theorems~\ref{t:RWtoSDE}--\ref{t:RWtoMFe} and Corollaries~\ref{c:RWtoMFvMFe}--\ref{c:RWtoMFvSDE} provide an explicit, $\delta,\beta$-dependent lower bound on $\e$ and $1/n$ along which solutions of $(RW_{\beta,n}^{\delta,\e})$ converge to those of $(MF_\beta^\delta)$. Together with the  result in \cite{GarroniVanMeursPeletierScardia19DOI} on the convergence of solutions of $(MF_\beta^\delta)$  to those of $(GB)$, this gives a convergence result of solutions of $(RW_{\beta,n}^{\delta,\e})$ to those of $(GB)$. Since \cite{GarroniVanMeursPeletierScardia19DOI} does not provide error estimates, we do not have a characterization of the parameter regime in which this convergence holds, and we have to settle for an implicitly given limit along which $(\e, n, \beta, \delta) \to (0, \infty, \infty, 0)$. The future challenge here is to quantify the rate as $(\beta, \delta) \to (\infty, 0)$ in $(MF_\beta^\delta)$, which would specify an asymptotic regime in $(\e, n, \beta, \delta)$ in which the corresponding solutions of $(RW_{\beta,n}^{\delta,\e})$ converge to those of $(GB)$. Incidentally, this would imply a notion of unique solutions to $(GB)$   corresponding to our atomistic model.

\paragraph{Other convergence results.}
For finitely many dislocations, in discrete systems similar to $\RW$, the literature already contains a number of convergence results in the limit of small lattice spacing. Hudson and Ortner showed that in a discrete system arbitrary collections of screw dislocations are locally stable if the lattice spacing $\e$ is sufficiently small~\cite{HudsonOrtner14,HudsonOrtner15}. For the `XY' model, Alicandro, De Luca, Garroni, and Ponsiglione showed the presence of a large number of local minimizers, and they proved convergence of an $n$-dislocation motion with `thermalized' deterministic dynamics, again  in the limit of small lattice spacing~\cite{AlicandroDeLucaGarroniPonsiglione14,AlicandroDeLucaGarroniPonsiglione16,AlicandroDeLucaGarroniPonsiglione17}. For a a discrete random walk, Hudson proved convergence to a deterministic differential equation~\cite{Hudson17}.

The time-dependent many-dislocation limit seems only to have been studied for continuous-space dislocations. When dislocations are points in two dimensions, as in this paper, the many-dislocation limit has a long history in the context of interacting stochastic and deterministic particle systems~\cite{McKean67,Dobrushin79,Oelschlager84,Spohn91,Golse16,JabinWang17}, with specific applications to dislocations in 
\cite{Duerinckx16,
ForcadelImbertMonneau09,
VanMeursMuntean14}. 
However, for dislocations with multiple Burgers vectors, as in this paper, many of the classical interacting-particle methods do not apply. Indeed,  \cite{ChapmanXiangZhu15} identifies the formation of dipoles, which are invisible in the many-dislocation limit, but may alter the many-dislocation limit. In \cite{GarroniVanMeursPeletierScardia19DOI} these computations are made rigorous in a two dimensional, dynamical setting. 
In one spatial dimension and with regular cross-interactions, \cite{vanMeurs18} explores the parts of parameter space where dipole formation is strong enough to affect the many-dislocation limit. In two dimensions and with singular interactions, Schochet devised a method for arbitrary-sign point-vortex solutions of the Euler equations with random initial positions~\cite{Schochet96}, which shares many aspects with dislocations. 
\medskip

In comparison to this earlier work we prove a number of new results: not only the estimates of Theorems~\ref{t:RWtoSDE}--\ref{t:MFetoMF} and Corollaries~\ref{c:RWtoMFvMFe}--\ref{c:RWtoMFvSDE} are new, but also each of the six convergence statements of Section~\ref{sec:convergence-results} was not proved before.

In addition, we believe that the shift from convergence results to non-asymptotic estimates is important. As systems become more complex, with higher-dimensional parameter spaces, their behaviour is more transparently characterized in terms of estimates than in terms of a multitude of scaling regimes.

%
%
%
%
%

\paragraph{Regularisation and dislocation annihilation.}
A key assumption of this study is that dislocations interact via the regularised interaction potential $V_\delta$ introduced in \S\ref{sec:disl-inter-potent}, and it is reasonable to ask what behaviour we expect in the limit where the regularisation parameter $\delta$ tends to zero.
The case where $\delta\to0$ in $(SDE^\delta_{\beta,n})$ introduced in \S\ref{s:SDE} was studied by Fournier and Jourdain in~\cite{FJ17}.
The authors show that there is a critical temperature $\beta_*^{-1}$ above which colliding particles will immediately separate again, and below which colliding particles stick together.

If instead we start from $(ODE^\delta_n)$, the interaction force $-\nabla V_\delta$ blows up as dislocations come close in the limit $\delta \to 0$. Hence, the limiting problem, denoted $(ODE_n)$, requires a collision rule in addition to the set of ODEs in which $V_\delta$ is replaced by $V$. Until the first collision time, the convergence as $\delta \to 0$ is standard ODE theory. However, defining a collision rule which is consistent with $(ODE^\delta_n)$ seems challenging. Indeed, the simulations in \cite[Chap.~9]{VanMeurs15} imply that even in one dimension, the group behaviour of the dislocation dynamics strongly depends on the choice of regularisation $V_\delta$. Also, since  $(ODE_n)$ is a model for screw dislocations, a natural collision rule for dislocations with opposite sign $b$ is that they annihilate each other (i.e., they both vanish). At present, there is no reason to believe that $(ODE^\delta_n)$ for small $\delta$ approximates $(ODE_n)$ except in the case where dislocations remain separated for all time.

In this paper we have avoided any annihilation effects, and thus the six models in the front of the diagram, $(MF_\beta)$ and $(GB)$ conserve the number of dislocations. Yet, to complete the diagram we have also included models that have to deal with collisions. For instance, a more accurate model for dislocation dynamics on the atomic scale is $(RW_{\beta,n}^\e)$ (see \cite{Hudson17}), which differs from $(RW_{\beta,n}^{\delta,\e})$ in that it replaces $V_\delta$ (based on regularising the iterations computed from linear elasticity) by the lattice Green function $\mathsf V_\e : \L \to \R$. While $\mathsf V_\e$ is not singular, a proper description of $(RW_{\beta,n}^\e)$ needs rules on annihilation and possibly creation of dislocations. Such rules are side-stepped in \cite{Hudson17} by assuming that the dislocations remain separated. Instead of starting from $(RW_{\beta,n}^\e)$ with the added complexity of annihilation and creation, we have chosen to replace it by the easier model $(RW_{\beta,n}^{\delta,\e})$; hence the curved arrow in the diagram, which indicates an uncontrolled modelling assumption.



Regarding the other problems in the back of the diagram which are mainly connected by dotted lines, any consistency between them will depend strongly on the chosen collision rule. For instance, in the case of annihilation, there can be no consistency with $(GB)$, because $(GB)$ is mass-conserving. We refer to \cite{BilerKarchMonneau10,AmbrosioMaininiSerfaty11} for the analysis of alternative models to $(GB)$ which include annihilation. Recently, \cite{vanMeursMorandotti19} established a first result on a connection of such models with a particle system which includes annihilation; a more complete result is in preparation.

\subsubsection{Conclusion and outlook}
We return to the main question posed in the introduction: How can one understand the domains of validity of different models describing the same physical phenomenon? 
The results of this paper build on the point of view that such understanding requires quantitative comparison at fixed parameter points. Limit theorems are not sufficient.  

Theorems~\ref{t:RWtoSDE}--\ref{t:RWtoMFe} and Corollaries~\ref{c:RWtoMFvMFe}--\ref{c:RWtoMFvSDE} provide such quantitative comparisons. In particular, they connect precisely an atomistic model to the fundamental dislocation density model in \cite{GromaBalogh99}. At this stage the set of systems that we can connect by quantitative estimates is relatively small, the restrictions on the parameters are severe, and the estimates themselves are clearly sub-optimal. On the other hand, the estimates hold for any parameter point and give a unifying view on the parameter space, which is a definite improvement over the proof of separate asymptotic limits. 

\medskip
Looking forward, even within the restriction to straight and parallel dislocations, the systems of this paper lack some natural physical features. For instance, annihilation of dislocations with opposite Burgers vectors would be an obvious extension. Progress in this direction could start from the atomistic model $\RW$ with $\delta = \e$ equipped with an annihilation rule, where the expected continuum dislocation density model is the system studied in~\cite{BilerKarchMonneau10,AmbrosioMaininiSerfaty11}. Success in this mathematically challenging direction would yield results connecting the models in the rear plane shown in Figure \ref{fig:overview}, and would link to yet further models studied in the literature.

One further physical feature lacking from our analysis is reflected by the upper bound on $\beta$ in Assumption \ref{ass:beta-e-delta-bound}.
Our error estimates show that this bound results in a linear relation between the velocity and the gradient of the energy in the asymptotic regime where $\e \gg 1$. Despite the common use of this linear relation in the literature, it is often criticised as being a crude modelling assumption. Despite this, there is currently no consensus on any improved, nonlinear dislocation dynamics model. We believe that a route which could lead to consensus is the extension of our approach to larger values of $\beta$, which would much more strongly quantify the rigorous justification of other nonlinear models, for example those established in \cite{Hudson17}. 

\smallskip

The choice to only consider straight and parallel screw dislocations also is a severe restriction. Looking beyond this study, it is reasonable to ask to what extent our results hold when considering more general curved dislocations. Naturally, both the technical mathematical and physical challenge of modelling this case is much more significant, so it is difficult to speculate on the validity of any extension of our results to this case, as neither an appropriate microscopic model nor the technical mathematical machinery are currently available to study this case at present. This is a situation we hope to remedy in future work.

\subsection{Outline}
  \label{sec:outline}
The proofs of Theorems~\ref{t:RWtoSDE}, \ref{t:MFetoMF}, \ref{t:SDEtoMF}, and \ref{t:RWtoMFe} are given in Sections~\ref{sec:RWtoSDE}, \ref{sec:MFetoMF}, \ref{s:SDEn:FP-Mark}, and \ref{sec:RWtoMFe};  Corollaries~\ref{c:RWtoMFvMFe} and \ref{c:RWtoMFvSDE} are proved in Section~\ref{sec:RWtoMF}. Each of these sections provides an overview of the proof as a series of auxiliary results; in each case, after the completion of the main argument, the auxiliary results are then proved subsequently within the same section.

\section{Proof of Theorem~\ref{t:RWtoSDE}}
\label{sec:RWtoSDE}
In this section we prove Theorem \ref{t:RWtoSDE}, which provides an estimate on the distance between the laws of $\RW$ and $\SDE$; we recall the statement here for convenience.

\thmRWSDE*

The main technique used here is to consider the Fokker--Planck equations which describe the evolution of the laws for the two processes, and show that the law of the random walk forms an approximation in space of the law for the SDE. This approach is suggested by the Lax-Richtmyer equivalence theorem (sometimes called the Fundamental Theorem of Numerical Analysis), which leads us to prove consistency and stability results concerning the approximation, and from these we deduce a bound by Gronwall's Lemma.

\subsection{Main argument}
As mentioned above, the key to the proof of this theorem is a treatment of the relevant Fokker-Planck equations $\FPe$ and $\FP$, which are given by 
\begin{align} \label{FPe}
  \FPe \qquad 
  \partial_t \bmu_\e
  &= \Omega_\e^* \bmu_\e,
  &\Omega_\e^* \bmu_\e
  &:= \e\!\sum_{\bh \in \Nhd^n_\e} D_{-\bh} \big( \mathcal{R}^\e_{n,\bh}\bmu_\e \big) \\
    &&&= \frac1{\beta\e} \sum_{\bh \in \Nhd^n_\e} 
       D_{-\bh} \Big( \bmu_\e\exp\big(\tfrac12\beta\,\bh\cdot \bF\big) \Big) 
  \label{FPe-F}\\[2\jot]
  \label{FP}
  \FP \qquad 
  \partial_t \bmu
  &= \Omega^* \bmu,
  &\Omega^* \bmu &:= -\div ( \rho \bF ) + \frac{1}{\beta} \Delta \bmu.
\end{align}
Both of these equations are to be understood in a classical sense,
and we note that the laws of $\RW$ and $\SDE$ satisfy the above equations as a direct
consequence of the general theory of Markov processes; see for example Chapter~2 of \cite{Norris97}
and Chapter~2 of \cite{Pavliotis14}.

We note in particular that $\FP$ is a linear parabolic equation on the smooth compact manifold
$\T^{2n}$, and so $\bmu$ is the unique classical solution of this equation for a given initial
condition $\bmu^\circ\in\mathcal{P}(\T^{2n})$. Since $V_\delta$ is assumed to be of class {$C^5$} for each $\delta>0$, the density
with respect to the underlying volume measure $\frac{\dd \bmu(t)}{\dd \bnu}$
is therefore at least $C^{4,1}$ on $(0,T] \times \T^{2n}$ (see for example Theorem~5.3 in \cite{LSU1968}). Applying Theorem~5.1 in \cite{JKO98} for any
$T > 0$ and any initial condition $\bmu^\circ$ which is absolutely continuous with respect to
$\bnu$, we obtain that the smooth solution $\bmu$ satisfies $\frac{\dd \bmu(t)}{\dd \bnu} \to \frac{\dd \bmu^\circ}{\dd \bnu}$ in
$L^1 (\bnu)$ as $t \to 0$.

With the regularity properties of $\bmu$ clarified, our approach can be viewed as an application of the Fundamental Theorem of Numerical Analysis, treating $\FPe$ as a discretisation of $\FP$. In keeping with this approach, we therefore establish appropriate consistency and stability results, encoded in the following two lemmas. To state them, we set some further notation. For $\bx \in \T^{2n}$ and for $\bF:\T^{2n}\to\R^{2n}$ as defined in \eqref{F}, we denote the components by $x_{i,j}$ and $F_{i,j}$, $i=1,\dots,n$, $j=1,2$, respectively. We set $\partial_{i,j}$ as the partial derivative with respect to $x_{i,j}$. 

\begin{lem}[Consistency] \label{l:consis}
Let $\beff \in C^4 (\T^{2n})$, and let $\Omega_\e^*$ and $\Omega^*$ be given by~\eqref{FPe-F} and~\eqref{FP}. Under the assumptions of Theorem~\ref{t:RWtoSDE}, it follows that
 \begin{equation}\label{eq:consis+Rm}
   \begin{gathered}
  \max_{\bell \in \L^n} \big| (\Omega^*_\e - \Omega^*) \beff (\bl) \big| 
  \leq \frac{C\e^2n}{\beta} \sum_{m=0}^4 \beta^m R_m[\beff] \\\text{where}\quad
  R_m[\beff] := \max_{\substack{i= 1,\ldots,n \\ j =1,2}} \big\|\partial_{i,j}^{4-m} \big[\big( F_{i,j}\big)^m \beff \big] \big\|_\infty,
\end{gathered}
\end{equation}
where $C>0$ is a constant independent of $\beta, n$ and $\beff$.
\end{lem}\medskip

\begin{lem}[Stability] \label{l:staby} 
 Let $\beff \in L^2( \bnu_\e )$ and let $\Omega_\e^*$ be given by~\eqref{FPe-F}. Then, under the assumption of Theorem~\ref{t:RWtoSDE}, we have
\[
 \big(\Omega^*_\e \beff ,\beff\big)_{L^2(\bnu_\e)} 
 \leq
 C \big( n\| d \bF \|_\infty + \beta n \| \bF \|_\infty^2 \big)
 \|\beff \|_{L^2(\bnu_\e)}^2,
\]
where $C>0$ is a constant independent of $\beta, n$ and $\beff$.
\end{lem}

We remark that neither of the two results above rely heavily on the precise nature of $\bF$ beyond its regularity, and we will exploit this in the subsequent proof of Theorem~\ref{t:RWtoMFe}.
The final auxiliary result we require in order to complete our proof provides control of derivatives of the density of the solution to \eqref{FP} in the supremum norm. The result stated here is a corollary of a lemma (Lemma \ref{l:rho:est}) which provides a sharper estimate.

\begin{restatable}{cor}{cordfbnd}
  \label{c:rho:est}
  Suppose that the assumptions of Theorem~\ref{t:RWtoSDE} hold, and further assume that
  \[
    \gamma := \frac{16 \mathsf C_V^2  \beta n^2}{\delta^2}\geq \frac{1}{\delta^2}.
  \]
  Suppose that $\bmu$ solves \eqref{FP}  with initial datum $\bmu^\circ$, and set $\beff := \frac{d\bmu}{d\bnu}$ and $\beff^\circ:= \frac{d\bmu^\circ}{d\bnu}$.
Then for each $k = 0,\dots,4$ there exists a universal constant $C_k$ independent of $\beta$, $n$, $\beff$ and $t$ such that
\[
\| \beff(t) \|_{k,\infty}
      \leq C_k \|\beff^\circ\|_{k,\infty} \mathrm{e}^{ \frac32 (k+1) \gamma T }.
\]
\end{restatable}



With these results in place, we may complete the proof of Theorem~\ref{t:RWtoSDE}.\smallskip

\noindent
\emph{Bounding the error terms $R_m[\beff]$.}
We first seek to bound the error terms in the consistency estimate \eqref{eq:consis+Rm}. For convenience, we will define the sum of these terms to be
\begin{equation}\label{KDefs}
    \begin{gathered}
    K_1(\beff) := \sum_{m=0}^4 \beta^{m-1} R_m[\beff].
  \end{gathered}
\end{equation}
Now, recalling the definition of $F_i$ given in \eqref{Fi}, it can be checked that for all
$k,m=0,\ldots,4$ it holds that
\begin{equation*} 
  \big|\partial_{i,j}^k\big(F_{i,j}\big)^m(\bx)\big|=\bigg|\partial^k_{i,j}\bigg(\frac{1}{n}\sum_{l=1}^nb_ib_l\partial_j V_\delta(x_i-x_l)\bigg)^m\bigg|
  \lesssim \delta^{-m-k}. 
\end{equation*}
To see this, we can view this as applying $k$ derivatives to a product of $m$ factors, where each factor is at worst $\delta^{-1-a}$ with $a$ being the number of derivatives that applies to that factor. Recalling the definition of $R_m$ given in \eqref{eq:consis+Rm} and
using the product rule, this entails that
\begin{align*}
  \big|R_m[\beff]\big|
  = \max_{\substack{i= 1,\ldots,n \\ j =1,2}} \bigg\|\sum_{k=0}^{4-m}\bin{4-m}{k}
  \partial_{i,j}^{4-m-k} \big( F_{i,j}\big)^m  \partial_{i,j}^k\beff \bigg\|_\infty 
  \lesssim  \sum_{k=0}^{4-m} \delta^{k-4} \| d^k \beff \|_\infty
  \lesssim \delta^{-4} \| \beff \|_{4,\infty}.
\end{align*}
Note that this estimate holds for an arbitrary smooth function $\beff$. Now, taking $\beff$ to be the solution of \eqref{FP}, applying this estimate to the definition of $K_1$ and then using Corollary~\ref{c:rho:est}, we obtain
\begin{align*}
  K_1\big(\beff(t)\big)
  \lesssim \beta^{3} \delta^{-4} \| \beff(t) \|_{4,\infty}
  \lesssim \gamma^3 \| \beff^\circ \|_{4,\infty} \mathrm{e}^{\frac{15}2  \gamma T}
\end{align*}
Now, from the given polynomial bound on $\|\beff^\circ\|_{4,\infty}$ in terms of $\beta, n, \delta^{-1}$, we have $\|\beff^\circ\|_{4,\infty} \lesssim \gamma^p$ for some $p \geq 0$. Then, from the elementary inequality $\gamma^{p + 3} \lesssim \mathrm{e}^{\gamma / 2} \leq \mathrm{e}^{\gamma T / 2} $, we finally obtain
\begin{equation}\label{eq:K1Est2}
  K_1(\beff(t)) 
  \lesssim \mathrm{e}^{8 \gamma T}.
\end{equation}

\noindent 
\emph{Applying Gronwall's Lemma}.
  Our aim is now to apply Gronwall's Lemma to conclude the proof. In terms of $\beff_\e$ and $\beff$,
  \eqref{FPe} and \eqref{FP} become
  \begin{equation*}
    \partial_t\beff_\e = \Omega^*_\e\beff_\e\quad \text{and}\quad \partial_t\beff = \Omega^*\beff.
  \end{equation*}
  Restricting the latter equation to $\L^n$
  and considering the difference between the resulting equations, we add and subtract
  $\Omega^*_\e\beff$ to find
  \begin{equation}\label{FPDiff} 
    \partial_t\big(\beff - \beff_\e\big) 
    = \Omega^*\beff - \Omega_\e^* \beff_\e
    = (\Omega^* - \Omega_\e^*)\beff + \Omega_\e^* (\beff-\beff_\e).
  \end{equation}
  Defining $v_\e:=\big\|\beff-\beff_\e\big\|_{L^2(\bnu_\e)}$,
  we multiply \eqref{FPDiff} by $\beff - \beff_\e$ and then integrate with respect to $\bnu_\e$,
  giving
  \begin{equation}\label{eq:ell2Err}
    \begin{aligned}
  \frac{1}{2} \frac{\mathrm{d}}{\mathrm{d}t} v_\e^2 
  &= \frac{1}{2}\frac{\mathrm{d}}{\mathrm{d}t}\|\beff- \beff_\e\|^2_{L^2(\bnu_\e)}\\
  &= \Big(\Omega^*\beff-\Omega_\e^*\beff, \beff - \beff_\e\Big)_{L^2(\bnu_\e)}+
  \Big(\Omega_\e^* (\beff - \beff_\e),\beff - \beff_\e\Big)_{L^2(\bnu_\e)}.
\end{aligned}
    \end{equation}
    For the former term, we apply H\"older's inequality, Lemma \ref{l:consis}, the definition of $K_1(\beff)$ given in \eqref{KDefs}, and the triangle inequality to find
\begin{equation*} 
  \big( (\Omega^* -\Omega^*_\e)\beff, \beff - \beff_\e\big)_{L^2(\bnu_\e)}
  \lesssim n \e^2 \, K_1(\beff) \, v_\e.
\end{equation*}
We estimate the latter term in \eqref{eq:ell2Err} directly using Lemma \ref{l:staby}. For the prefactor in the estimate in Lemma \ref{l:staby}, we use Assumption~\ref{ass:Vd} and Assumption~\ref{ass:beta-e-delta-bound} to simplify it to
\begin{equation*}
  n \bigg( {\max_{\substack{i=1,\dots,n\\j=1,2}} \| \partial_{i,j}F_{i,j} \|_\infty} + \beta\, \max_{\substack{i=1,\dots,n\\j=1,2}} \| F_{i,j} \|_\infty^2 \bigg) 
  \lesssim n \delta^{-2} (1 + \beta)
  \lesssim n \beta \delta^{-2}.
\end{equation*}

Plugging these estimates into \eqref{eq:ell2Err} and then applying Young's inequality, we obtain
\begin{equation*} 
  \frac{\mathrm{d}}{\mathrm{d}t} v_\e^2
  \lesssim C' n \e^2 \, K_1(\beff) \, v_\e +
    C n \beta \delta^{-2} v_\e^2 
  \leq C' n \e^4 \beta^{-1} \delta^2 K_1(\beff)^2 +
    C n \beta \delta^{-2} v_\e^2.
\end{equation*}
Now, via a standard application of Gronwall's Lemma, we obtain
\begin{equation}\label{eq:TimeEst}
  v_\e^2 (t)
  \leq \bigg( v_\e^2 (0) + C' n \e^4 \beta^{-1} \delta^2 \int_0^t K_1\big(\beff(s)\big)^2 \, \mathrm{e}^{-C n \beta \delta^{-2} s} \mathrm{d}s \bigg)
  \mathrm{e}^{C n \beta \delta^{-2} t}.
\end{equation}
Finally, we insert estimate \eqref{eq:K1Est2} into \eqref{eq:TimeEst}, giving
\begin{align*}
  v_\e^2(t)
  &\leq \bigg( v_\e^2(0) + C' n \e^4 \beta^{-1} \delta^{2} \mathrm{e}^{16 \gamma T} \int_0^t \mathrm{e}^{-C n \beta \delta^{-2} s} \mathrm{d}s \bigg)
       \mathrm{e}^{C n \beta \delta^{-2} t} \\
  &\leq \bigg( v_\e^2(0) + C' \e^4 \beta^{-2} \delta^{4} \mathrm{e}^{16 \gamma T} \bigg)
       \mathrm{e}^{C n \beta \delta^{-2} t} \\
  &\leq v_\e^2(0) \mathrm{e}^{C n \beta \delta^{-2} t} + C' \e^4 \mathrm{e}^{C'' n^2 \beta \delta^{-2} T}. 
  \end{align*} 
  Taking a square root and using that $\sqrt{a+b}\leq \sqrt{a}+\sqrt{b}$ for $a,b\geq0$,
  this therefore completes the proof of Theorem~\ref{t:RWtoSDE}.

\subsection{Proofs of auxiliary results}
\label{sec:cons-stab-proofs}

In this section we provide the proofs of the various auxiliary results stated in the previous section.

\subsubsection{Proof of Lemma~\ref{l:consis}}
  First, recalling the definition $\mathcal{R}^\e_{n,\bh}$ given in \eqref{RW:Rates}, we Taylor-expand the exponential factor to obtain
  \begin{equation}\label{eq:Rate_Expansion}
    \mathcal{R}^\e_{n,\bh} (\bl) 
    = \frac1{\beta\e^2 }\sum_{m=0}^3\frac{1}{m!}\Big({\frac12}\beta\,\bh \cdot \bF (\bl)\Big)^m
    +{\mathcal O\bigg(\frac{\beta^3}{\e^2}\big|\bh\cdot \bF(\bl)\big|^4\bigg)}.
  \end{equation}
  Here, we have used that the variable in which we Taylor-expand is bounded; indeed, using Assumptions \ref{ass:Vd} and  \ref{ass:beta-e-delta-bound}, we find
  \begin{equation*}
    \Big|{\frac12}\beta\,\bh \cdot \bF (\bl)\Big|
    \leq C \beta \e \| \nabla V_\delta \|_{\infty}
    \leq C \beta \frac\e\delta
    \leq C.
  \end{equation*}
    
  Introducing linear operators $L_{\e,m}$ for $m=0,1,2,3$ which are defined by
\begin{equation}\label{Om:split}
 L_{\e,m} \beff (\bl) 
 := \frac{1}{\beta\e}\sum_{\bh \in \Nhd_\e^n}\frac{1}{m!}D_{-\bh}\Big(\big({\tfrac12}\beta\,\bh\cdot\bF\big)^m\beff\Big) (\bl),
\end{equation}
  we expand to find
\begin{align}
    \Omega_\e^* \beff (\bl)
    &= \e\!\sum_{\bh \in \Nhd^n_\e} D_{-\bh} \big( \mathcal{R}^\e_{n,\bh} \beff \big) (\bl)\notag\\
    &= \sum_{m=0}^3 L_{\e,m} \beff (\bl) 
    + {
      \mathcal O\bigg(n\beta^3\e^2 \max_{\substack{i=1,\dots,n\\j=1,2}} \sup_{\bx \in B(\bell, \e)}\big|\big(F_{i,j}\big)^4(\bx)\big||\beff(\bx)|\bigg),}\notag\\
    &= \sum_{m=0}^3 L_{\e,m} \beff (\bl) 
    + {
      \mathcal O\bigg(n\beta^3\e^2R_4[\beff](\bl)\bigg),}
  \label{eq:Ome_remainder}\end{align}
To continue the expansion of $\Omega_\e^* \beff$, we note that for any smooth function $\boldsymbol{g}:\T^{2n}\to\R$
\begin{equation*} 
  D_{-\bh}\boldsymbol{g}(\bl) 
  = \frac{1}{\e} \sum_{k=1}^{K-1} \frac{1}{k!} (-\bh\cdot\nabla)^k \boldsymbol{g}(\bl) + \mathcal{O} \Big( \frac1\e \sup_{\theta \in[0,1]} \big| (\bh\cdot\nabla)^K \boldsymbol{g}(\bl+\theta \bh) \big| \Big).
\end{equation*}
We note in particular that
\begin{equation*}
  \begin{aligned}
    L_{\e,0}\beff(\bl)
    &= \frac{1}{\beta\e}\sum_{\bh\in\Nhd_\e^n}D_{-\bh}\beff(\bl)\\
    &= \frac{1}{\beta\e^2} \sum_{k=1}^3 \frac{1}{k!} \sum_{\bh\in\Nhd^\e_n}(-\bh\cdot\nabla)^k \beff(\bl)
    + \mathcal{O} \Big( \frac{\e^2n}{\beta} \sup_{\substack{i=1,\ldots,n\\j=1,2}}\sup_{\theta \in[0,1]}
    \big| \partial^4_{i,j} \beff(\bl+\theta \bh) \big| \Big).
  \end{aligned}
\end{equation*}
Since $\Nhd^\e_n$ is symmetric under inversion, i.e. $\bh\in\Nhd^\e_n$ implies $-\bh\in\Nhd^\e_n$,
the terms in the first sum corresponding to odd values of $k$ vanish, and hence using the
definition of $R_0$ given in \eqref{eq:consis+Rm}, we have
\begin{equation}\label{eq:L0e_expansion}
  L_{\e,0}\beff(\bl)
  = \frac{1}{\beta} \Delta \beff(\bl)
  + \mathcal{O} \bigg( \frac {\e^2n}{\beta} R_0[\beff](\bl) \bigg).
\end{equation}
Similarly, for $L_{\e,1}\beff$, we find
\begin{equation*} 
  \begin{aligned}
    L_{\e,1}\beff(\bl)
    &= \frac{1}{\beta\e}\sum_{\bh\in\Nhd_\e^n}D_{-\bh}\big(\tfrac12\beta\bh\cdot\bF\,\beff\big)(\bl)\\
    &= \frac{1}{2\e^2} \sum_{k=1}^2 \frac{1}{k!} \sum_{\bh\in\Nhd^\e_n}(-\bh\cdot\nabla)^k(\bh\cdot\bF\,\beff)(\bl) \\
    &\qquad + \mathcal{O} \Big( n\e^2 \sup_{\substack{i=1,\ldots,n\\j=1,2}}\sup_{\theta \in[0,1]}
    \big| \partial^3_{i,j} (F_{i,j}\beff)(\bl+\theta \bh) \big| \Big).
  \end{aligned}
\end{equation*}
Again using the inversion symmetry of $\Nhd^\e_n$, we find that the $k=2$ term vanishes in the
first sum, which yields
\begin{equation}\label{eq:L1e_expansion}
  L_{\e,1}\beff(\bl)
    = -\div(\bF\,\beff)(\bl)
    + \mathcal{O} \Big( n\e^2  R_1[\beff] \Big).
\end{equation}
Similar arguments then show that
\begin{equation}\label{eq:L2e_expansion}
  L_{\e,2} \beff (\bl)
  =\mathcal{O}\Big(n\beta\e^2 R_2[\beff](\bl) \Big)\quad\text{and}\quad
  L_{\e,3} \beff (\bl)
  =\mathcal{O}\Big(n\beta^2\e^2 R_3[\beff](\bl)\Big).
\end{equation}
Combining the expansions \eqref{eq:L0e_expansion}, \eqref{eq:L1e_expansion} and
\eqref{eq:L2e_expansion} with \eqref{eq:Ome_remainder}, we obtain the result.


\subsubsection{Proof of Lemma~\ref{l:staby}}
  Using the operators $L_{\e,m}$ as defined in \eqref{Om:split}, we expand
  \begin{equation} \label{Omea:expa}
    \!\big(\Omega^*_\e \beff ,\beff\big)_{L^2(\bnu_\e)}
    = \underbrace{\big(L_{\e,0} \beff , \beff\big)_{L^2(\bnu_\e)}}_{=:T_1}
      + \underbrace{\big(L_{\e,1} \beff , \beff\big)_{L^2(\bnu_\e)}}_{=:T_2}
      + \underbrace{\big(\Omega^*_\e\beff - L_{\e,0}\beff-L_{\e,1}\beff,\beff\big)_{L^2(\bnu_\e)}}_{=:T_3}.
  \end{equation}
  In the remainder of the proof, we bound all three terms in the right-hand side separately. \medskip
  
\noindent{\it Estimate on $T_1$.}
  For the first term in \eqref{Omea:expa}, using
  translation invariance, we have that $\sum_{\bl \in \L^n} \boldsymbol{g}(\bl) = \sum_{\bl \in \L^n} \boldsymbol{g}(\bl + \boldsymbol{m})$ for any function $\boldsymbol{g}:\L\to\R$ and any
  translation $\boldsymbol{m} \in \e\Z^{2n}$. Using this fact and the symmetry of $\Nhd^n_\e$
  enables us to `sum by parts' to obtain
    \begin{align}
  T_1&= \frac{\e^{2n-1}}{\beta} \sum_{\bl\in\L^n}\sum_{\bh\in\Nhd^n_\e} \beff(\bl) D_{-\bh}\beff (\bl ) \notag\\
      &= \frac{\e^{2n-1}}{\beta} \bigg(\frac{1}{2}\sum_{\bl\in\L^n}\sum_{\bh\in\Nhd^n_\e} \beff(\bl)  D_{-\bh}\beff (\bl)  + \frac{1}{2}\sum_{\bl\in\L^n}\sum_{\bh\in\Nhd^n_\e} \beff(\bl+\bh) D_{-\bh}\beff (\bl+\bh) \bigg) \notag\\
      &= \frac{\e^{2n-1}}{\beta} \bigg(\frac{1}{2}\sum_{\bl\in\L^n}\sum_{\bh\in\Nhd^n_\e} \beff(\bl)  D_{\bh}\beff (\bl)  - \frac{1}{2}\sum_{\bl\in\L^n}\sum_{\bh\in\Nhd^n_\e} \beff(\bl+\bh) D_{\bh}\beff (\bl) \bigg) \notag\\
  &=-\frac{\e^{2n}}{2\beta} \sum_{\bl\in\L^n}\sum_{\bh\in\Nhd^n_\e} \big|D_{-\bh}\beff(\bl)\big|^2
  \leq 0.\label{eq:T1_stab}
\end{align}
  
\noindent{\it Estimate on $T_2$.}
Next we bound the second term in \eqref{Omea:expa}:
\begin{align*}
  T_2
  &= \frac{\e^{2n-1}}2 \sum_{\bl \in \L^n} \sum_{\bh \in \Nhd^n_\e} \beff(\bl) D_{-\bh}\big((\bh \cdot \bF) \, \beff\big) (\bl) \\
  &= \frac{\e^{2n-2}}2 \sum_{\bl\in\L^n}\sum_{\bh\in\Nhd^n_\e} \Big(\bh \cdot \bF (\bl - \bh) \, \beff (\bl - \bh) \, \beff(\bl) - \bh \cdot \bF (\bl) \, \beff(\bl)^2\Big).
\end{align*}
Employing the inversion symmetry of $\Nhd^n_\e$ as used in the proof of Lemma~\ref{l:consis},
the sum of the second terms in the summand vanishes. To treat the first terms in the summand,
we apply translation invariance and symmetry of $\Nhd^n_\e$, and then the Cauchy--Schwarz inequality to obtain
\begin{align}
  T_2
  &= \frac{\e^{2n-2}}{4} \sum_{\bl\in\L^n}\sum_{\bh\in\Nhd^n_\e} \bh \cdot \bF (\bl - \bh) \, \beff (\bl - \bh)
    \,\beff(\bl) -\bh \cdot \bF (\bl) \, \beff (\bl) \, \beff(\bl-\bh) \notag\\
  &= \frac{\e^{2n-1}}{4} \sum_{\bl\in\L^n}\sum_{\bh\in\Nhd^n_\e} D_{-\bh}(\bh \cdot \bF)
    (\bl)\,\beff(\bl) \,\beff (\bl - \bh) \notag\\
  &\leq n  {\max_{\substack{i=1,\dots,n\\j=1,2}} \| \partial_{i,j}F_{i,j} \|_\infty} \,\|\beff\|_{L^2(\bnu_\e)}^2.\label{eq:T2_stab}
\end{align}

\noindent{\it Estimate on $T_3$.}
Finally, we bound the third term in \eqref{Omea:expa}. Defining $r_\bh(\bl) = \beta\e^2\,\mathcal{R}^\e_{n,\bh}(\bl)-1-{\frac12}\beta\,\bh\cdot\bF(\bl)$, we observe that
\begin{align*}
  T_3
  &= \bigg( \frac1{\beta\e} \sum_{\bh\in\Nhd^n_\e}
    D_{-\bh} \Big[ \Big( \beta\e^2\,\mathcal{R}^\e_{n,\bh}(\bl)-1
    -{\tfrac12}\beta\,\bh\cdot\bF(\bl) \Big) \beff \Big] , \beff \bigg)_{L^2(\bnu_\e)} \\
  &= \frac{\e^{2n-1}}{\beta}\sum_{\bl\in\L^n}\sum_{\bh\in\Nhd^n_\e}D_{-\bh}(r_\bh\beff)(\bl)\beff(\bl).
\end{align*}
Using a similar Taylor series expansion as in \eqref{eq:Rate_Expansion}, we find that
\begin{equation*}
  \max_{\bl \in \L^n} | r_\bh(\bl) | 
  = {\mathcal O \Big( \beta^2\,\max_{\bl \in \L^n} | \bh\cdot\bF(\bl) |^2 \Big)=
  \mathcal O \Big( \beta^2\e^2\,\max_{\substack{i=1,\dots,n\\j=1,2}} \| F_{i,j} \|^2 \Big)}. 
\end{equation*}
Then, a similar argument to that made in \eqref{eq:T2_stab} yields
\begin{equation}\label{eq:T3_stab}
  T_3
  \lesssim \beta n\,\|\beff\|_{L^2(\bnu_\e)}^2 \max_{\substack{i=1,\dots,n\\j=1,2}} \| F_{i,j} \|^2.
\end{equation}
Combining the estimates on $T_1$, $T_2$ and $T_3$ made in \eqref{eq:T1_stab}, \eqref{eq:T2_stab} and \eqref{eq:T3_stab}, we conclude that the statement holds.

\subsubsection{Proof of Corollary~\ref{c:rho:est}}
The final result required in order to establish Theorem~\ref{t:RWtoSDE} is Corollary~\ref{c:rho:est}, which is a simplification of the more precise estimate given in the following lemma.

\begin{lem} \label{l:rho:est}
Suppose that $\bmu$ solves \eqref{FP}  with initial datum $\bmu^\circ$, and set $\beff := \frac{d\bmu}{d\bnu}$.
  Under the assumptions of Theorem~\ref{t:RWtoSDE}, for any $k=0,\dots,4$, the derivative $d^k \beff(s,\cdot)$ is continuous and bounded on $[0,T]\times \T^{2n}$, and there exist universal constants $C_\ell$ and $C_{\ell,m}$ 
such that
\begin{equation} \label{pf:est:pkijft}
      \big\| \beff(t)\big\|_{k,\infty}
      \leq \sum_{\ell=0}^k C_\ell b_\ell \big\| \beff^\circ\big\|_{k-\ell, \infty }
\end{equation}
where $b_\ell$ satisfies the recursion relation
\begin{equation}\label{e:bl:defn}
  b_0 = \mathrm{e}^{\gamma t}, \qquad 
  b_\ell = \sqrt{\frac{\gamma t}{\pi}} \mathrm{e}^{\gamma t} \sum_{m=0}^{\ell-1} C_{\ell,m} \delta^{m-\ell-1} b_m,
\end{equation}
where $\gamma := 16 \mathsf C_V^2 \beta n^2 \delta^{-2}.$
\end{lem}

In order to establish this result, we will require two further general estimates, which we state next. 

\begin{restatable}{lem}{GradPhiBeta}
  \label{l:GradPhiBeta}
  Let $\Phi_\beta\in \mathcal D'\big( \T^{2n} \times (0, T) \big)$ be the heat
  kernel on the torus, satisfying
  \begin{equation}\label{eq:TorusHeatKernel}
    \partial_t\Phi_\beta-\beta^{-1}\Delta \Phi_\beta = 0\quad\text{and}\quad
    \Phi_\beta |_{t=0} = \delta_0,
  \end{equation}
  in the sense of distributions. Then 
  \begin{equation*}
    \|\Phi_\beta(\cdot,t)\|_{L^1(\bnu)}=1\quad\text{and}\quad
    \max_{ \substack{ i = 1,\dots,n \\ j = 1,2} } \|\partial_{i,j}\Phi_\beta(\cdot,t)\|_{L^1(\bnu)}\leq \sqrt{ \frac{4\beta}{t\pi} }\quad
    \text{for all }t>0.
  \end{equation*} 
\end{restatable}

\begin{lem} \label{l:Gron:frac}
  Let $C, T > 0$. Let $u, g \in L^1(0,T)$ be non-negative, with $g$ non-decreasing. If 
  \begin{equation*}
    u(t) \leq g(t) + C \int_0^t \frac{u(s)}{ \sqrt{t-s} } \, ds
    \quad \text{for a.e.~$t \in (0,T)$}, 
  \end{equation*}
  then
    $u(t) \leq 2 g(t) \exp (C^2 \pi t)$ for a.e.~$t \in (0,T)$.
\end{lem} 

Lemma~\ref{l:Gron:frac} is a specific form of Gronwall's Lemma for fractional derivatives; the proof is a straightforward application of \cite[Cor.~2]{YeGaoDing07} with a standard estimate on the Mittag-Leffler function, and so we omit it.

\begin{proof}[Proof of Lemma~\ref{l:GradPhiBeta}]
  We begin by noting that the heat kernel on $\R^{2n}$, denoted
  $\Psi_\beta\in \mathcal D'\big( [0, T] \times \R^{2n} \big)$, satisfies the equation
  \begin{equation*} 
    \partial_t\Psi_\beta-\beta^{-1}\Delta \Psi_\beta = 0,\quad\text{and}\quad
    \Psi_\beta \Big|_{t=0} = \delta_0,
  \end{equation*}
  and has the expression
  \begin{equation*} 
    \Psi_\beta(t,\bx) := \Big(\frac\beta{4\pi t}\Big)^{n}\mathrm{e}^{-\beta|\bx|_2^2/4t},
  \end{equation*}
  where $|\cdot|_2$ denotes the Euclidean norm.
  Identifying $\bx\in \T^{2n}$ with a point in $Q^n = [-\frac12, \frac12)^{2n}$, It follows that we may express
  $\Phi_\beta$ by summing over translates of $\Psi_\beta$, i.e.
  \begin{equation*} 
    \Phi_\beta(t,\bx) = \sum_{\boldsymbol{m}\in\Z^{2n}}\Psi_\beta(t,\bx-\boldsymbol{m}).
  \end{equation*}
  The decay of $\Psi_\beta$ ensures that this sum is well--defined, and converges absolutely
  for all $(t,\bx)\in(0,T]\times\T^{2n}$. As a consequence, it is clear that $\Phi_\beta$ is periodic, and
  satisfies \eqref{eq:TorusHeatKernel}.
  
  To prove the first results, we note that $\Phi_\beta$ is positive, so 
  \begin{equation*}
    \|\Phi_\beta(t)\|_{L^1(\bnu)} = \int_{Q^n} \sum_{m\in\Z^{2n}}\Psi_\beta(t,\bx-\boldsymbol{m}) d\bx =
    \int_{\R^{2n}}\Psi_\beta(t,\bx) d\bx = 1.
  \end{equation*}
  The second result follows by noting that
  \begin{align*}
      \|\partial_{i,j} \Phi_\beta(t,\cdot)\|_{L^1(\bnu)}
      &= \int_{Q^n} \left|\sum_{\boldsymbol{m}\in\Z^{2n}}\partial_{i,j} \Psi_\beta(t,\bx-\boldsymbol{m})\right| \, d\bx \\
      &\leq\sum_{\boldsymbol{m}\in\Z^{2n}} \int_{Q^n} \left|\partial_{i,j} \Psi_\beta(t,\bx-\boldsymbol{m})\right|  \, d\bx
      =\left\|\partial_{i,j} \Psi_\beta(\cdot,t)\right\|_{L^1(\R^{2n})}.
  \end{align*} 
  The result now follows by transforming to cylindrical coordinates and integrating:
    \begin{equation*}
      \begin{aligned}
        \left\|\partial_{i,j} \Psi_\beta(\cdot,t)\right\|_{L^1(\R^{2n})}
        &=2\pi\left(\frac\beta{4\pi t}\right)^{n+1}\int_{\R^{2n}}|x_{i,j}| \mathrm{e}^{-\beta|\bx|^2_2/4t}d\bx\\
        &=2\pi\left(\frac\beta{4\pi t}\right)^{n+1}
        \int_{-\infty}^\infty|x_{i,j}| \mathrm{e}^{-\beta x^2_{i,j}/4t}dx_{i,j}
        \int_{\R^{2n-1}}\mathrm{e}^{-\beta |\boldsymbol{y}|^2_2/4t}d\boldsymbol{y}\\
        &=4\pi\left(\frac\beta{4\pi t}\right)^{n+1}\int_0^\infty r\mathrm{e}^{-\beta r^2/4t}dr
        \int_{\R^{2n-1}}\mathrm{e}^{-\beta|\boldsymbol{y}|^2_2/4t}d\boldsymbol{y}\\
        &= 2\left(\frac\beta{4\pi t}\right)^{n}
        \big|\mathbb{S}^{2n-2}\big|\int_0^\infty R^{2n-2}\mathrm{e}^{-\beta R^2/4t}dR\\
        &=(4n-2)\frac{\Gamma(n-\tfrac12)}{\Gamma(n+\tfrac12)}
        \left(\frac\beta{4\pi t}\right)^{1/2}=\sqrt{\frac{4\beta}{t\pi}}
    \end{aligned}
  \end{equation*}
  The final equality follows as a straightforward consequence of the fact that
  $\Gamma(x+1)=x\,\Gamma(x)$.
\end{proof}

Using the results of Lemma~\ref{l:GradPhiBeta} and Lemma~\ref{l:Gron:frac}, we may now establish Lemma~\ref{l:rho:est}.

\begin{proof}[Proof of Lemma~\ref{l:rho:est}]
Let $\Phi_\beta$ be the heat kernel as defined in \eqref{eq:TorusHeatKernel}.
We can characterise the solution $\bmu$ to~\eqref{FP} as
\begin{equation*}
  \bmu(t) = \Phi_\beta(t) * \bmu^\circ
  - \int_0^t \sum_{l=1}^n\sum_{m=1,2}\partial_{l,m} \Phi_\beta(t-s)*\big(F_{l,m}\,\bmu(s)\big)\mathrm{d}s,
\end{equation*}
where all convolutions are spatial. Taking densities with respect to $\bnu$, this becomes
\begin{equation*}
  \beff(t) = \Phi_\beta(t) * \beff^\circ
  - \int_0^t \sum_{l=1}^n\sum_{m=1,2}\partial_{l,m} \Phi_\beta(t-s)*\big(F_{l,m}\,\beff(s)\big)\mathrm{d}s.
\end{equation*}
Differentiating, we find that
\begin{equation*}
  \partial^k_{i,j} \beff(t) = \Phi_\beta(t) * \partial^k_{i,j}\beff^\circ 
  - \int_0^t \sum_{l=1}^n\sum_{m=1,2} \partial^k_{i,j} \big(F_{l,m}\,\beff(s) \big) * \partial_{l,m} \Phi_\beta(t-s) \mathrm{d}s.
\end{equation*}
Taking norms and using the fact that $\|f*g\|_{L^\infty}\leq \|f\|_{L^1}\|g\|_{L^\infty}$ we obtain
\begin{align*}
  \big\|\partial^k_{i,j} \beff(t)\big\|_\infty
  &\leq \big\|\Phi_\beta(t)\big\|_{L^1(\bnu)}\big\|\partial^k_{i,j}\beff^\circ\big\|_\infty \\
  &\qquad\qquad+ \int_0^t \sum_{l=1}^n\sum_{m=1,2} \big\|\partial^k_{i,j} \big(F_{l,m}\,\beff(s) \big)
    \big\|_\infty \|\partial_{l,m} \Phi_\beta(t-s)\|_{L^1(\bnu)}ds.
\end{align*}
Now using the expressions provided in Lemma~\ref{l:GradPhiBeta}, we obtain
\begin{equation}\label{eq:dkf_Linfty}
  \big\|\partial^k_{i,j} \beff(t)\big\|_\infty
  \leq \big\|\partial^k_{i,j}\beff^\circ\big\|_\infty
  + 2 \sqrt{\frac\beta{\pi}} \int_0^t \sum_{l=1}^n\sum_{m=1,2}  \frac{ \big\|\partial^k_{i,j} \big(F_{l,m}\,\beff(s) \big)
  \big\|_\infty }{\sqrt{t-s}} ds.
\end{equation}
We now use this expression to prove the estimate by induction.
For $k = 0$, we have
\begin{equation*}
  \|\beff(t)\|_\infty
  \leq \|\beff^\circ\|_\infty
  + \frac{4n\mathsf C_V}{\delta}\sqrt{\frac{\beta}{\pi }} \int_0^t \frac{ \|\beff(s)\|_\infty }{\sqrt{t-s}}
  \mathrm{d}s.
\end{equation*}
Applying the result of Lemma \ref{l:Gron:frac}, we obtain
\begin{equation}\label{eq:0OrderSupBound}
  \|\beff(t)\|_\infty 
  \leq G_0(t)
  := 2 \|\beff^\circ\|_\infty \exp \left( 16 \mathsf C_V^2 \beta n^2 \delta^{-2} t \right) 
  = 2\|\beff^\circ\|_\infty\mathrm{e}^{\gamma t}.
\end{equation}

For $k=1$, using the product rule and \eqref{eq:0OrderSupBound}, it follows that
\eqref{eq:dkf_Linfty} may be estimated by
\begin{equation*} 
\begin{aligned}
  &\big\|\partial_{i,j} \beff(t)\big\|_\infty\\
  & \leq \big\|\partial_{i,j}\beff^\circ\big\|_\infty
  + 2 \sqrt{\frac\beta{\pi}} \int_0^t \sum_{l=1}^n\sum_{m=1,2} \big\|\partial_{i,j} F_{l,m}\big\|_\infty\big\|\beff(s)
  \big\|_\infty \frac{ 1 }{\sqrt{t-s}}ds\\
  & \qquad+ 2 \sqrt{\frac\beta{\pi}} \int_0^t \sum_{l=1}^n\sum_{m=1,2} \big\|F_{l,m}\big\|_\infty\big\|\partial_{i,j} \beff(s)
  \big\|_\infty \frac{1}{\sqrt{t-s}}ds\\
  & \leq \big\|\partial_{i,j}\beff^\circ\big\|_\infty
  + 4 \sqrt{\frac\beta{\pi}} \mathsf C_V n \delta^{-2} \int_0^t \frac{G_0(s)}{\sqrt{t-s}} ds
  + 4 \sqrt{\frac\beta{\pi}} \mathsf C_V n \delta^{-1} \int_0^t \frac{ \big\|\partial_{i,j}\beff(s)\big\|_\infty }{\sqrt{t-s}}ds\\
  & \leq\big\|\partial_{i,j}\beff^\circ\big\|_\infty
  + 2\sqrt{\frac{\gamma t}{\pi}} \delta^{-1} G_0(t)
  + \sqrt{\frac{\gamma}{\pi}} \int_0^t \big\|\partial_{i,j}\beff(s)\big\|_\infty\frac{1}{\sqrt{t-s}}ds,
\end{aligned}
\end{equation*}
where, to estimate the first integral term on the penultimate line, we have relied on the
elementary upper bound
\begin{equation*}
  \int_0^t \frac{ h(s) }{\sqrt{t-s}}\mathrm{d}s
  \leq h(t) \int_0^t \frac1{\sqrt{t-s}}\mathrm{d}s = 2 \sqrt t \,h(t),
\end{equation*}
for any monotone increasing function $h$.
Applying Lemma \ref{l:Gron:frac} once more, we now obtain
\begin{equation*}
  \begin{aligned}
  \big\|\partial_{i,j} \beff(t)\big\|_\infty
  \leq G_1(t)
  &:= 2 \left(\big \|\partial_{i,j}\beff^\circ\big\|_\infty + 2\sqrt{\frac{\gamma t}{\pi}} \delta^{-1} G_0(t) \right)\mathrm{e}^{\gamma t} \\
  &= 2 \big\|\partial_{i,j}\beff^\circ\big\|_\infty\mathrm{e}^{\gamma t} 
     + 4\sqrt{\frac{\gamma t}{\pi}} \delta^{-1} \|\beff^\circ\|_\infty\mathrm{e}^{2 \gamma t}.
\end{aligned}
\end{equation*}
For $k=2$, \eqref{eq:dkf_Linfty} becomes 
\begin{equation*}
  \begin{aligned}
  \big\|\partial^2_{i,j} \beff(t)\big\|_\infty
  &\leq \big\|\partial_{i,j}^2 \beff^\circ\big\|_\infty
  + 2\sqrt{\frac{\gamma t}{\pi}} \delta^{-1} \sum_{a=0}^1\bin{2}{a} G_a(t) \delta^{a-2} \\
  &\qquad + \sqrt{\frac{\gamma}{\pi}} \int_0^t \big\|\partial_{i,j}^2 \beff(s)\big\|_\infty\frac{1}{\sqrt{t-s}}ds.
\end{aligned}
\end{equation*}
Again, applying Lemma \ref{l:Gron:frac}, we obtain
\begin{equation*}
  \begin{aligned}
  \big\|\partial^2_{i,j} \beff(t)\big\|_\infty
  &\leq G_2(t)
  := 2 \left(\big\|\partial^2_{i,j}\beff^\circ\big\|_\infty
    + 2\sqrt{\frac{\gamma t}{\pi}} \delta^{-1} \sum_{a=0}^1\bin{2}{a} G_a(t) \delta^{a-2} \right)\mathrm{e}^{\gamma t} \\
  &= 2 \big\|\partial^2_{i,j}\beff^\circ\big\|_\infty\mathrm{e}^{\gamma t} 
     + 16\sqrt{\frac{\gamma t}{\pi}} \delta^{-2} \big\|\partial_{i,j}\beff^\circ\big\|_\infty\mathrm{e}^{2 \gamma t}\\
  &\qquad+ 8\sqrt{\frac{\gamma t}{\pi}} \delta^{-3} \|\beff^\circ\|_\infty\mathrm{e}^{2 \gamma t}
  + 32 \frac{\gamma t}{\pi} \delta^{-2} \|\beff^\circ\|_\infty\mathrm{e}^{3 \gamma t}.
\end{aligned}
\end{equation*}
By repeating the argument above inductively, the assertion of the lemma follows.
\end{proof}

We now restate and prove Corollary~\ref{c:rho:est}, which simplifies the previous lemma, and is used to prove Theorem~\ref{t:RWtoSDE}.

\cordfbnd*

\begin{proof}
Replacing $t$ by $T \geq 1$ in the expression in \eqref{e:bl:defn}, the term $m = \ell-1$ dominates in the sum as $\delta\to0$. Then, estimating $\delta^{-2} \leq \gamma$, we get
\begin{align*}
  b_k
  \lesssim \sqrt{ \gamma T } \gamma\mathrm{e}^{\gamma T} b_{k-1}
  &\lesssim (\gamma T)^{k/2} \gamma^k\mathrm{e}^{(k+1) \gamma T}.
\end{align*}
Inserting this upper bound into \eqref{pf:est:pkijft}, we obtain
\begin{align*} 
      \| \beff(t) \|_{k,\infty}
  &\leq \sum_{\ell=0}^k C_\ell (\gamma T)^{\ell/2} \gamma^\ell \| \beff^\circ \|_{k - \ell,\infty}\mathrm{e}^{(\ell+1) \gamma T}
\end{align*}
for different universal constants $C_\ell$.

Next, we continue the estimate by bounding the factors $(\gamma T)^{\ell/2}$ and $\gamma^\ell$. Since 
$$
 (\gamma T)^{\ell/2} \leq\mathrm{e}^{\ell \gamma T/2}
 \quad \text{and} \quad
 \gamma^\ell \lesssim\mathrm{e}^{\gamma /2} \leq\mathrm{e}^{\gamma T /2},
 $$
we obtain
\begin{align*} 
      \| \beff(t) \|_{k,\infty}
  &\leq \sum_{\ell=0}^k C_\ell\mathrm{e}^{\ell\gamma t/2}\mathrm{e}^{\gamma T /2} \| \beff^\circ \|_{k - \ell,\infty}\mathrm{e}^{(\ell+1) \gamma T}
  \leq\sum_{\ell=0}^k C_\ell \| \beff^\circ \|_{k - \ell,\infty}\mathrm{e}^{\frac32 (\ell+1) \gamma T}.
\end{align*}
This now directly yields the result as stated.
\end{proof}

\section{Proof of Theorem~\ref{t:MFetoMF}}
\label{sec:MFetoMF}
In this section we prove Theorem~\ref{t:MFetoMF}, which we recall here for convenience. 

\thmMFMFe*

As in the proof of Theorem~\ref{t:RWtoSDE}, the main technique used here is an application of Gronwall's Lemma after splitting the difference of the `nonlinear generators' that appear on the right-hand side of $\MFe$ and $\MF$; we recall the definitions of these equations from \eqref{RW:FP-MF} and~\eqref{MF}. This splitting yields three terms: two of these can be dealt with in analogue to our treatment of the splitting performed in the proof of Theorem~\ref{t:RWtoSDE}, while the third is a product of the nonlinearity, and requires new arguments.

\subsection{Main argument}
To facilitate our analysis, we first introduce convenient notation. First, to avoid cluttering of sums over $\pm$, we introduce the volume measures $\cV \in \cM_+ (\T_\pm^2)$ and $\cV_\e \in \cM_+ (\Lpm)$ by defining
\begin{align*}
  \cV(A^+, A^-) &:= \nu(A^+) + \nu(A^-) && \text{for all } \big(A^+\times\{+\}\big)\cup \big(A^-\times\{-\}\big) \subset \T_\pm^2 \\
  \cV_\e(A^+,A^-) &:= \nu_\e(A^+) + \nu_\e(A^-) && \text{for all } \big(A^+\times\{+\}\big)\cup \big(A^-\times\{-\}\big) \subset \Lpm.
\end{align*}
Then, given $\rho \in \cP (\T_\pm^2)$, we set $f^\pm := \frac{d \rho^\pm}{d \nu}$ as the related densities on $\T^2$, define
\begin{equation*}
  f(x, b) := \left\{ \begin{array}{ll}
    f^+(x)
    & \text{if } b = +1 \\
    f^-(x)
    & \text{if } b = -1
  \end{array} \right.
\end{equation*}
as the density on $\cP (\T_\pm^2)$, and observe from
\begin{multline*}
  \int_A d\rho 
  = \sum_\pm \int_{A^\pm} d\rho^\pm
  = \sum_\pm \int_{A^\pm} f^\pm d\nu
  = \int_{A} f d\cV \\
   \text{for all } A = \big(A^+\times\{+\}\big)\cup \big(A^-\times\{-\}\big) \subset \T_\pm^2
\end{multline*}
that $f = \frac{d \rho}{d \cV}$. Similarly, for given $\rho_\e \in \cP (\Lpm)$, we define $f_\e^\pm$ and $f_\e$ with respect to the volume measures $\nu_\e$ and $\cV_\e$. With these definitions, the square of the left-hand side in the asserted estimate in Theorem \ref{t:MFetoMF} reads as (removing the time variable)
\begin{equation*}
  \sum_\pm\left\| f_\e^\pm - f^\pm \right\|_{L^2(\nu_\e)}^2
  = \sum_\pm \int_{\L} (f_\e^\pm - f^\pm)^2 d\nu_\e
  = \int_{\Lpm} (f_\e - f)^2 d\cV_\e
  = \| f_\e - f \|_{L^2(\cV_\e)}^2.
\end{equation*}

Second, we introduce the generators in the right-hand sides of \eqref{RW:FP-MF} and~\eqref{MF}. Recall that the function $F:\T_\pm^2\times \mathcal P(\T_\pm^2) \to \R^2$ defined in \eqref{RW-MF:Rates} is given by 
\begin{equation*}
 F(x,b;\rho) := -b\nabla V_\delta*\big(\rho^+-\rho^-\big)(x).
\end{equation*}
When considering $\MFe$, for any $\rho\in \mathcal P(\T_\pm^2)$ and $\sigma_\e \in \mathcal P(\Lpm)$, we introduce 
\[
A_\e[\rho] \frac{d \sigma_\e}{d \cV_\e} (\ell,b) := \frac1{\e\beta} \sum_{h\in \Nhd_\e} D_{-h} 
  \Bigl[ \frac{d \sigma_\e}{d \cV_\e} \exp\bigl( \tfrac12 \beta h \cdot F(\,\cdot\,; \rho) \bigr) \Bigr](\ell,b) 
  \qquad \text{for }(\ell,b)\in \Lpm, 
\]
so that $\MFe$ reads as 
\begin{equation} \label{eq:MF-MFe-eqns-fe}
  \partial_t f_\e = A_\e [\rho_\e] f_\e\qquad\text{on } \Lpm.
\end{equation}
Similarly, when considering $\MF$, for any $\rho,\sigma\in \mathcal P(\T_\pm^2)$, we set 
\[
A[\rho]\frac{d \sigma}{d \cV} := -\div\Bigl(\frac{d \sigma}{d \cV} F(\,\cdot\,; \rho)\Bigr)
+ \beta^{-1} \Delta \frac{d \sigma}{d \cV}
\qquad \text{on } \T^2_\pm
\]
and note that $\MF$ reads as
\begin{equation} \label{eq:MF-MFe-eqns-f}
  \partial_t f = A[\rho]f \qquad \text{on } \T^2_\pm.   
\end{equation}
\bigskip

Next we sketch the proof of Theorem \ref{t:MFetoMF}. Given the setting of the theorem, we first argue that the densities $f$ and $f_\e$ are uniquely defined through \eqref{eq:MF-MFe-eqns-fe} and \eqref{eq:MF-MFe-eqns-f}. Since \eqref{eq:MF-MFe-eqns-fe} is a system of ODEs of size $2/\e^2$ with regular right-hand side (which is moreover bounded since the mass of $\rho_\e$ is conserved), $f_\e$ is indeed uniquely defined for all $t \in [0,T]$. Regarding \eqref{eq:MF-MFe-eqns-f}, since the nonlinear part in the right-hand side is regular and contains derivatives only up to first order, it follows from the regularity theory of semilinear parabolic PDEs \cite{LSU1968} that \eqref{eq:MF-MFe-eqns-f} admits a unique, global in time classical solution $f$.

With $f$ and $f_\e$ characterised, we obtain from \eqref{eq:MF-MFe-eqns-fe} and \eqref{eq:MF-MFe-eqns-f} that
\begin{align}
\label{eq:time-deriv-MF-MFe}
\frac12 \frac d{dt} \big\|f_\e(t)-f(t)\big\|^2 _{L^2(\cV_\e)} 
&= \big( f_\e - f , A_\e [\rho_\e] f_\e - A [\rho] f \bigr)_{L^2(\cV_\e)}.
\end{align}
We split the difference of the generators as 
\begin{align} \label{T123}
A_\e [\rho_\e] f_\e - A [\rho] f
= \underbrace{A_\e [\rho_\e] (f_\e - f)}_{T_1}
 + \underbrace{(A_\e[\rho_\e] - A[\rho_\e])f}_{T_2} 
 + \underbrace{(A[\rho_\e] - A[\rho])f}_{T_3},
\end{align}
and treat the three terms individually in the lemmas below. The constants $C$ appearing in these lemmas share the same independency of the parameters as the constants appearing in Theorem \ref{t:MFetoMF}. The two first lemmas are direct analogues of the stability and consistency lemmas in Section~\ref{sec:RWtoSDE}.

\begin{lem}[Estimate of $T_1$:\ stability]
\label{l:T1}
There exists $C>0$ such that for all $\rho_\e\in \mathcal P(\Lpm)$ and all $g\in L^2(\cV_\e)$, we have 
\[
\big(A_\e[\rho_\e] g, g \big)_{L^2(\cV_\e)}
\leq C \beta \delta^{-2} \|g\|^2_{L^2(\cV_\e)}.
\]
\end{lem}

\begin{lem}[Estimate of $T_2$:\ consistency]
\label{l:T2}
There exists $C>0$ with the following property. Let $\phi \in C^4(\T_\pm^2)$ and $\rho_\e\in \mathcal P(\Lpm)$. Then
\[
\max_{(\ell,b) \in \Lpm} \big| ( A_\e[\rho_\e] - A[\rho_\e] ) \phi (x,b) \big|
\leq \frac{C\,\e^2\,\beta^3}{\delta^4} \|\phi\|_{4,\infty}.
\]
\end{lem}
\bigskip

The third lemma is an extra ingredient that arises from the nonlinearity of equations $\MF$ and $\MFe$.

\begin{lem}[Estimate of $T_3$:\ nonlinearity]
\label{l:T3}
There exists $C>0$ such that for any $\phi \in C^1(\T^2)$, $\rho_\e\in \mathcal P(\Lpm)$ and $\rho\in \mathcal P(\T_\pm^2)$ with $f \in C^2(\T^2)$, we have 
\begin{multline*}
\max_{(x,b) \in \T^2_\pm} \big| ( A[\rho_\e] - A[\rho] ) \phi (x,b) \big|
\leq \frac{C}{\delta^4} \Big[ \delta \|d\phi\|_\infty + \|\phi\|_\infty\Big]\Big(\|f_\e - f\|_{L^2(\cV_\e)}
+ \e^2 \|d^2f\|_{\infty} \Big).
\end{multline*}
\end{lem}

The application of these lemmas requires the solution $\rho$ to $\MF$ to be sufficiently regular. The \emph{nonlinear} problem $\MF$ has similar regularity properties as the linear counterpart $\FP$ introduced in \eqref{FP}; in particular, we have the following result.

\begin{lem}[Regularity estimates of the solution $f$]
\label{l:reg-MF}
There exist universal constants $C_k>0$ such that for any solution $\rho$ of $\partial_t f = A[\rho] f$ with initial datum $\rho^\circ$, we have
\begin{subequations}\label{est:reg-MF}
\begin{align}\label{est:reg-MF:k0}
\|f(t)\|_\infty &\leq C_0 \|f^\circ\|_\infty\mathrm{e}^{\gamma t}, \\\label{est:reg-MF:k1}
\|d f(t)\|_\infty &\leq C_1 \Big( \|f^\circ\|_{1, \infty}\mathrm{e}^{\gamma t} + \sqrt \beta \delta^{-3} \|f^\circ\|_\infty \sqrt t\mathrm{e}^{2 \gamma t} \Big), \\\label{est:reg-MF:k234}
\|d^k f(t)\|_\infty &\leq C_k\mathrm{e}^{2(k+2)\gamma T}
\qquad \text{for $k=0, \ldots, 4$,}
\end{align}
\end{subequations}
with $\gamma = 16 \mathsf C_V^2 \beta / \delta^2$.
\end{lem}

Leaving the proofs of these four lemmas to the following section,
we complete the proof of Theorem~\ref{t:MFetoMF}. Let
\begin{equation*}
  v_\e^2(t) := \|f_\e (t) - f (t)\|^2_{L^2(\nu_\e^\pm)},
\end{equation*}
Recalling \eqref{eq:time-deriv-MF-MFe} and \eqref{T123}, we estimate the inner products of $f_\e (t) - f (t)$ with $T_i$ by applying Lemmas~\ref{l:T1}--\ref{l:T3}. This yields 
\begin{align*}
\frac d{dt}  v_\e^2(t)
&\lesssim
\beta \delta^{-2} v_\e^2(t)
+ \frac{\e^2\,\beta^3}{\delta^4} v_\e(t) \|f(t)\|_{4,\infty} \\
&\qquad {}+ \frac{1}{\delta^4} \; \Big( \big( \delta \|df(t)\|_\infty + \|f(t)\|_\infty \big) \; \Big( v_\e(t) + \e^2 \|d^2 f(t)\|_{\infty} \Big) v_\e(t) \\
&= \delta^{-4} \big(\beta \delta^2 + \delta \|df(t)\|_\infty + \|f(t)\|_\infty \big) v_\e^2(t) \\
&\qquad {}+ \e^2 \delta^{-4} \Big( \beta^3 \|f(t)\|_{4,\infty} + \big(\delta \|df(t)\|_\infty + \|f(t)\|_\infty \big) \|d^2 f(t)\|_{\infty} \Big) v_\e(t).
\end{align*}

Next we estimate the right-hand side further before applying Gronwall's Lemma. Since the main contribution in the final result comes from the prefactor of $v_\e(t)^2$, we apply fine estimates to control it, and use more rough estimates to bound the prefactor of $v_\e(t)$. In particular, using \eqref{est:reg-MF:k0} and \eqref{est:reg-MF:k1} in Lemma~\ref{l:reg-MF}, we obtain
\begin{align*}
  &\beta \delta^2 + \delta \|df(t)\|_\infty + \|f(t)\|_\infty \\
  &\lesssim \beta \delta^2 
      + \delta \| f^\circ \|_{1,\infty}\mathrm{e}^{\gamma t} 
      + \sqrt\beta \delta^{-2} \| f^\circ \|_{\infty} \sqrt t\mathrm{e}^{2 \gamma t} 
      + \| f^\circ \|_{\infty}\mathrm{e}^{\gamma t} \\
  &\lesssim \big( \beta \delta^2 + \delta \| f^\circ \|_{1,\infty} + \sqrt{\beta T} \delta^{-2} \| f^\circ \|_{\infty} \big)\mathrm{e}^{2 \gamma t}.
\end{align*}
For the prefactor of $v_\e(t)$, it suffices to apply \eqref{est:reg-MF:k0}. This yields
\begin{multline*}
  \delta^{-4} \Big( \beta^3 \|f(t)\|_{4,\infty} + \big(\delta \|df(t)\|_\infty + \|f(t)\|_\infty \big) \|d^2 f(t)\|_{\infty} \Big) \\
  \lesssim \delta^{-4} \Big( \beta^3\mathrm{e}^{12 \gamma T} + \big(\delta\mathrm{e}^{6 \gamma T} +\mathrm{e}^{4 \gamma T} \big)\mathrm{e}^{8 \gamma T} \Big)
  \lesssim\mathrm{e}^{15 \gamma T},
\end{multline*}
where we have estimated negative powers of $\delta$ by $e^{\gamma T}$. 
Collecting the estimates, we obtain
\begin{equation*}
  \frac d{dt}  v_\e^2(t)
  \lesssim K_2\mathrm{e}^{2 \gamma t} v_\e^2(t) + \e^2\mathrm{e}^{15 \gamma T} v_\e(t), \qquad K_2 := \beta \delta^{-2} + \sqrt{\beta T} \delta^{-6} \| f^\circ \|_{\infty} + \delta^{-3} \| f^\circ \|_{1,\infty}.
\end{equation*}
Applying Cauchy-Schwarz to the second term and using that $K_2\mathrm{e}^{2 \gamma t} \geq 1$, we find that
\begin{align*}
\frac d{dt}  v_\e^2(t)
\leq C K_2\mathrm{e}^{2 \gamma t} v_\e^2(t) + C' \e^4\mathrm{e}^{30 \gamma T}.
\end{align*}
Finally, we apply Gronwall's Lemma. This yields
\begin{align*}
 v_\e^2(t)
\leq \big( v_\e(0)^2 + C' \e^4\mathrm{e}^{30 \gamma T} t \big) \exp \big( 2 C K_2 \gamma\mathrm{e}^{2 \gamma t} \big)
\end{align*}
The assertion of Theorem~\ref{t:MFetoMF} follows by observing that
\begin{equation*}
  K_2 \gamma \lesssim \beta \delta^{-4} \big( \delta^{-1} \| f^\circ \|_{1,\infty} + \sqrt{\beta T} \delta^{-4} \| f^\circ \|_{\infty} + \beta \big).
\end{equation*}

\subsection{Proofs of auxiliary results}
%

For the proofs of Lemmas \ref{l:T1} and \ref{l:T2} below, we set
$$F^\pm := F(\cdot,\pm1;\rho_\e):\T^2\to\R^2.$$

\subsubsection{Proof of Lemma~\ref{l:T1}}
We apply an argument similar to that used in the proof of Lemma~\ref{l:staby} with the choice $n=1$. Replacing $\bF$ in this argument by $F^\pm$ as defined above, the operator $A[\rho_\e]$ acting on functions with $b=\pm1$ fixed coincides with $\Omega_\e^*$ in~\eqref{FPe-F}. Following the argument of the proof of Lemma~\ref{l:staby} then directly implies that
\[
\big(A_\e[\rho_\e] g^\pm, g^\pm \big)_{L^2(\nu_\e)}
\leq
C \big(\|d F\|_\infty + \beta \|F\|^2_\infty \big) \|g^\pm \|^2_{L^2(\nu_\e)}.
\]
Applying Assumption~\ref{ass:Vd} and using $\beta \geq C$ from Assumption \ref{ass:beta-e-delta-bound}, we find that 
\[
\big(A_\e[\rho_\e] g, g\big)_{L^2(\cV_\e)}
\leq
C \beta \delta^{-2} \|g\|^2_{L^2(\cV_\e)}.
\]

\subsubsection{Proof of Lemma~\ref{l:T2}}
As in the previous proof, we follow the strategy of the proof of consistency Lemma~\ref{l:consis} applied to the operators $A[\rho_\e]$ and $A[\rho_\e]$ to find
\[
\max_{\ell \in \L} \big|(A_\e[\rho_\e]-A[\rho_\e]) \phi^\pm (\ell) \big|
\leq 
\frac{C \e^2}\beta \sum_{m=0}^4 \beta^m \max_{j=1,2}\big\|d^{4-m} \big( \phi^\pm (F_j^\pm)^m \big) \big\|_\infty.
\]
Using Asssumption~\ref{ass:Vd} and $\beta \geq C$ from Assumption \ref{ass:beta-e-delta-bound}, we find that
\[
\sum_{m=0}^4 \beta^m \max_{j=1,2}\big\|d^{4-m} \big( \phi^\pm (F_j^\pm)^m \big) \big\|_\infty
\leq C\beta^4 \sum_{k=0}^4 \delta^{k-4}\|d^k \phi\|_\infty
\leq \frac{C\beta^4}{\delta^4} \|\phi\|_{4,\infty}.
\]
The assertion of the lemma now follows.

\subsubsection{Proof of Lemma~\ref{l:T3}}
For the proof of Lemmas \ref{l:T3}, the dependence of $F^\pm$ on $\rho_\e$ becomes important; we set
$$F^\pm(\rho) := F(\cdot,\pm1;\rho):\T^2\to\R^2.$$

Since $A[\rho]\phi = -\div (\phi F(\rho)) + \beta^{-1}\Delta \phi$, and $F$ is linear in $\rho$, we can estimate 
\begin{align*}
\big\|(A[\rho_\e] - A[\rho]) \phi \big\|_\infty
&= 
\big\|\div (\phi F(\rho_\e-\rho))\big\|_\infty\\
&\leq 
\big\| \nabla \phi \cdot F(\rho_\e-\rho)\big\|_\infty +
  \big\| \phi\,\div F(\rho_\e-\rho)\big\|_\infty\\
&\leq \|d\phi\|_\infty \|F(\rho_\e-\rho)\|_\infty + \|\phi\|_\infty \|\div F(\rho_\e-\rho)\|_\infty.
\end{align*}
We continue with the two norms of $F$. Writing $\kappa := (\rho_\e^+ - \rho^+) - (\rho_\e^- - \rho^-)$, 
\begin{align*}
\|F(\rho_\e-\rho)\|_\infty 
&=\|\nabla V_\delta * \kappa \|_\infty 
  \leq \|\nabla V_\delta\|_{2,\infty} \|\kappa\|_{2,\infty}^*
  \leq \frac C{\delta^3} \|\rho_\e - \rho\|_{2,\infty}^*\\
\|\div F(\rho_\e-\rho)\|_\infty 
&=\|\Delta V_\delta * \kappa\|_\infty 
  \leq \|\Delta V_\delta\|_{2,\infty} \|\kappa\|_{2,\infty}^*
  \leq \frac C{\delta^4} \|\rho_\e - \rho\|_{2,\infty}^*.
\end{align*}
We find that 
\begin{equation}
\label{est:MF-MFe-nonlinearity}
\big\|(A[\rho_\e] - A[\rho])\phi \big\|_\infty
\leq 
\frac C{\delta^4}  \Big[ \delta \|d\phi\|_\infty + \|\phi\|_\infty\Big]
 \|\rho_\e - \rho\|_{2,\infty}^*.
\end{equation}

\bigskip

To estimate the norm  $\|\rho_\e - \rho\|_{2,\infty}^*$ we split it as
\begin{equation}
\label{est:MF-MFe-split}
\|\rho_\e - \rho\|_{2,\infty}^*
\leq \|\rho_\e - f \cV_\e \|_{2,\infty}^*
+\|f \cV_\e - \rho\|_{2,\infty}^*.
\end{equation}
The first term is estimated by 
\begin{align}
\|\rho_\e - f \cV_\e\|_{2,\infty}^*
&= \sup_{\|\varphi\|_{2,\infty}\leq 1} \int_{\T^2} \varphi \, (f_\e - f) \, d \cV_\e \notag\\
&= \sup_{\|\varphi\|_{2,\infty}\leq 1} \big(\varphi, f_\e - f \big)_{L^2(\cV_\e)}
  \leq \|f_\e-f\|_{L^2(\cV_\e)} .
\label{est:MF-MFe-2inf-a}
\end{align}
For the second term, let $\{(\ell_i,b_i)\}_{i=1,\dots,2/\e^{2}}$ be an enumeration of $\Lpm$, and let $\{\mathcal C_i\}_{i=1,\dots,2/\e^{2}}$ be the corresponding open, square Voronoi cells in $\T^2_\pm$. The cells $\mathcal C_i$ are disjoint, have $\cV (\mathcal{C}_i) = \e^{2}$,  and $\T^2_\pm\setminus \bigcup_i \mathcal C_i$ is a Lebesgue null set. Each cell $\mathcal C_i$ is centered with respect to $\ell_i$, i.e., $\int_{\mathcal C_i} (x-\ell_i) \, dx = 0$. Then, for any $g \in C^2(\T_\pm^2)$, it follows from the second order Taylor approximation and the symmetry of $\cC_i$ that 
\[
\bigg|\,\frac1{\e^2}\int_{\mathcal C_i} \big(g(\ell_i,b_i) -  g(x,b_i)\big)\, dx\,\bigg|
\leq C\e^2 \|d^2g\|_\infty.
\]
Therefore
\begin{align} \notag
\|f \cV_\e - \rho\|_{2,\infty}^*
&= \sup_{\|\varphi\|_{2,\infty}\leq 1} \int_{\T^2_\pm} \varphi f \, d (\cV_\e - \cV) \\\notag
&= \sup_{\|\varphi\|_{2,\infty}\leq 1} \sum_{i=1}^{2/\e^2} 
  \int_{\mathcal C_i} \big[ (\varphi f)(\ell_i,b_i) - (\varphi f)(x,b_i) \big]\, dx\\\notag
&\leq C \e^4 \sum_{i=1}^{2/\e^2} 
  \sup_{\|\varphi\|_{2,\infty}\leq 1}
  \|d^2( \varphi f)\|_{\infty}\\\label{fvp:2infty}
&\leq C \e^2 \|f\|_{2,\infty}.
\end{align}
Combining this estimate with~\eqref{est:MF-MFe-nonlinearity}, \eqref{est:MF-MFe-split}, \eqref{est:MF-MFe-2inf-a}, and an application of Cauchy--Schwarz we find the assertion of the Lemma. 

\subsubsection{Proof of Lemma~\ref{l:reg-MF}}
Since Lemma~\ref{l:rho:est} and Corollary~\ref{c:rho:est} only require that $F^\pm$ satisfies $\|d^k F^\pm\|_\infty \leq \mathsf C_V\, \delta^{-k-1}$ for $k=0,\dots,4$, we can apply them with $n=1$ to find the same estimate~\eqref{est:reg-MF}. In addition, in \eqref{est:reg-MF:k234} we further rely on the given polynomial bound on $\|f^\circ\|_{k, \infty}$ to absorb it in the exponential.

\section{Proof of Theorem~\ref{t:SDEtoMF}}
\label{s:SDEn:FP-Mark} 
This section is devoted to the proof of Theorem~\ref{t:SDEtoMF}; for convenience, we restate this result here.

\thmSDEMF*

The proof of this result is based on an established strategy for proving `propagation of chaos', i.e., the property that the components of the process $\bX_t$ are approximately independent when $n$ is large. This strategy goes back at least to McKean~\cite{McKean67}; we follow Sznitman's treatment~\cite{Sznitman91,Philipowski07} while generalizing to particles of two signs. Duong and Tugaut~\cite{DuongTugaut18TR} prove a similar result, but impose weaker assumptions and obtain a weaker bound; since we care about the bound, we give a full proof of the result here. 

Overall, the strategy in the proof of Theorem~\ref{t:SDEtoMF} is to carry out the following steps:
\begin{itemize}
  \item We first construct an auxiliary stochastic process $\{ \obX_t \}_{0 \leq t \leq T}$ which is driven by the same noise as $\{ \bX_t \}_{0 \leq t \leq T}$, but in which $\oX_{t,i}$ are all independent and all particles of the same sign are identically distributed.
  \item We split the norm $\E \| \rho_n^\pm (t) - \rho^\pm(t) \|_{1,\infty}^*$ which is our measure of comparison between the law of $\SDE$ and $\MF$ into two parts, by introducing the empirical measure $\orho_n$ of the auxiliary process $\obX$.
  \item One of these parts is bounded via an estimate of $\E \sum_{i=1}^n | X_i - \oX_i |(t)$ using propagation of chaos techniques.
  \item The other part is bounded using the quantitative Glivenko--Cantelli estimate first used by Fournier and Jourdain~\cite{FJ17}. 
  \end{itemize}
  The latter estimates are then combined to complete the proof. As in previous sections, we provide an overview of the proof in the following section, and postpone the proofs of various technical results to the end of the section.

  \subsection{Main argument}

Let $\rho(t)$ and $\{ \bX_t \}_{0 \leq t \leq T}$ be as asserted in Theorem \ref{t:SDEtoMF}. In Sections \ref{sec:RWtoSDE} and \ref{sec:MFetoMF} we prove that $\rho(t)$ and $\{ \bX_t \}_{0 \leq t \leq T}$ are well defined. 
We assume that $n^+, n^- \geq 1$ and $\rho^{\circ,+},\rho^{\circ,-}\not=0$; all alternative cases can be treated with a simplification of the arguments below. 
We set 
\begin{equation} \label{mu:rho}
  \mu_t^\pm := \frac{\rho^\pm(t)}{\|\rho^\pm(t)\|_{\mathrm{TV}}} \in \cP(\T^2)
  \quad \text{for all } t \in [0,T],
\end{equation}
where the $TV$-norm of a non-negative measure is simply the total mass of that measure.

As laid out in the strategy described above, we define the stochastic process $\{ \obX_t \}_{0 \leq t \leq T}$ as the solution to
\begin{equation} \label{oSDE}
  \oSDE \qquad 
  \left\{ \begin{aligned}
    d \oX_i
    &= F \big(\oX_i, b_i; \rho(t)\big) \,dt + \sqrt{2\beta^{-1}}\, d B_i
    && t \in [0,T], \: i = 1, \ldots, n \\
    \obX_0
    &= \bX^\circ,
  \end{aligned} \right.  
\end{equation}
where $B_i$ are the same Brownian Motion processes as in \eqref{SDE}, and $F$ is defined in \eqref{RW-MF:Rates}. Recall that $X^\circ_i$ are assumed to be independently distributed with law proportional to $\rho^{\circ,+}$ if $i\in I^+$ or ${\rho^{\circ,-}}$ if $i\in I^-$ (the index sets $I^\pm$ are defined in \eqref{Ipm}). Before using this process as a tool to prove the estimates we seek, we first state some properties of $\{\obX_t\}_{0\leq t\leq T}$, which are encoded in the following lemma.

\begin{lem}[Properties of $\{ \obX_t \}_{0 \leq t \leq T}$] \label{l:oSDE:props}
Let the stochastic process $\{ \obX_t \}_{0 \leq t \leq T}$ be as defined in \eqref{oSDE}. Then:
\begin{enumerate}[label=(\roman*),ref=\roman*] 
  \item The stochastic processes $\oX_i$ for $i = 1,\ldots,n$ are independent; \label{l:oSDE:props:indep}
  
  \item For any $t \in [0,T]$, the law of $\oX_{t,i}$ with $i\in I^\pm$ is given by $\mu_t^\pm$. \label{l:oSDE:props:law}
\end{enumerate}
\end{lem}\medskip

\noindent
We now proceed to prove the estimate in Theorem \ref{t:SDEtoMF}. Defining the empirical measure
\begin{equation*}
  \overline \rho_n := \frac1n \sum_{i=1}^n \delta_{(\oX_i, b_i)}
\end{equation*}
and decomposing into $\orho_n^\pm$ as in \eqref{rhon:pm}, we use the triangle inequality to estimate
\begin{equation} \label{est:PoC}
 \E \| \rho_n^\pm(t) - \rho^\pm(t) \|_{1,\infty}^* 
 \leq \E \| \rho_n^\pm(t) - \orho_n^\pm(t) \|_{1,\infty}^* 
      + \E \| \orho_n^\pm(t) - \rho^\pm(t) \|_{1,\infty}^*.
\end{equation}
The two terms on the right hands side are now estimated in turn. To estimate the first of these, we rely on the following propagation of chaos result.

\begin{lem}[Propagation of chaos] \label{l:PoC}
For all $t \in [0,T]$,
\begin{equation*} 
\frac1n \E \sum_{i=1}^n | X_i - \oX_i |(t)
\leq \frac{ 2 \mathsf C_V  \,t }{\delta} \Big(\frac 1{\sqrt n} + \kappa\Big) \, \exp \Big( 2\mathsf C_V \frac t{\delta^2} \Big),
\end{equation*}
where the constant $\mathsf C_V$ is defined in Assumption \ref{ass:Vd}, and $\kappa$ is the mass discrepancy~\eqref{def:gamma-mass-disc}.
\end{lem}\medskip

\noindent
Applying Lemma~\ref{l:PoC} to the first term in the right-hand side yields:
\begin{align*}
  \E \| \rho_n^\pm(t) - \orho_n^\pm(t) \|_{1,\infty}^* 
  &= \E \sup_{\|\varphi\|_{1,\infty}\leq 1} 
     \int_{\T^2} \varphi (\rho_n^\pm(t) - \orho_n^\pm(t))\\
  &= \E \sup_{\|\varphi\|_{1,\infty}\leq 1} 
     \frac1n \sum_{i \in I^\pm} \bigl[ \varphi(X_i(t)) - \varphi(\overline X_i(t)) \bigr] \\
  &\leq \frac1n \E \sum_{i \in I^\pm} | X_i(t) - \oX_i(t) | \\
  &\leq \frac{ 2 \mathsf C_V  \,t }{\delta} \Big(\frac 1{\sqrt n} + \kappa \Big) \, \exp \Big( 2\mathsf C_V \frac t{\delta^2} \Big).
\end{align*}

The second term in \eqref{est:PoC} is now estimated via the result of the following lemma.

\begin{lem}[Quantitative Glivenko--Cantelli estimate]
  \label{l:GlivenkoCantelli}
We have
\begin{equation*}\E \| \orho_n^\pm(t) - \rho^\pm(t) \|_{1,\infty}^*
\leq
\kappa + 
C \frac{\log n}{\sqrt n}\, \|\rho^\pm(t)\|_{\mathrm {TV}},
\end{equation*}
for some universal constant $C > 0$.
\end{lem}

\noindent
Applying this result and the estimate of the first term in \eqref{est:PoC} already established now  concludes the proof of Theorem~\ref{t:SDEtoMF}.

\subsection{Proofs of auxiliary results}
This section is devoted to the proofs of the three auxiliary results used above to prove Theorem~\ref{t:RWtoSDE}.

\subsubsection{Proof of Lemma~\ref{l:oSDE:props}}
Lemma~\ref{l:oSDE:props} establishes two properties of the process $\{\obX_t\}_{0\leq t\leq T}$. To establish assertion \eqref{l:oSDE:props:indep}, we observe that since $X_i^\circ$ are independent random variables and $B_i$ are independent processes for all $i = 1,\ldots,n$, the fact that $\oX_i$ solves \eqref{oSDE} directly entails that $\oX_i$ are also independent processes. 
  
To show assertion \eqref{l:oSDE:props:law}, note that via an application of the {Feynman--Kac} formula we find that the Fokker--Planck equation for the law $\omu_t^\pm$ of $\oX_i(t)$ for $i \in I^\pm$ is
  \begin{equation} \label{oFP}
  \left\{ \begin{aligned}
    \partial_t \omu^\pm 
    & = - \div \big( \omu^\pm F ( \cdot , \pm1; \rho) \big) + \beta^{-1} \Delta \omu^\pm
    && \text{on } \R^2 \times (0, T), \\
    \omu_0^\pm
    &= \mu^\pm(0).
  \end{aligned} \right.
  \end{equation}
It is clear that \eqref{oFP} has a unique classical solution $\omu^\pm$, since this is a linear uniformly parabolic equation with smooth coefficients; see for example \cite{LSU1968}. From \eqref{MF} we observe that $\mu^\pm$ solves \eqref{oFP}; hence $\omu = \mu$ as stated.

\subsubsection{Proof of Lemma~\ref{l:PoC}}
To prove Lemma~\ref{l:PoC}, consider a realization $\{\bX(t)\}_{t}$ and $\{\obX(t)\}_t$ of the two processes. For each $i$, the curves $t\mapsto X_i(t)$ and $t\mapsto \oX_i(t)$ are almost surely continuous, and we can temporarily consider them as $\R^2$-valued curves. Since $X_i(0)$ and $\oX_i(0)$ are the same point in $\T^2$, we can translate the curve $\oX_i$ by a vector in $\Z^2$ such that $X_i(0)=\oX_i(0)$ as elements of $\R^2$. Since the function $F(\cdot,b;\rho(t))$ is $\Z^2$-periodic, this entails no loss of generality.

By definition of the SDEs \eqref{SDE} and \eqref{oSDE} we have for all $i = 1,\ldots,n$,
\begin{align*}
X_i(t) - \oX_i(t)
&= X_i(0) - \oX_i(0) + \int_0^t \big[ F_i(\bX(s), \bb) - F(\oX_i(s), b_i; \rho(s)) \big] \, ds \\
&= b_i \int_0^t \Big[ \big(\nabla V_\delta * (\rho_s^+ - \rho_s^-) \big)(\oX_i(s)) - \frac1n \sum_{j=1}^n b_j \nabla V_\delta(X_i(s) - X_j(s)) \Big] \, ds.
\end{align*}
Now, since 
\begin{align}
  \big| \nabla V_\delta(\oX_i - \oX_j) - \nabla V_\delta(X_i - X_j) \big|
  &\leq \| d^2 V_\delta \|_\infty \big| (\oX_i - \oX_j) - (X_i - X_j) \big| \notag\\
  &\leq\frac{\mathsf C_V}{ \delta^2} \big( |\oX_i - X_i| + |\oX_j - X_j| \big),
\label{est:PoC:pf0}
\end{align}
by adding and subtracting the same term and applying the triangle inequality, we obtain
\begin{align} 
&\sum_{i=1}^n | X_i(t) - \oX_i(t) | \notag\\
&\quad\leq \sum_{i=1}^n \int_0^t \Big| \big(\nabla V_\delta * (\rho_s^+ - \rho_s^-) \big)(\oX_i(s)) 
  - \frac1n \sum_{j=1}^n b_j \nabla V_\delta(\oX_i(s) - \oX_j(s)) \Big| \, ds \notag\\
&\quad\qquad{}+\sum_{i=1}^n \int_0^t \Big| \frac1n \sum_{j=1}^n b_j \nabla V_\delta(\oX_i(s) - \oX_j(s))
   -\frac1n \sum_{j=1}^n b_j \nabla V_\delta(X_i(s) - X_j(s)) \Big|\, ds \notag\\
&\quad\leq \sum_{i=1}^n \int_0^t \Big| \big(\nabla V_\delta * (\rho_s^+ - \rho_s^-) \big)(\oX_i(s)) - \frac1n \sum_{j=1}^n b_j \nabla V_\delta(\oX_i(s) - \oX_j(s)) \Big| \, ds \notag\\
&\quad\qquad{}+ \frac{\mathsf C_V}{\delta^2 n} \sum_{i,j=1}^n \bigg( \int_0^t |\oX_i - X_i|(s) \,ds + \int_0^t |\oX_j - X_j|(s) \,ds \bigg).
  \label{est:PoC:pf1}
\end{align}
To simplify the integrand of the first integral in the final upper bound, we write
\begin{align}
\nabla V_\delta *(\rho_s^+-\rho_s^-) (\oX_i)
&= \frac{n^+}n \nabla V_\delta * \mu_s^+ (\oX_i) - \frac{n^-}n \nabla V_\delta*\mu_s^- (\oX_i)
\notag\\
&\qquad{} + 
  \nabla V_\delta * \Big( \rho_s^+ - \frac{n^+}n \mu_s^+\Big) (\oX_i) 
   - \nabla V_\delta * \Big( \rho_s^- - \frac{n^-}n \mu_s^-\Big) (\oX_i).
\end{align}\label{Vd:rho:to:mu}
To treat the first two terms on the right--hand side of \eqref{Vd:rho:to:mu}, we will compare them to
$\frac{n^\pm}{n}\nabla V_\delta(\oX_i-y)$, which are of a similar form to those terms which appear in the second term forming the first integrand on the right--hand side of the bound given in \eqref{est:PoC:pf1}. To make this comparison, we set
\begin{equation}
\label{def:g-poc}
  g : (\R^2)^2 \times \{\pm1\} \to \R^2, \quad g (x,y,b) := b \big[ \big(\nabla V_\delta * \mu_s(\cdot, b) \big)(x) - \nabla V_\delta(x - y)\big]
\end{equation}
and note that
\begin{equation} \label{gpm:props}
  \| g \|_\infty 
  \leq 2 \| \nabla V_\delta \|_\infty
  \leq \frac{2 \mathsf C_V}\delta
  \quad \text{and} \quad
  \forall \, x \in \R^2, \, b \in \{\pm1\} : \int_{\T^2} g (x, y, b) \, \mu_s(dy, b) = 0.
\end{equation}

To treat the latter terms on the right--hand side of \eqref{Vd:rho:to:mu}, we use the definition of $\mu^\pm_s$ given in \eqref{mu:rho} and estimate
\begin{equation}
\Big|\nabla V_\delta * \Big( \rho_s^\pm - \frac{n^\pm}n \mu_s^\pm\Big) (\oX_i) \Big|
\leq \Big| \|\rho^\pm_s\|_{\mathrm{TV} } -\frac{n^\pm}n \Big| \ |\nabla V_\delta * \mu_s^\pm|
\leq \kappa \|\nabla V_\delta\|_\infty \leq \frac{\mathsf C_V}{\delta} \kappa, 
\label{est:PoC1-gamma}
\end{equation}
where $\kappa$ is the mass discrepancy defined in~\eqref{def:gamma-mass-disc}.

Using these estimates, we find that the upper bound in \eqref{est:PoC:pf1} can be further estimated above as
\begin{align}
\sum_{i=1}^n | X_i - \oX_i |(t)
&\leq \frac{2n\mathsf C_V\kappa} \delta \, t + \frac1{n} \sum_{i=1}^n \int_0^t \Big| \sum_{j=1}^n g(\oX_i(s), \oX_j(s), b_j) \Big| \, ds \notag\\
&\qquad{}+ \frac{2\mathsf C_V}{\delta^2} \int_0^t \Big( \sum_{i=1}^n |\oX_i - X_i|_{\T^2}(s) \Big) \,ds.
  \label{est:PoC:pf2}
\end{align}

Next we prepare to take the expectation of \eqref{est:PoC:pf2}; we focus on the first integral. Fixing $s$ in the integrand and removing it from the notation since we may exchange taking expectations and integrating in time, we find
\begin{equation*}
  \E \Big| \sum_{j=1}^n g(\oX_i, \oX_j, b_j) \Big|^2
  = \sum_{j=1}^n \E \big[ g(\oX_i, \oX_j, b_j)^2 \big]
    + \sum_{j \neq k} \E \big[ g(\oX_i, \oX_j, b_j) g(\oX_i, \oX_k, b_k) \big].
\end{equation*}
Using the  bound in \eqref{gpm:props}, we estimate the diagonal part as
\begin{equation*}
  \sum_{j=1}^n \E \big[ g(\oX_i, \oX_j, b_j)^2 \big] 
  \leq \frac{4 n \mathsf C_V^2}{\delta^2}.
\end{equation*}
All the off-diagonal terms turn out to be zero. To see this, we first treat the case $j \neq i \neq k$. Then, since $\oX_i$, $ \oX_j$ and $ \oX_k$ are independent, we obtain from \eqref{gpm:props} that
\begin{multline*}
  \E \big[ g(\oX_i, \oX_j, b_j) g(\oX_i, \oX_k, b_k) \big]
  = \iiint_{(T^2)^3} g(x,y, b_j) g(x,z, b_k) \, \mu(dz, b_k) \mu(dy, b_j) \mu(dx, b_i) \\
  = \int_{T^2} \bigg[ \int_{T^2} g(x,y, b_j) \mu(dy, b_j) \bigg] \bigg[ \int_{T^2} g(x,z, b_k) \mu(dz, b_k) \bigg] \mu(dx, b_i)
  = 0.
\end{multline*}
Similarly, when $k = i$, we obtain
\begin{equation*}
  \E \big[ g(\oX_i, \oX_j, b_j) g(\oX_i, \oX_i, b_i) \big]
  = \int_{T^2} \bigg[ \int_{T^2} g(x,y, b_j) \mu(dy, b_j) \bigg] g(x,x, b_i) \mu(dx, b_i)
  = 0.
\end{equation*}
The case $j = i$ can be treated analogously.

In conclusion, by applying the Cauchy--Schwarz inequality, we obtain
\begin{equation*}
  \E \Big| \sum_{j=1}^n g(\oX_i, \oX_j, b_j) \Big|
  \leq \bigg( \E \Big| \sum_{j=1}^n g(\oX_i, \oX_j, b_j) \Big|^2 \bigg)^{\tfrac12}
  \leq \frac{2 \mathsf C_V}\delta \sqrt{ n }.
\end{equation*}

Finally, taking the expectation of \eqref{est:PoC:pf2}, we get
  \begin{equation*} 
\E \sum_{i=1}^n | X_i - \oX_i |(t)
\leq \frac{2 \mathsf C_V}{\delta} (\sqrt n + \kappa n)\, t  + \frac{2\mathsf C_V}{\delta^2} \int_0^t \E \sum_{i=1}^n | X_i - \oX_i |(s) \,ds.
\end{equation*}
By applying Gronwall's Lemma we find the assertion of Lemma \ref{l:PoC}.

\subsubsection{Proof of Lemma~\ref{l:GlivenkoCantelli}}
Finally, to establish Lemma~\ref{l:GlivenkoCantelli}, we apply the quantitative  Glivenko--Cantelli estimate  of Fournier and Jourdain~\cite{FJ17} in the Wasserstein metric $W_1$. For non-negative measures $\eta_1$ and $\eta_2$ on $\T^2$ with equal mass, we have the straightforward estimate
\begin{align*}
\|\eta_1-\eta_2\|_{1,\infty}^*&=
\sup_{\|\varphi\|_{1,\infty}\leq 1} \int_{\T^2} \varphi(\eta_1-\eta_2)\\
&\leq\sup_{\|\nabla \varphi\|_{\infty}\leq 1} \int_{\T^2} \varphi(\eta_1-\eta_2)
=:  W_1(\eta_1,\eta_2).
\end{align*}
We set $\alpha := n\|\rho^+(t)\|_{\mathrm{TV}}/n^+$, so that $\alpha \orho_n^+(t)$ and $\rho^+(t)$ have equal mass. We then estimate 
\begin{align*}
\E \| \orho_n^+(t) - \rho^+(t) \|_{1,\infty}^*
&\leq 
\E \| \orho_n^+(t) - \alpha\orho_n^+(t) \|_{1,\infty}^*
+\E \|  \alpha\orho_n^+(t)- \rho^+(t) \|_{1,\infty}^*\\
&\leq \E \| (1-\alpha) \orho_n^+(t) \|_{1,\infty}^*
+\E W_1( \alpha\orho_n^+(t), \rho^+(t)) .
\end{align*}
The first term is equal to 
\[
|1-\alpha|\,\|\orho_n^+(t)\|_{\mathrm{TV}} = |1-\alpha|\frac{n^+}n = \left| \frac{n^+}n - \|\rho^+(t)\|_{\mathrm {TV}}\right| = \kappa,
\]
where as before, $\kappa$ is the discrepancy defined in \eqref{def:gamma-mass-disc},
and for the second we apply Theorem~1 in \cite{FJ17} to find
\[
\E W_1( \alpha\orho_n^+(t), \rho^+(t))
\leq 
C \frac{\log n}{\sqrt n}\, \|\rho^+(t)\|_{\mathrm {TV}}
\]
for some universal constant $C > 0$. The assertion of the lemma follows for the case $b = 1$, and the argument for $b=-1$ is identical, so the proof of Lemma~\ref{l:GlivenkoCantelli} is concluded.

\section{Proof of Theorem~\ref{t:RWtoMFe}}
\label{sec:RWtoMFe}
This section concerns itself with the proof of Theorem~\ref{t:RWtoMFe}. For convenience, we restate the result here in full.

\thmRWMFe*

\noindent
The proof we give of this result follows the same philosophy as the proof of Theorem~\ref{t:SDEtoMF} in the previous section. The main difference is that the state space is the discrete torus $\Lambda_\varepsilon$ instead of the continuum torus $\T^2$, and correspondingly the stochastic process is a random walk instead of a diffusion. 

In the previous section the main ingredient of the proof is a propagation--of--chaos statement. This statement estimates the divergence of two processes; one is the original process, and the second is a vector of i.i.d.\ processes constructed to have the same distribution as the solution of the mean-field problem.  The crucial point is that the estimate is obtained by subtracting the two SDEs   for the same realization of the noise. This reduces the impact of the randomness significantly and allows an estimate of the divergence of the solutions by using independence and Gronwall's Lemma. In the current setup on the discrete lattice such simple `subtraction of equations' is not possible; instead, in order to achieve a similarly strong coupling between the two processes, we write the processes in the `random time-change'
formulation. This formulation was  developed by Volkonskii~\cite{Volkonskii58}, Helms~\cite{Helms74}, and Kurtz~\cite{Kurtz80}, and an overview can be found in~\cite[Ch.~6]{EthierKurtz}.

\subsection{Main argument}
We first fix some notation. As in the proof of Theorem \ref{t:SDEtoMF}, we assume  that $n^+, n^- \geq 1$, with the cases where $n^+=0$ or $n^-=0$ being proved completely analogously.

We fix a probability space $(\Omega, \Sigma, \Prob)$. Each jump that a process on $\Lambda_\e^n$ can make is characterized by a particle number $i \in \{1,\ldots, n\}$ and a direction $h \in \Nhd_\e$.
We collect these into a single direction object $\bh\in \Nhd_\e^n$ as in \eqref{eq:Nhdn}, i.e.,
\begin{equation*}
  \bh = \big(\underbrace{0,\dots,0}_{i-1\text{ times}},h,\underbrace{0,\dots,0}_{n-i\text{ times}}\big)^T.
\end{equation*}
Recall that Theorem~\ref{t:RWtoMFe} assumes a given choice of vector $\bb$ (fixing $I^\pm$ as defined in \eqref{Ipm}) and a given initial distribution $\rho_\e^\circ\in \mathcal P(\T^2_\pm)$. Let the components of the initial vector $\bX^\circ_\e$ be chosen independently, with the law of $X^\circ_{\e,i}$ proportional to $\rho^\circ_+$ for $i\in I^+$ and $\rho^\circ_-$ for $i\in I^-$.

We construct a solution $\bX_\e$ of the random walk~$\RW$ as follows. 
For each of the $4n$ possible values of $\bh\in \Nhd_\e^n$, let $N^\bh$ be an independent standard Poisson process. 
These Poisson processes $N^\bh$ are the counterpart of the Brownian Motion processes in \eqref{SDE}, and we  couple the two processes by using the same realizations of $N^\bh$ in both processes. The stochastic process $\{ \bX_\e(t) \}_{0 \leq t \leq T}$ then is defined by the set of equations
\begin{subequations}
\label{RW:inte}
\begin{align} 
\tau^\bh(t) &= \int_0^t \cR^\e_{n,\bh}(\bX_\e(s), \bb)\, ds &\qquad &\bh\in \Nhd_\e^n, \ t\geq 0\\
\bX_\e(t) &= \bX_\e^\circ + \sum_{\bh\in\Nhd_\e^n} \bh N^\bh(\tau^\bh(t)) &&t\geq 0,
\end{align}
\end{subequations}
where we recall that the rates $\cR^\e_{n,\bh}$ are given in~\eqref{RW:Rates}.
In Lemma~\ref{lem:props-coupled-discrete} below we show that for each realization of $N^\bh$ these equations admit a solution $((\tau^\bh)_\bh, \bX_\e)$, and that $\bX_\e$ is a solution to~$\RW$. 

Next we construct the random walk counterpart of $\oSDE$ defined in \eqref{oSDE}. 
Let $\rho_\e(t)$  be as asserted in Theorem~\ref{t:RWtoMFe} (existence and uniqueness are proven in Section \ref{sec:MFetoMF}). 
For the same $\bX_\e^\circ$ and $N^\bh$ as given above, we construct the auxiliary process $\{ \obX_\e(t) \}_{0 \leq t \leq T}$ defined by
\begin{subequations}
\label{oRW:inte}
\begin{align} 
\otau^\bh(t) &= \int_0^t \cR_h^\e(\oX_{\e,i}(s), b_i; \rho_\e(s)) \, ds&\qquad &\bh\in \Nhd_\e^n, \ t\geq 0\\
\obX_\e(t) &= \bX_\e^\circ + \sum_{\bh\in\Nhd_\e^n} \bh N^\bh(\otau^\bh(t))&&t\geq 0,
\end{align}
\end{subequations}
where in this case, we recall that the rates $\cR_h^\e$ are given in~\eqref{RW-MF:Rates}.

As in the proof of Theorem~\ref{t:SDEtoMF} given in the previous section, we define normalized versions of the mean-field solution components $\rho_\e^+(t)$ and $\rho_\e^-(t)$:
\[
\mu_\e^\pm (t) := \frac{\rho_\e^\pm(t)}{\|\rho_\e^\pm(t)\|_{\mathrm {TV}}}.
\]
The following lemma now provides some preliminary properties of the processes we consider here.

\begin{lem}[Properties of $\{ \bX_\e(t) \}_{0 \leq t \leq T}$ and $\{ \obX_\e(t) \}_{0 \leq t \leq T}$] \label{lem:props-coupled-discrete} 
\begin{enumerate}
\item (Existence) \label{i:lem:props-coupled-discrete:existence}
For $\Prob$-\\
almost-every realization $(N^\bh)_{\bh\in\Nhd_\e^n}$ there exist unique functions $t\mapsto \bX_\e(t), (\tau^\bh(t))_\bh$ satisfying \eqref{RW:inte} and unique  functions $t\mapsto \obX_\e(t), (\otau^\bh(t))_\bh$ satisfying \eqref{oRW:inte}.
\item (Solutions) \label{l:props-coupled-discrete:RW} For each $t\geq0$ and each $\bh\in \Nhd_\e^n$, $\bX_\e(t)$, $\obX_\e(t)$, $\tau^\bh(t)$, and $\otau^\bh(t)$ are $\Sigma$-measurable; $\bX_\e$ is a solution of the interacting jump problem~$\RW$, and $\obX_\e$ is a list $\{ \oX_{\e,i} \}_{i=1}^n$ of independent processes, whose law for any $t \in [0,T]$ and any $i \in I^\pm$ is given by $\mu_\e^\pm(t)$.
\item (Expectation identity) \label{l:props-coupled-discrete:Ex} For each $\bh\in \Nhd_\e^n$, the functions $N^\bh\circ \tau^\bh$ and $N^\bh\circ \otau^\bh$ are $\Sigma$-measurable, and we have for each $t\geq0$
\begin{equation}
\label{eq:Y-tau}
\E|N^\bh(\tau^\bh(t))-N^\bh(\otau^\bh(t))| = \E|\tau^\bh(t)-\otau^\bh(t)|.
\end{equation}
\end{enumerate}
\end{lem} 

\noindent
The proof of this result is postponed to the following section. Now we define the empirical measure
on $\Lpm$ to be
\begin{equation*}
  \orho_{\e,n}:=\frac{1}{n}\sum_{i=1}^n\delta_{(\overline{X}_{\e,i},b_i)},
\end{equation*}
and as in \eqref{est:PoC}, we may bound the quantity estimated in Theorem~\ref{t:RWtoMFe} via triangle inequality, writing
\begin{equation}
  \E \| \rho_{\e,n}^\pm(t) - \rho_\e^\pm(t) \|_{1,\infty}^* 
 \leq \E \| \rho_{\e,n}^\pm(t) - \orho_{\e,n}^\pm(t) \|_{1,\infty}^* 
 + \E \| \orho_{\e,n}^\pm(t) - \rho_\e^\pm(t) \|_{1,\infty}^*.\label{est:PoCe}
\end{equation}
Each of the terms on the right hand side can now be bounded in analogue with the arguments made in Section~\ref{s:SDEn:FP-Mark}. In the case of the first term, an upper bound can be obtained as a direct corollary of the following propagation of chaos result which forms the equivalent of Lemma~\ref{l:PoC} for the random walk; the proof of this result is postponed to the next section.

\begin{lem}[Propagation of chaos on the lattice]
\label{l:PoCe}
There exists universal constants $C$ and $C'$, such that for all $t \in [0,T]$,
\begin{equation*}
  \frac1n \E \sum_{i=1}^n \big| X_{\e,i} -\oX_{\e,i} \big|(t)
  \leq  \frac{C' t}{\delta } \Big( \frac1{\sqrt n} + \kappa \Big) \mathrm{e}^{ C \delta^{-2} t }
\end{equation*} 
where $\kappa$ is the mass discrepancy given in~\eqref{def:gamma-mass-disc-discrete}.
\end{lem}

\noindent
The second term in \eqref{est:PoCe} can be bounded using the estimate established in Lemma~\ref{l:GlivenkoCantelli}, translated to the random walk setting; since the remainder of the proof of Theorem~\ref{t:RWtoMFe} is therefore completely analogous to the arguments given to prove Theorem~\ref{t:SDEtoMF} in Section~\ref{s:SDEn:FP-Mark}, we omit the details.

\subsection{Proofs of auxiliary results}
This section provides detailed proofs of the two important auxiliary results used above.

\subsubsection{Proof of Lemma~\ref{lem:props-coupled-discrete}}
Lemma~\ref{lem:props-coupled-discrete} establishes existence and two other important properties of the processes considered in the proof of Theorem~\ref{t:RWtoMFe}. 
We refer to  Helms~\cite{Helms74} and Ethier and Kurtz~\cite[Ch.~6]{EthierKurtz} for further background on the concepts that we use in this proof. 

For each $\bh \in \Nhd_\e^n$, let $\filtF^\bh$ be a filtration for the Poisson process $N^\bh$ such that $\{N^\bh\}_\bh$ are independent. For vectors $\bu=(u^\bh)_{\bh\in \Nhd_\e^n} \in [0,\infty]^{4n}$ and $\bv = (v^\bh)_{\bh\in\Nhd_\e^n} \in [0,\infty]^{4n}$  we define the inequality $\bu \leq \bv$ coordinate-wise, i.e. $\bu \leq \bv \Longleftrightarrow \bigl[ u^\bh \leq v^\bh \text{ for all }\bh\bigr]$. 
For given $\bu$, we define the multiparameter filtration 
\[
\filtF_\bu := \sigma\Bigl( \bigcup_{ \bh \in \Nhd_\e^n } \filtF_{u^\bh}^\bh \Bigr).
\]
This filtration satisfies $\filtF_\bu \subset \filtF_\bv \subset \Sigma$ if $\bu\leq \bv$. 
An $\filtF_\bu$ stopping time  $\bT=(T^\bh)_{\bh\in \Nhd_\e^n}$ is defined to be a $[0,\infty]^{4n}$-valued random variable such that for each $\bu\in [0,\infty]^{4n}$ the set $\{\bT \leq \bu \}$ is an element of $\filtF_\bu$. 

The existence and uniqueness of solutions of~\eqref{RW:inte} for almost all realizations $(N^\bh)_\bh$, assertion~\ref{i:lem:props-coupled-discrete:existence} of the Lemma, is shown by Helms in~\cite[Sec.~4]{Helms74}, where in order to fit \eqref{RW:inte} and \eqref{oRW:inte} to \cite[(4.2)]{Helms74}, one needs to extend the collection $\{N^\bh\}_\bh$ of Markov processes by the constant-in-time  process $\bX^\circ_\varepsilon$ and the deterministic Markov process $t \mapsto t$. He also shows that $\bX_\e(t)$ and $\obX_\e(t)$ are $\Sigma$-measurable, and that for each $t\geq0$ the speed functions $\btau(t)$ and $\obtau(t)$ are $\filtF_\bu$ stopping times (and therefore $\Sigma$ measurable). Helms also shows that $\{N^\bh \circ \tau^\bh\}_\bh$ is a Markov process, and notes at the start of Section 5 that it is even a Feller process due to its finite state space. This property will allow us to apply~\cite[Thm.~10]{Helms74} below.

By~\cite[Thm.~10]{Helms74}, the generator of the process $\bX_\e$ is given by 
\[
\Omega_\e f := \e\sum_{\bh\in\Nhd_\e^n} \cR^\e_{n,\bh}(\cdot, \bb) D_{\bh} f
\]
which is consistent with the adjoint $\Omega_\e^*$ defined in~\eqref{FPe}. Since the generator is a bounded linear operator on the finite-dimensional state space $\L^n$, the generator uniquely characterizes the process, which in turn proves assertion \ref{l:props-coupled-discrete:RW} of Lemma~\ref{lem:props-coupled-discrete} concerning $\bX_\e$. 

Similarly, the autonomous process $Z(t) := (\obX_\e(t),t)$ has, setting \\
$\cR_i (\bell, s) := \cR_h^\e (\ell_i, b_i; \rho_\e(s))$, time--change representation 
\begin{align*}
\otau^\bh(t) &= \int_0^t \cR_i(Z(s))\, ds&\qquad &\bh\in \Nhd_\e^n, \ t\geq 0,\\
Z(t) &= Z(0) + \sum_{\bh\in\Nhd_\e^n} (\bh,0) N^\bh(\otau^\bh(t)) + (0,t)&&t\geq 0.
\end{align*}
Again applying \cite[Thm.~10]{Helms74}, we obtain that its generator is given by
\begin{equation*} 
\overline \Omega_\e f(\bell,t) := \e\sum_{\bh\in\Nhd_\e^n}\cR_h^\e(\ell_i, b_i; \rho_\e(t)) D_{\bh} f(\bell,t) 
 + \partial_t f(\bell,t).
\end{equation*}
Note that $\overline \Omega_\e$ is the generator of $i = 1,\ldots,n$ independent random walks on $\L$, each with time--dependent rate $\cR_h^\e(\ell, b_i; \rho_\e(t))$ to jump to the neighbouring lattice site $\ell + h$. Since the rate only depends on $i$ through $b_i$, the rate is the same for all $i \in I^+$ or for all $i \in I^-$. Hence, given $i \in I^\pm$, the law $\frac n{n^\pm} \orho^\pm$ of $\oX_{\e,i}(t)$ satisfies the Fokker--Planck equation
\begin{equation} \label{oMFe} 
\left\{ \begin{aligned}
\ds
\partial_t \orho^\pm_\e 
  &= \e \sum_{h\in \Nhd_\e} D_{-h}\Big(\mathcal{R}^\e_h (\,\cdot\,,\pm 1;\rho_\e)\orho_\e^\pm\Big)
  &&\quad \text{on }  \L\times (0, T),  
  \\
  \orho_\e(0) 
  &= \rho_\e^\circ
  &&\quad \text{on }  \L. 
\end{aligned} \right.
\end{equation}
Since \eqref{oMFe} is a system of linear ODEs with bounded, regular right-hand side, it has a unique solution $\orho_\e$. It is then clear from $\MFe$ that this solution is given by $\rho_\e$, i.e., $\orho_\e = \rho_\e$, and this completes the proof of assertion \ref{l:props-coupled-discrete:RW} of Lemma \ref{lem:props-coupled-discrete}. 

We finally prove the identity~\eqref{eq:Y-tau} by applying Doob's Optimal Stopping Theorem to the martingale $M_t^\bh := N^\bh(t) - t$ and the stopping times $\tau^\bh$ and $\otau^\bh$. To do so, we must show that $M^\bh$, $\tau^\bh$ and $\otau^\bh$ can be adapted to the same filtration. By construction, $M^\bh$ is an $\filtF^\bh$-martingale and $\tau^\bh$ and $\otau^\bh$ are stopping times adapted to
\begin{equation*}
  \filtM^\bh_t
  :=  \sigma\Bigl( \filtF_t^\bh \cup \bigcup_{ \bh' \neq \bh } \filtF_\infty^{\bh'} \Bigr)
  \supset \filtF^\bh_t.
\end{equation*}
Since the events $\{\filtF_\infty^{\bh}\}_\bh$ are independent, $M^\bh$ is also an $\filtM^\bh$-martingale. Then, applying Doob's Optimal Stopping Theorem to $M^\bh$ and the stopping time $\tau^\bh (t) \vee \otau^\bh (t)$, we obtain for any $t \in [0,T]$ and any $\bh \in \Nhd_\e^n$ that
\begin{align*}
0  
= \E[M_0^\bh]
= \E \big[M_{\tau^\bh (t) \vee \otau^\bh (t)}^\bh\big]
= \E \bigl[N^\bh(\tau^\bh(t)\vee\otau^\bh(t))\bigr] - \E \bigl[\tau^\bh(t)\vee\otau^\bh(t)\bigr].
\end{align*}
Analogously, a similar expression for $\tau^\bh\wedge \otau^\bh$ follows. It follows that
\begin{align*}
\E\big|N^\bh(\tau^\bh(t))-N^\bh(\otau^\bh(t))\big|
&= \E\Big[ N^\bh(\tau^\bh(t)\vee\otau^\bh(t)) - N^\bh(\tau^\bh(t)\wedge\otau^\bh(t))\Big]\\
&= \E \Big[\tau^\bh(t)\vee\otau^\bh(t) - \tau^\bh(t)\wedge\otau^\bh(t)\Big]\\
&= \E\big|\tau^\bh(t)-\otau^\bh(t)\big|,
\end{align*}
which proves assertion~\ref{l:props-coupled-discrete:Ex} of Lemma~\ref{lem:props-coupled-discrete}, and therefore concludes the proof.

\subsubsection{Proof of Lemma~\ref{l:PoCe}}
Lemma~\ref{l:PoCe} provides the analogue of Lemma~\ref{l:PoC} in the random walk setting. 
Using the formulae provided by \eqref{RW:inte}, \eqref{oRW:inte} and \eqref{eq:Y-tau}, we estimate
\begin{align}
&\E \sum_{i=1}^n \big| X_{\e,i} -\oX_{\e,i} \big|(t) 
\notag\\
&\quad= \E \sum_{i=1}^n \Big| \sum_{h\in \Nhd_\e} h \big( N^\bh(\tau^\bh(t))-N^\bh(\otau^\bh(t)) \big) \Big|
\notag\\
&\quad\leq \sum_{\bh\in \Nhd_\e^n} \e \E\bigl|N^\bh(\tau^\bh(t))-N^\bh(\otau^\bh(t))\bigr|
\notag\\
&\quad= \sum_{\bh\in \Nhd_\e^n} \e \E \bigl|\tau^\bh(t)-\otau^\bh(t)\bigr| \notag\\
&\quad\leq \int_0^t \E \Big[ \sum_{\bh\in\Nhd_\e^n} \e \left| \cR^\e_{n,\bh}(\bX_\e(s), \bb) - \cR^\e_{n,\bh}(\obX_\e(s), \bb) \right| \Big] \,ds  \notag\\
&\quad\qquad + \int_0^t \E \Big[ \sum_{\bh\in\Nhd_\e^n} \e \left| \cR^\e_{n,\bh}(\obX_\e(s), \bb) - \cR_h^\e(\oX_{\e,i}(s), b_i; \rho_\e(s)) \right| \Big] \,ds,
\label{ineq:MF-discrete-EZ}
\end{align}
where we have used the characterisation of $\bh$ in terms of $i$ and $h$ to write $\sum_{i=1}^n \sum_{h\in \Nhd_\e}$ as $\sum_{\bh\in \Nhd_\e^n}$.
We now treat the two integrals on the right--hand side of the above upper bound separately. For the first integral, we consider the integrand only for a fixed time, thereby allowing us to omit $s$ from our notation. Using the elementary inequality $|\mathrm{e}^a - \mathrm{e}^b| \leq \mathrm{e}^{a \vee b} |a-b|$, the estimate
\begin{align*}
\big| \bh \cdot \bF (\bell, \bb) \big| 
&\leq \frac1n \sum_{j=1}^n \left| h \cdot \nabla V_\delta (\ell_i - \ell_j) \right|
\leq \e \|\nabla V_\delta\|_\infty 
\leq \frac{\e \mathsf C_V }\delta
\quad \text{for all } \bh \in \Nhd_\e^n, \, \bell \in \L^n,
\end{align*}
and \eqref{est:PoC:pf0}, we first estimate
\begin{align} \notag
&\e \left| \cR^\e_{n,\bh}(\bX_\e, \bb)-\cR^\e_{n,\bh}(\obX_\e, \bb) \right|
= \frac1{\e\beta} \Big| \exp\Bigl(\frac\beta2 \bh \cdot \bF (\bX_\e, \bb) \Bigr) - \exp\Bigl(\frac\beta2 \bh \cdot \bF (\obX_\e, \bb) \Bigr) \Big| \\\notag
&\leq \frac1{\e\beta} \exp\Bigl(\frac\beta2 \max_{\bell \in \L^n} \big| \bh \cdot \bF (\bell, \bb) \big| \Bigr) 
\bigg| \frac\beta2 \frac hn \cdot \sum_{j=1}^n b_j \big( \nabla V_\delta (X_{\e,i} - X_{\e,j} ) - \nabla V_\delta (\oX_{\e,i} - \oX_{\e,j} ) \big) \bigg|\\\label{est:oRW:11}
&\leq \frac{\mathsf C_V}{2\delta^2} \mathrm{e}^{ \mathsf C_V \e\beta /2\delta } \Big( |X_{\e,i} - \oX_{\e,i}| + \frac1n \sum_{j=1}^n|X_{\e,j} - \oX_{\e,j} | \Big).
\end{align}
By assumption~\ref{ass:beta-e-delta-bound}, $\e\beta/\delta$ is bounded from above by a constant. Hence, \eqref{est:oRW:11} may be written
\begin{align*} 
\e \left| \cR^\e_{n,\bh}(\bX_\e, \bb)-\cR^\e_{n,\bh}(\obX_\e, \bb) \right|
\lesssim \frac{1}{\delta^2} \Big( |X_{\e,i} - \oX_{\e,i}| + \frac1n \sum_{j=1}^n|X_{\e,j} - \oX_{\e,j} | \Big).
\end{align*}
This estimate entails that the first integral on the right-hand side in \eqref{ineq:MF-discrete-EZ} can be bounded above by
\begin{align}
  &\int_0^t \E \Big( \sum_{\bh\in\Nhd_\e^n} \e \left| \cR^\e_{n,\bh}(\bX_\e(s), \bb) - \cR^\e_{n,\bh}(\obX_\e(s), \bb) \right| \Big) \,ds
  \lesssim \frac{1}{\delta^2} \int_0^t \E \sum_{i=1}^n \big| X_{\e,i} -\oX_{\e,i} \big|(s) \, ds. 
\end{align} \label{est:oRW:2}

For the second integral in \eqref{ineq:MF-discrete-EZ}, we follow a similar procedure to that in Section~\ref{s:SDEn:FP-Mark}. Since
\begin{equation*}
\big| h \cdot F (\ell, b; \rho_\varepsilon) \big| 
\leq \int_{\T^2} \left| h \cdot \nabla V_\delta (\ell - y) \right| \, (\rho_\varepsilon^+ + \rho_\varepsilon^-)(dy)
\leq \e \|\nabla V_\delta\|_\infty 
\leq \frac{\mathsf C_V \e}\delta
\end{equation*}
for all $h \in \Nhd_\e$, $\ell \in \L$, arguing as in \eqref{est:oRW:11}, we obtain that
\begin{align*}
&\e \left| \cR^\e_{n,\bh}(\obX_\e, \bb) - \cR_h^\e(\oX_{\e,i}, b_i; \rho_\e) \right|\\
&\quad= \frac1{\e\beta} \Big| \exp\Bigl(\frac\beta2 \bh \cdot \bF (\obX_\e, \bb) \Bigr) - \exp\Bigl(\frac\beta2 h \cdot F (\oX_{\e,i}, b_i; \rho_\varepsilon) \Bigr) \Big| \\\notag
&\quad\leq \frac12 \mathrm{e}^{ \mathsf C_V\e\beta /2\delta }
\bigg| \big( \nabla V_\delta * (\rho_\varepsilon^+ - \rho_\varepsilon^-) \big)(\oX_{\e,i}) - \frac1n \sum_{j=1}^n b_j \nabla V_\delta (\oX_{\e,i} - \oX_{\e,j} ) \bigg|.
\end{align*}
The factor on the right--hand side inside the modulus sign is similar to the first term on the right-hand side in \eqref{est:PoC:pf1}, and following the same line of argument as in~\eqref{est:PoC1-gamma}, we estimate
\begin{align*}
\bigg| \big( \nabla V_\delta * (\rho_\varepsilon^+ &- \rho_\varepsilon^-) \big)(\oX_{\e,i}) - \frac1n \sum_{j=1}^n b_j \nabla V_\delta (\oX_{\e,i} - \oX_{\e,j} ) \bigg|\leq\\
&\leq {}\bigg| \big( \nabla V_\delta * (\mu_\varepsilon^+ - \mu_\varepsilon^-) \big)(\oX_{\e,i}) - \frac1n \sum_{j=1}^n b_j \nabla V_\delta (\oX_{\e,i} - \oX_{\e,j} ) \bigg| + 2\mathsf C_V \delta^{-1} \kappa.
\end{align*}
Now, using the function $g$ defined in~\eqref{def:g-poc}, we may estimate the second term on the right-hand side in~\eqref{ineq:MF-discrete-EZ} as
\begin{align*}
&\int_0^t \E \Big[ \sum_{\bh\in\Nhd_\e^n} \e \left|\cR^\e_{n,\bh}(\obX_\e(s), \bb) - \cR_h^\e(\oX_{\e,i}(s), b_i; \rho_\e(s)) \right| \Big] \,ds \\
  &\quad\lesssim \frac{n \kappa t }\delta
+ \int_0^t \E\biggl[  \sum_{\bh\in\Nhd_\e^n} 
   \Big| \big( \nabla V_\delta * (\mu_\varepsilon^+(s) - \mu_\varepsilon^-(s)) \big)(\oX_{\e,i}) 
     - \frac1n \sum_{j=1}^n b_j \nabla V_\delta (\oX_{\e,i} - \oX_{\e,j} ) \Big|\bigg] \, dt\\
&\quad= \frac{n \kappa t }\delta 
  + \int_0^t \E\biggl[  \sum_{\bh\in\Nhd_\e^n} 
   \Big| \frac1n \sum_{j=1}^n g(\oX_{\e,i}(s),\oX_{\e,j}(s),b_j) \Big|\bigg] \, dt.
\end{align*}
Following the same argument made in the proof of Lemma~\ref{l:PoC} after
\eqref{est:PoC:pf2}, we find that
\begin{equation} \label{est:oRW:3}
  \int_0^t \E 
  \Big[ \sum_{\bh\in\Nhd_\e^n} \e \left| \cR^\e_{n,\bh}(\obX_\e(s), \bb) - \cR_h^\e(\oX_{\e,i}(s), b_i; \rho_\e(s)) \right| \Big] \,ds
  \lesssim \frac{(n\kappa + \sqrt n) t}\delta.
\end{equation}

Finally, using the upper bounds \eqref{est:oRW:2} and \eqref{est:oRW:3} to estimate the right--hand side of \eqref{ineq:MF-discrete-EZ}, we find that 
\[
E \sum_{i=1}^n \big| X_{\e,i} -\oX_{\e,i} \big|(t)
\leq C' \frac{(n\kappa + \sqrt n)\, t}\delta + \frac{C}{\delta^2} \int_0^t \E \sum_{i=1}^n \big| X_{\e,i} -\oX_{\e,i} \big|(s) \, ds
\]
for some general constants $C$ and $C'$. The result of Lemma~\ref{l:PoCe} then follows by an application of Gronwall's Lemma.

\section{Proof of Corollaries \ref{c:RWtoMFvMFe} and \ref{c:RWtoMFvSDE}}
\label{sec:RWtoMF}
This section concerns itself with the proofs of Corollary~\ref{c:RWtoMFvMFe} and Corollary~\ref{c:RWtoMFvSDE}.

\subsection{Proof of Corollary \ref{c:RWtoMFvMFe}}

The proof is an application of the triangle inequality and the regularity estimates on the solution $\rho$ to $\MF$ in Lemma \ref{l:reg-MF}. We omit the time variable for convenience.
\begin{align*}
  \E \| \rho_{\e,n}^\pm - \rho^\pm \|_{1,\infty}^*
  \leq \E \| \rho_{\e,n}^\pm - \rho_\e^\pm \|_{1,\infty}^* + \| f_\e^\pm \nu_\e - f^\pm \nu \|_{1,\infty}^*.
\end{align*}
We use Theorem \ref{t:RWtoMFe} to estimate the first term by $R_2$. We split the second term as
\begin{align} \label{cpf:1}
  \big\| f_\e^\pm \nu_\e - f^\pm \nu \big\|_{1,\infty}^*
  \leq \big\| (f_\e^\pm - f^\pm) \nu_\e \big\|_{1,\infty}^* + \big\| f^\pm (\nu_\e - \nu) \big\|_{1,\infty}^*.
\end{align}
We employ H\"older's inequality twice and the fact that $\nu_\e$ is a probability measure to estimate the first term by
\begin{align*}
  \| (f_\e^\pm - f^\pm) \nu_\e \|_{1,\infty}^*
  &= \sup_{ \| \varphi \|_{1,\infty} \leq 1 } \int_{\T^2} (f_\e^\pm - f^\pm) \varphi \, d\nu_\e \\
  &\leq \sup_{ \| \varphi \|_{1,\infty} \leq 1 } \| \varphi \|_\infty \| f_\e^\pm - f^\pm \|_{L^1 (\nu_\e)}
  \leq \| f_\e^\pm - f^\pm \|_{L^2 (\nu_\e)} 
  \leq R_1,
\end{align*}
where we have applied Theorem \ref{t:MFetoMF} in the last inequality. Finally, we estimate the second term in \eqref{cpf:1} in an analogous way to the argument which leads up to \eqref{fvp:2infty}; this yields
\begin{align*}
\|f^\pm \nu_\e - \rho^\pm\|_{1,\infty}^*
\leq C \e \|f^\pm\|_{1,\infty}.
\end{align*}
Applying Lemma \ref{l:reg-MF} and using the given polynomial bound on $\|f^\circ\|_{1, \infty}$, we obtain
\begin{equation*}
  \|f^\pm \nu_\e - \rho^\pm\|_{1,\infty}^*
\lesssim \e \Big( \|f^\circ\|_{1, \infty}\mathrm{e}^{\gamma t} + \sqrt \beta \delta^{-3} \|f^\circ\|_\infty \sqrt t\mathrm{e}^{2 \gamma t} \Big)
\lesssim \e\mathrm{e}^{\beta \delta^{-2} T} \mathrm{e}^{2^5 \mathsf C_V \beta \delta^{-2} t},
\end{equation*}
which completes the proof.

\subsection{Proof of Corollaries \ref{c:RWtoMFvSDE}}

The proof of Corollary \ref{c:RWtoMFvSDE} is an application of the triangle inequality and the regularity estimates on the solution $\bmu$ to the Fokker-Planck equation \eqref{FP} of $\SDE$ in Corollary \ref{c:rho:est}. Throughout, we will omit the time variable for convenience. 

Our application of the triangle inequality involves the measure $\beff \bnu_\e$, which need not have mass 1. To extend $W_1$ to a metric $\tilde W_1$ on $\cM (\T^{2n})$, we note that in the definition of $W_1$ in \eqref{W1:def}, the value of $W_1$ does not change if we add a constant to $\bvarphi$. As such, since the maximal Euclidean distance between any two points on $\T^{2n}$ is $\sqrt{n/2}$, we find that
\begin{equation*}
  \tilde W_1 (\bmu, \bmu') := \sup_{ \bvarphi \in Y } \int_{\T^{2n}} \bvarphi \, d(\bmu - \bmu'),
  \quad Y := \{ \bvarphi \in W^{1,\infty} (\T^{2n}) : \| \bvarphi \|_\infty \leq \sqrt n, \ \| d \bvarphi \|_\infty \leq 1 \}
\end{equation*}
equals $W_1$ on $\cP (\T^{2n})$. Note that $\tilde W_1$ is a bounded Lipschitz norm with different constants for the bounds on the test functions.

Finally, we apply the triangle inequality:
\begin{align} \label{cpf:2}
  W_1( \bmu_\e, \brho)
  \leq \tilde W_1 \big( \beff_\e\bnu_\e ,\beff\bnu_\e \big) 
       + \tilde W_1 \big(  \beff \bnu_\e , \beff \bnu \big)
       + W_1 \big(  \bmu , \brho \big).
\end{align} 
The first term can be treated the same way as above, now relying on Theorem \ref{t:RWtoSDE}. Because of the $L^\infty$-bound on the test function by $\sqrt n$, we obtain the bound $\tilde W_1 \big( \beff_\e\bnu_\e , \beff\bnu_\e \big) \leq \sqrt n R_1$.

The second term in \eqref{cpf:2} can also be estimated similar as above, but with several minor changes from \eqref{fvp:2infty}. It goes as follows. Let $\{ \cC_i \}_{i=1}^{\e^{-2n}}$ be the tessellation of $\T^{2n}$ where $\cC_i$ are the $2n$-cubes of size $\e$ with midpoint $\bell_i \in \L^n$. Then
\begin{align} \notag
\tilde W_1 \big(  \beff \bnu_\e , \beff \bnu) \big)
&= \sup_{\bvarphi \in Y} \sum_{i=1}^{\e^{2n}} 
  \int_{\mathcal C_i} \big[ (\bvarphi \beff)(\bell_i) - (\bvarphi \beff)(\bx) \big]\, d \bx\\\notag
&\leq \sum_{i=1}^{\e^{2n}} |\cC_i| 
  \sup_{\bvarphi \in Y} \sup_{\by \in \cC_i } \frac{| (\bvarphi \beff)(\bell_i) - (\bvarphi \beff)(\by) |}{ \sum_{j=1}^n| (\bell_i)_j - y_j| } \, \sup_{\bx \in \cC_i } \Big( \sum_{j=1}^n | (\bell_i)_j - x_j| \Big) \\\notag
&\leq  \Big( \|\beff\|_{\infty} + \sqrt n \| d \beff \|_{\infty} \Big) \e n \\\label{fvp:2fatinfty}
&\lesssim \e n^{3/2} \| \beff^\circ \|_{1, \infty}\mathrm{e}^{48 \mathsf C_V^2 \beta \delta^{-2} n T}
\lesssim \e n\mathrm{e}^{49 \mathsf C_V^2 \beta \delta^{-2} n T},
\end{align}
where in the last steps we have used Corollary \ref{c:rho:est} and the polynomial bound on $\beff^\circ$.

Finally, we use Lemma \ref{l:PoC} to estimate the third term in \eqref{cpf:2}. This lemma gives a bound on
\begin{equation*}
  \frac1n \E \sum_{i=1}^n | X_i - \oX_i |
  = \frac1n \iint_{(\T^{2n})^2} \Big( \sum_{i=1}^n | x_i - \ox_i | \Big) \, d \mathbb P(\bx, \obx),
\end{equation*} 
where $\mathbb P$ is the joint probability distribution of the processes $\bX$ and $\obX$ constructed in Section \ref{s:SDEn:FP-Mark}. Alternatively, we interpret $\mathbb P$ as a coupling between $\bmu$ and $\mu^{\otimes n}$, i.e.,
\begin{equation*}
  \mathbb P 
  \in \Gamma(\bmu, \brho)
  := \left\{ \bgamma \in \cP( (\T^{2n})^2 ) : \forall \, A \subset \T^{2n} : \Big\{ \begin{array}{ll}
    \bgamma(A, \T^{2n}) 
    &= \bmu (A) \\
    \bgamma(\T^{2n}, A)  
    &= \brho (A)
  \end{array} \right\}.
\end{equation*}
The connection with the third term in \eqref{cpf:2} is as follows; using the  Kantorovich duality (see e.g. \cite[(7.1.2)]{AmbrosioGigliSavare08}), we obtain
\begin{multline*}
  W_1 (\bmu, \brho )
  = \inf_{ \bgamma \in \Gamma(\bmu, \mu^{\otimes n}) } \iint_{(\T^{2n})^2} \Big( \sum_{i=1}^n | x_i - \ox_i | \Big) \, d \bgamma (\bx, \obx) \\
  \leq \iint_{(\T^{2n})^2} \Big( \sum_{i=1}^n | x_i - \ox_i | \Big) \, d \mathbb P(\bx, \obx)
  = \E \sum_{i=1}^n |X_i - \overline X_i|,
\end{multline*}
to which Lemma \ref{l:PoC} applies directly.  

\appendix
\section{Norms and function spaces}
\label{app:function-spaces}


\paragraph{Vectors and tensors.}
For the definition of the norms of vectors, matrices, and higher-order tensors, we interpret vectors as linear maps from $\R^{2n}$ to $\R$, matrices as bilinear maps from $\R^{2n}\times \R^{2n}$ to~$\R$, and general $k$-tensors as multilinear maps from $(\R^{2n})^k$ to $\R$. We write $U_k$ for the space of $k$-tensors on $\R^{2n}$ (which is viewed as the tangent space to $\T^{2n}$); $U_0 = \R$ is the space of scalars, $U_1$ the space of vectors, $U_2$ the space of matrices, etc. The norm of a $k$--tensor $\mathsf K\in U_k$ is defined by duality as
\[
|\mathsf K| := |\mathsf K|_{U_k} := \sup\Big\{ \mathsf K[\by_1,\dots,\by_k] : \by_i = (y_{i,1},\ldots,y_{i,n})\in(\R^2)^n,\sum_{j=1}^n |y_{i,j}|_{\R^2} \leq 1, \forall i \Big\}.
\]
When $n=1$, this reduces to the Euclidean norm on $\R^2$, and for the space of matrices $\R^{2 \times 2}$ to the spectral norm, which also is the operator norm as operator on $\R^2$ endowed with the Euclidean norm. For $n>1$, this norm acts as the maximum of such norms over sub-tensors corresponding to fixed indices, e.g., for $\mathsf K=(\mathsf K_i)_{i=1}^n \in U_1$ with $\mathsf K_i \in \R^2$, and $\mathsf L = (\mathsf L_{ij})_{i,j=1}^n\in U_2$ with $\mathsf L_{ij}\in\R^{2\times 2}$, we have
\begin{equation*}
  |\mathsf K| := \max_{i=1,\ldots,n} |\mathsf K_i|_{\R^2},\quad\text{and}\quad |\mathsf L| := \max_{i,j=1,\ldots,n} |\mathsf L_{ij}|_{\R^{2\times 2}}.
\end{equation*}

\paragraph{The norms $\|\cdot\|_{k,\infty}$.} The norm we will generally use for functions taking values in spaces of tensors $g:\T^{2n}\to U_k$ is the supremum norm, which we define to be 
\begin{equation*} 
  \|g\|_\infty  := \sup_{\bx\in \T^{2n}} \big|g(\bx)\big|_{U_k}.  
\end{equation*}
Abstractly, differentiation is viewed as a map from $U_k$--valued functions to $U_{k+1}$--valued functions, indicated with the letter $d$:
\[
dg(\bx)[\by_1,\dots,\by_k,\by_{k+1}] := 
  \lim_{h\to0} \frac{g(\bx+h\by_{k+1})[\by_1,\dots,\by_k]-g(\bx)[\by_1,\dots,\by_k]}h.
\]
We use this to define for a  function $g:\T^{2n} \to U_k$ and for an integer $m\geq0$, 
\[
\big\|d^m g\big\|_\infty := \sup_{\bx\in \T^{2}} \big|d^m g(\bx)\big|_{U_{k+m}} .
\]
With this notation, the chain rule estimate applies with constant one; for instance, for any $g \in C^1(\T^{2n}; U_k)$ we will often use the inequality
\[
|g(\bx)-g(\by)|
\leq \|dg\|_\infty \sum_{i=1}^n d_{\T^2}(x_i , y_i)
\qquad\text{for any sufficiently smooth $U_k$-valued $g$}.
\]
For an integer $k\geq 0$ and for a function $g:\T^{2n}\to\R^d$, set 
\[
\|g\|_{k,\infty} := \sum_{\ell=0}^k \|d^\ell g\|_\infty.
\]
These norms can be used to define the Banach spaces of $k$--times weakly--differentiable functions defined on $\T^{2n}$, whose $k$th--weak derivative is essentially bounded. These spaces are denoted $W^{k,\infty}$; setting $W^{0,\infty} = L^\infty$, we say that $\varphi \in L^\infty$ is in $W^{k,\infty}$ for $k \geq 1$ if $\varphi$ is $k$--times weakly differentiable and
  \begin{equation*}
    \|\varphi\|_{k,\infty}
    := \sum_{m=0}^{k}\|d^m\varphi\|_{\infty}
    < +\infty.
  \end{equation*}
  We note that the spaces $W^{k,\infty}$ may be identified with the H\"older spaces $C^{k-1,1}$ through an application of Rademacher's Theorem.

In analogy with the case above where the domain of functions considered is $\T^{2n}$, we will use various natural generalizations of the norms above to other domains; for instance, for $g:\L^n\to U_k$, 
\[
\|g\|_\infty := \sup_{\bl\in \L^n} |g(\bl)|_{U_k},
\]
and for $g: \T^2_\pm \to U_k$, 
\[
\|dg\|_\infty := \sup_{(x,b)\in \T^2_\pm} |dg(x,b)|_{U_{k+1}}.
\]
For first derivatives of a function $g$ we will often use the traditional notation $\nabla g$, for which by the construction above we have 
\[
\|\nabla g\|_\infty = \|d g\|_\infty.
\]

\paragraph{The dual norms $\|\cdot\|_{k,\infty}^*$.} As noted above $W^{k,\infty}$ is isomorphic to $C^{k-1,1}$. Since $C^{k-1,1}$ is contained in $C^{0,1}$, and the dual space $(C^{0,1})^*$ may be identified with the space of measures $\mathcal{M}_+$, it follows that the norms $\|\cdot\|_{k,\infty}$ naturally induce a topology on $\mathcal{M}_+$ which is dual to that on $W^{k,\infty}$. The resulting metric will be essential for describing convergence in the measure theoretic framework that we use.

For $k \geq 1$ and $\mu\in \mathcal M_+ \subset (W^{k,\infty})^*$, we define the  dual norm in the usual way,
\[
\|\mu\|_{k,\infty}^* 
:= \sup_{\|\varphi\|_{k,\infty} \leq 1} \int_{\T^2} \varphi(x) \mu(dx).
\]
We remark that $\| \cdot \|_{1,\infty}^*$ is often referred to as the dual bounded Lipschitz norm, and the resulting dual space norm metrizes the narrow topology in the space of finite non--negative measures, which is alternatively characterized by convergence against continuous and bounded functions~\cite[Th.~8.3.2]{Bogachev07.II}.

\paragraph{Other $L^p$ spaces.}
The usual Lebesgue spaces with respect to a given measure $\mu\in\mathcal{M}_+(A)$ are
denoted by $L^p(\mu)$, which is a Banach space with norm 
\begin{equation*}
  \| f \|_{L^p(\mu)}^p := \int_A |f|^p \,\dd \mu \qquad \text{for } p = 1,2.
\end{equation*}

\paragraph{Acknowledgments}
TH gratefully acknowledge support from an Early Career Fellowship awarded by the Leverhulme Trust (ECF-2016-526).

PvM gratefully acknowledges support from the International Research Fellowship of the Japanese Society for the Promotion of Science and the associated JSPS KAKENHI grant 15F15019.

MAP gratefully acknowledges support from NWO project 613.001.552, Large Deviations and Gradient Flows: Beyond Equilibrium.

\newcommand{\etalchar}[1]{$^{#1}$}

\end{document}